\documentclass[11pt]{article}
\usepackage[utf8]{inputenc}
\usepackage[margin=1in]{geometry}
\usepackage{amssymb}
\usepackage{parskip}
\usepackage{amsmath}
\usepackage{amsthm}
\usepackage{mathrsfs}

\usepackage{amsfonts}
\usepackage{graphicx}
\usepackage{enumitem}
\usepackage{amsthm}
\usepackage{booktabs}
\usepackage{multirow}
\usepackage{nicematrix}
\usepackage{comment}
\usepackage{tikz}
\usetikzlibrary{decorations.pathreplacing,decorations.markings,snakes,patterns}

\usepackage{hyperref}
\hypersetup{
	colorlinks=true,
	linkcolor=blue,
	citecolor=blue,
	filecolor=magenta,      
	urlcolor=cyan,
	linktoc=page
}

\newcommand{\A}{\mathcal{A}}

\renewcommand{\S}{\mathcal{S}}
\renewcommand{\L}{\mathcal{L}}
\newcommand{\R}{\mathbb{R}}
\newcommand{\Q}{\mathbb{Q}}

\newcommand{\N}{\mathbb{N}}
\newcommand{\B}{\mathcal{B}}

\newcommand{\x}{\mathbf{x}}
\newcommand{\z}{\mathbf{z}}
\newcommand{\y}{\mathbf{y}}

\newcommand{\Ov}{\textup{Ov}}

\newcommand{\ep}{\epsilon}
\newcommand{\de}{\delta}

\theoremstyle{plain}
\newtheorem{theorem}{Theorem}[section]
\newtheorem{corollary}[theorem]{Corollary}

\newtheorem{prop}[theorem]{Proposition}
\newtheorem{lemma}[theorem]{Lemma}

\theoremstyle{definition}
\newtheorem{remark}[theorem]{Remark}
\newtheorem{definition}[theorem]{Definition}

\theoremstyle{plain}

\title{Geodesic networks and the disjointness gap in the directed landscape}

\author{Duncan Dauvergne \and Oliver Scott Pankratz}

\begin{document}
\maketitle

\begin{abstract}
The directed landscape is a random directed metric on the plane that arises as the
scaling limit of metric models in the KPZ universality class. 
For a pair of points $p, q$, the disjointness gap $\mathcal G(p; q)$ measures the shortfall when we optimize length over pairs of \textit{disjoint} paths from $p$ to $q$ versus optimizing over all pairs of paths. Any spatial marginal of $\mathcal G$ is simply the gap between the top two lines in an Airy line ensemble. In this paper, we show that when the start and end time are fixed, the disjointness gap fully encodes the set of exceptional geodesic networks. The correspondence uses simple features of the disjointness gap, e.g.\ zeroes, local minima. We give a similar correspondence relating semi-infinite geodesic networks to a Busemann gap function. The proofs are deterministic given a list of soft properties related to the coalescent geometry of the directed landscape.
\end{abstract}

\setcounter{tocdepth}{1}
\tableofcontents

\newpage 
\section{Introduction} \label{sec: intro}

The Kardar, Parisi, and Zhang (KPZ) universality class is a collection of one-dimensional random growth processes and two-dimensional random metrics that are expected to exhibit the same universal behaviour under rescaling. Physical predictions for scaling exponents in the KPZ class go back to the 1980s, see \cite{kardar1985commensurate, huse1985huse, kardar1986dynamic}. Mathematically, the seminal work of Baik, Deift, and Johansson \cite{baik1999distribution} rigorously 
verified the fluctuation exponent in one random metric model, the longest increasing subsequence in a uniformly chosen permutation. This work revealed that certain KPZ models have an underlying exactly solvable structure, which opened the door for countless breakthroughs over the next twenty-five years, e.g.\, see \cite{ johansson2000shape, borodin2000asymptotics, prahofer2002scale, balazs2010order, amir2011probability,corwin2014tropical, matetski2016kpz}. See \cite{romik2015surprising, corwin2012kardar, quastel2015one, zygouras2022some, ganguly2022random} and references therein for more background and an introduction to the area.

The directed landscape is the richest scaling limit in the KPZ universality class. It was first constructed in \cite{Dauvergne_2022} as the scaling limit of Brownian last passage percolation, and is now known to be the scaling limit of a handful of models with at least some exact solvability, see \cite{dauvergne2022scalinglimitlongestincreasing,aggarwal2024scalinglimitcoloredasep, wu2025kpzequationdirectedlandscape, dauvergne2024characterization}.

 The directed landscape, $\L$, is a continuous function from $\R^4_\uparrow:=\{(x,s;y,t)\in \R^4: s<t\}$ to $\R$. Interpreting $x,y$ as spatial coordinates, and $s,t$ as time coordinates, we think of $\L(x,s;y,t)$ as assigning a distance between two points $(x, s)$ and $(y, t)$ in the space-time plane. The directed landscape $\L$ is not a metric, as it can take on negative values and is not symmetric. However, it does satisfy a (reverse) triangle inequality,
$$
    \L(x,s;y,t) \geq \L(x,s;z,r) + \L(z,r;y,t) \text{\hspace{5mm} for all $x,y,z\in \R, s<r<t$},
$$
which allows us to view it as defining distances in the plane. In particular, we can define path lengths in $\L$. For a continuous function $\pi:[s,t]\to \R$, henceforth a \textbf{path}, define its length
\begin{align*}
    \|\pi\|_\L := \inf_{k\in \N} \inf_{s=t_0 < t_1 <... < t_k = t}\sum_{i=1}^k \L(\pi(t_{i-1}), t_{i-1}; \pi(t_i),t_i).
\end{align*}
We say $\pi$ is a geodesic from $p = (\pi(s),s)$ to $q = (\pi(t), t)$, or more succinctly a $(p; q)$-geodesic, if $\|\pi\|_\L = \L(\pi(s),s; \pi(t),t)$.

One striking feature of the directed landscape is that most paths are uncompetitive. For example, if we consider any deterministic path $\pi:[s, t] \to \R$, then almost surely $\|\pi\|_\L = -\infty$. Moreover, because there are so few competitive paths, geodesics will typically coalesce. We refer the reader to \cite{basu2019coalescence, pimentel2016duality, hammond2020exponents, hammond2022brownian} for on-scale coalescence results in prelimiting models, and \cite{bates2022hausdorff, basu2019fractal, bhatia2024dualitydirectedlandscapeapplications, dauvergne202327geodesicnetworksdirected} for results in the directed landscape itself. Finally, unlike smooth geometry, small perturbations of geodesics are typically uncompetitive unless the perturbation overlaps exactly with the original geodesic on long segments. This geodesic coalescence and the sparse geometry of the directed landscape lead to interesting qualitative properties that make the directed landscape quite distinct from smooth deterministic geometry. 

First, while a typical point $(p; q) \in \R^4_\uparrow$ has a unique geodesic, there are exceptional points connected by more exotic geodesic networks. The network type of a point $(p; q) \in \R^4_\uparrow$ is the directed graph induced by the collection of all $(p; q)$-geodesics, see Section \ref{S: geo and optimizers} for a precise definition. In \cite{dauvergne202327geodesicnetworksdirected}, it is shown that up to graph isomorphism and time reversal, there are exact 27 network types that appear in the directed landscape.

Next, for $(p; q) = (x, s; y, t)\in \R^4_\uparrow$ define the \textbf{disjointness gap}
\begin{equation}
    \nonumber
    \mathcal G(p; q) = 2 \L(x, s; y, t) - \sup_{\pi_1, \pi_2:p \to q} (\|\pi_1\|_\L + \|\pi_2\|_\L).
\end{equation}
Here the supremum is over all paths $\pi_1, \pi_2$ from $p$ to $q$ which are \textit{disjoint} on their interior $(s, t)$. Observing that $2 \L(x, s; y, t)$ is equal to the supremum above without the disjointness constraint, we see that
$\mathcal G(p; q)$ measures the cost of forcing two paths from $p$ to $q$ to be disjoint. Unlike in smooth geometry where the disjointness gap would be trivially $0$, the disjointness gap in $\L$ carries behind it a rich structure. For example, for fixed $x, s, t$, the functions $y \mapsto \mathcal G(x, s; y, t), y \mapsto \mathcal G(y, s; x, t)$ are equal in law to a rescaling of the gap between the top two lines of an Airy line ensemble. In particular, these functions are both locally Brownian, and have determinantal formulas. The disjointness gap $\mathcal G$ is a marginal of the extended landscape \cite{dauvergne2022disjointoptimizersdirectedlandscape}, see Section \ref{SS:extended} for more background.

\begin{figure}   
    \centering
\begin{tikzpicture}[scale=1.5]

     \node at (0,-0.2)  {I} ;
     \node at (1,-0.2)  {IIa} ;
      \node at (2,-0.2)  {IIb};
      \node at (3,-0.2)  {III} ;
       \node at (4,-0.2)  {IV} ;
      \node at (5,-0.2)  {Va} ;
       \node at (6,-0.2)  {Vb} ;

  \draw[line width=0.5mm, decorate, decoration={snake, amplitude=0mm, segment length=26mm}] (0,0) -- (0,3);

  \draw[line width=0.5mm, decorate, decoration={snake, amplitude=0mm, segment length=20mm}] (1,0) -- (1,1.5);

  \draw[line width=0.5mm] plot [smooth, tension=1.5] coordinates {(1,1.5)  (0.6,2.25) (1,3)};
  \draw[line width=0.5mm] plot [smooth, tension=1.5] coordinates {(1,1.5) (1.4,2.25) (1,3)};

  \draw[line width=0.5mm, decorate, decoration={snake, amplitude=0mm, segment length=20mm}] (2,1.5) -- (2,3);
  \draw[line width=0.5mm] plot [smooth, tension=1.5] coordinates {(2,0) (1.6,0.75) (2,1.5)};
  \draw[line width=0.5mm] plot [smooth, tension=1.5] coordinates {(2,0) (2.4,0.75) (2,1.5)};

  \draw[line width=0.5mm] plot [smooth, tension=1] coordinates {(3,0) (2.7,0.5) (3,1)};
  \draw[line width=0.5mm] plot [smooth, tension=1] coordinates {(3,0) (3.3,0.5) (3,1)};
  \draw[line width=0.5mm] plot [smooth, tension=1] coordinates {(3,2) (2.7,2.5) (3,3)};
  \draw[line width=0.5mm] plot [smooth, tension=1] coordinates {(3,2) (3.3,2.5) (3,3)};
  \draw[line width=0.5mm] plot [smooth, tension=0.5] coordinates {(3,1)  (3,2)};

  \draw[line width=0.5mm] plot [smooth, tension=1.3] coordinates {(4,0) (3.6,1.5) (4,3)};
  \draw[line width=0.5mm] plot [smooth, tension=1.3] coordinates {(4,0) (4.4,1.5) (4,3)};

  \draw[line width=0.5mm] plot [smooth, tension=1.3] coordinates {(5,0) (4.6,1.5) (5,3)};
  \draw[line width=0.5mm] plot [smooth, tension=1.3] coordinates {(5,0) (5.4,1.5) (5,3)};
  \draw[line width=0.5mm] plot [smooth, tension=1.3] coordinates {(4.63,1) (5.37,2)};

  \draw[line width=0.5mm] plot [smooth, tension=1.3] coordinates {(6,0) (5.6,1.5) (6,3)};
  \draw[line width=0.5mm] plot [smooth, tension=1.3] coordinates {(6,0) (6.4,1.5) (6,3)};
  \draw[line width=0.5mm] plot [smooth, tension=1.3] coordinates {(5.63,2) (6.37,1)};

    \filldraw[black] (0,0) circle (1pt);
    \filldraw[black] (0,3) circle (1pt);
    \filldraw[black] (1,0) circle (1pt);
    \filldraw[black] (1,3) circle (1pt);
    \filldraw[black] (1,1.5) circle (1pt);
    \filldraw[black] (2,0) circle (1pt);
    \filldraw[black] (2,1.5) circle (1pt);
    \filldraw[black] (2,3) circle (1pt);
    \filldraw[black] (3,0) circle (1pt);
    \filldraw[black] (3,1) circle (1pt);
    \filldraw[black] (3,2) circle (1pt);
    \filldraw[black] (3,3) circle (1pt);
    \filldraw[black] (4,0) circle (1pt);
    \filldraw[black] (4,3) circle (1pt);
    \filldraw[black] (5,0) circle (1pt);
    \filldraw[black] (4.63,1) circle (1pt);
    \filldraw[black] (5.37,2) circle (1pt);
    \filldraw[black] (5,3) circle (1pt);
    \filldraw[black] (6,0) circle (1pt);
    \filldraw[black] (5.63,2) circle (1pt);
    \filldraw[black] (6.37,1) circle (1pt);
    \filldraw[black] (6,3) circle (1pt);

\end{tikzpicture}
    \caption{The 7 networks that can appear for a fixed time.}
    \label{fig:fixed-time-networks}
\end{figure}
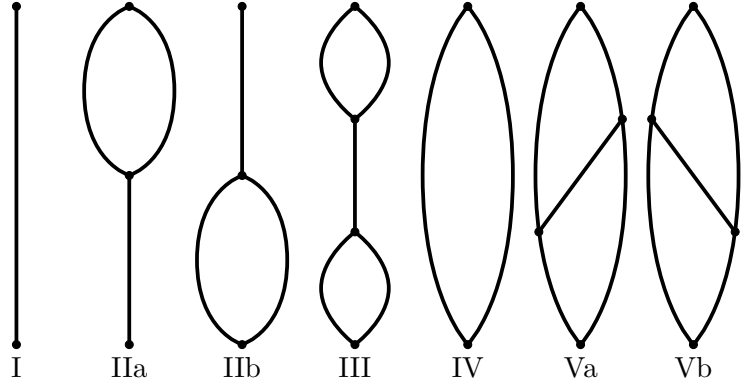 
The goal of this paper is to develop a dictionary translating features of the disjointness gap to the classification of geodesic networks. To manage the scope, we will only work with fixed start and end times, which by symmetry we may set as $s = 0, t = 1$.
In this restricted setting, Lemma \ref{L: fixed-time-possible-networks} shows that there are only five possible network types, or seven if we separate out networks which are equivalent by time reversal. These are depicted in Figure \ref{fig:fixed-time-networks}.  
The disjointness gap $\mathcal G$ also becomes more tractable when we fix the start and end times. Indeed, the \textbf{gap sheet} $G(x, y) = \mathcal G(x, 0; y, 1)$ can be written as a last passage Busemann function across the Airy line ensemble. 

Our main theorem relates features of the gap sheet to the geodesic network classification in Figure \ref{fig:fixed-time-networks}. We write, say, $(x,y) \in \textup{II}$ to mean that $(x,0;y,1)$ has network type $\textup{II}$. For $x, y \in \R$ we also write $G_x = G(x, \cdot)$ and $\hat G_y = G(\cdot, y)$.

    \begin{theorem} \label{T: fin-time-thm}
        Let $Q_{x,y}^{-,+} = (-\infty, x) \times (y, \infty)$,  $Q_{x,y}^{+,-} = (x, \infty) \times (-\infty, y)$, and let $Z:= G^{-1}(\{0\})$ be the zero set of the gap sheet. The following claims hold almost surely for all $x,y\in \R$.
        \begin{enumerate}
            \item $(x, y) \in \textup{I}$ if and only if neither $G_x$ has a local minimum at $y$ nor $\hat G_y$ has a local minimum at  $x$.
            \item $(x,y) \in \textup{IIa}$ if and only if $G_x$ has a local minimum at $y$ and $\hat G_y$ does not have a local minimum at $x$.
            \item $(x,y)\in \textup{IIb}$ if and only if $\hat G_y$ has a local minimum at $x$, and $G_x$ does not have a local minimum at $y$. 
            \item $(x,y) \in \textup{III}$ if and only if $G_x$ has a local minimum at $y$, $\hat G_y$ has a local minimum at $x$, and $(x,y) \notin Z$.
            \item $(x, y)\in\textup{IV}$ if and only if $(x,y) \in Z$ and $(x,y)$ is not an isolated point in both $Z \cap Q_{x,y}^{-,+}$ and $Z\cap Q_{x,y}^{+,-}$.
            \item $(x,y)\in \textup{Va}$ if and only if $(x,y) \in Z$, $(x,y)$ is an isolated point in $Z\cap Q_{x,y}^{-,+}$, and $(x,y)$ is not an isolated point in $Z \cap Q_{x,y}^{+,-}$.
            \item $(x,y)\in \textup{Vb}$ if and only if $(x,y) \in Z$, $(x,y)$ is not an isolated point in $Z\cap Q_{x,y}^{-,+}$, and $(x,y)$ is an isolated point in $Z \cap Q_{x,y}^{+,-}$.
        \end{enumerate}
    \end{theorem}

To help clarify the picture in parts $5, 6, 7$ above we note that any point $(x, y) \in Z$ is an isolated point in the set $Z \cap \{(x', y') \in \R^2 : (x', y') \le (x, y) \text{ or } (x, y) \le (x', y')\}$, where $\le$ is the coordinatewise partial order, see also Corollary \ref{Cor: dec-fns}.

\subsection{Applications} \label{sec: applications}
Theorem \ref{T: fin-time-thm} enables a translation of results about the gap sheet to results about geodesics in the directed landscape. To illustrate this, we record some results which fall out immediately from this correspondence. We prove these in Section \ref{Sec: applic-proofs}.

\begin{corollary} \label{Cor: network-density}
\quad 
\begin{enumerate} 

    \item Fix $x \in \R$. Then almost surely, every point $(x, y), y \in \R$ has network type \textup{I} or \textup{IIa}. Moreover, for any compact interval $[a, b]$ the law of the process $G_x(y) - G_x(0), y \in [a, b]$ is absolutely continuous with respect to the law of a two-sided Brownian motion on $K$, and so the set of points $(x, y), y \in (a, b)$ of type \textup{IIa} is absolutely continuous with respect to the law of the set of local minima of a Brownian motion.

    \item Almost surely, the set $Z = G^{-1}(\{0\}) = \textup{IV} \cup \textup{Va} \cup \textup{Vb}$ is nowhere dense.
     \item Almost surely, $Z$ has no isolated points. Equivalently, the gap sheet has no isolated zeroes.
    \item Almost surely, the set of points of type \textup{Va} is countable and dense in the set $Z$. Symmetrically, the same is true of the set \textup{Vb}.
\end{enumerate}

\end{corollary}

\begin{corollary} \label{Cor: dec-fns}
Let $Z = G^{-1}(\{0\})$ be the zero set for the gap sheet. Then almost surely, we can decompose $Z$ as a countable disjoint union of closed sets $(Z_i: i \in \N)$ such that:
\begin{itemize}[nosep]
    \item For any compact set $K \subset \R^2$, $Z_i \cap K \ne \emptyset$ for at most finitely many $i$. 
    \item Each $Z_i$ is contained in the closure of the graph of a strictly decreasing function.
\end{itemize}
\end{corollary}

 One explicit construction of the sets $Z_i$ is as follows. Let $A \subset \R$ be the (countable) set of points $a$ such that there is a geodesic $\pi:[0, 1] \to \R$ with $\pi(1/2) = a$. For $a_1 < a_2 \in A$, we can define the set
 $$
 Z_{a_1, a_2} = \{u \in Z : \pi_{u, L}(1/2) = a_1, \pi_{u, R}(1/2) = a_2\},
 $$
where $\pi_{u, L}, \pi_{u, R}$ are the leftmost and rightmost geodesics from $(u_1, 0)$ to $(u_2, 1)$. We check in Section \ref{Sec: applic-proofs} that the countable collection of sets $Z_{a_1, a_2}$ satisfies the corollary.

Not all aspects of the above corollaries are new. Indeed,the equality $Z = \textup{IV} \cup \textup{Va} \cup \textup{Vb}$ which follows from parts 5-7 of Theorem \ref{T: fin-time-thm} was first shown in \cite{bates2022hausdorff}. That paper also proves that $Z$ is nowhere dense and moreover, has Hausdorff dimension $1/2$. The zero set $Z$ is also the support of the \textbf{shock measure} for the Airy sheet $(x, y) \mapsto \L(x, 0; y, 1)$, can be used to construct the entire Airy sheet, and certain coordinate projections of the set $Z$ can be connected precisely to Brownian local time, see \cite{ganguly2023local, dauvergne2023last} for details.

\subsection{The semi-infinite case} 
\label{SS:semi-infinite}

A continuous function $\pi: [s,\infty)\to \R$ is a \textbf{semi-infinite geodesic} emanating from $(\pi(s), s)$ if, for any $t\geq s, \tau\big|_{[s,t]}$ is a geodesic. A semi-infinite geodesic has direction $\theta$ if $\pi(t)/t \to \infty$ as $t \to \infty$. For $p \in \R^2, \theta \in \R$, the semi-infinite geodesic network for $(p; \theta)$ is the collection of all semi-infinite geodesics from $p$ in direction $\theta$.  As in the finite setting, we can classify geodesic networks by looking at the underlying directed graph structure.

Two semi-infinite geodesics $\pi, \pi'$ in direction $\theta$ \textbf{eventually coalesce} if $\pi(r) = \pi'(r)$ for all large enough $r$. We say $\pi, \pi'$ are \textbf{eventually disjoint} if $\pi(r) \ne \pi'(r)$ for all large enough $r$.
Due to the work of \cite{rahman2023infinitegeodesicscompetitioninterfaces, busani2024stationary, busani2024nonexistencenoncoalescinginfinitegeodesics} we know that, for any $\theta \in \R$ and $p \in \R^2$, there exists at least one semi-infinite geodesic from $p$ in direction $\theta$. Moreover, for $\theta$ outside of a random countable dense set $\Xi$, all semi-infinite geodesics in direction $\theta$ eventually coalesce. For $\theta \in \Xi$, for every $p \in \R^2$ there exist leftmost and rightmost semi-infinite geodesics from $p$ in direction $\theta$ which are eventually disjoint. Thanks to results from \cite{busani2024nonexistencenoncoalescinginfinitegeodesics}, we know that even for $\theta \in \Xi$, all semi-infinite geodesics in direction $\theta$ eventually coalesce with either the rightmost or leftmost geodesic in direction $\theta$. Together with the network classification in the finite setting, this gives us $7$ possible semi-infinite network types for point-direction pairs of the form $(x, 0; \theta)$. These are recorded in Figure \ref{fig:semi-inf-networks}. 

It is natural to ask for an analogue of Theorem \ref{T: fin-time-thm} in the semi-infinite setting. While a complete analogue is possible, for brevity we have focused only on studying networks with two infinite ends, since this is where we see differences from the finite case. In other words, we are interested in a correspondence for networks in directions $\theta$ in the exceptional set $\Xi$. Our correspondence will go through a semi-infinite analogue of the gap sheet $G$. For $x \in \R, \theta \in \Xi$, let $\pi^\theta_{x, L}, \pi^\theta_{x,R}$ denote the leftmost and rightmost semi-infinite geodesics in direction $\theta$ from $(x,0)$. Then for $\theta \in \Xi$ we define 
\begin{equation}
\label{E:Gtheta}
   G^\theta(x) = \lim_{t \to \infty} \L(x, 0; \pi^\theta_{x,L}(t), t) + \L(x, 0; \pi^\theta_{x,R}(t), t) - \sup_{\pi_1, \pi_2} (\|\pi_1\|_\L + \|\pi_2\|_\L), 
\end{equation}
where the supremum is over paths $\pi_1$ from $(x, 0)$ to $(\pi^\theta_{x,L}(t), t)$ and $\pi_2$ from $(x, 0)$ to $(\pi^\theta_{x,R}(t), t)$ which are disjoint on the interval $(0, t)$. To see why this should be the correct analogue of the gap sheet, first note that if $\pi^\theta_{x, L}(t) = \pi^\theta_{x, R}(t)$, then the expression under the limit above is simply the disjointness gap $\mathcal G(x, 0; \pi^\theta_{x,L}(t), t)$. On the other hand, if we simply try to take a limit of disjointness gaps in direction $\theta$ without taking care about the terminal endpoints, then we can end up with too much local information at large times. Anchoring the endpoints of our paths on semi-infinite geodesics resolves this issue. For all $\theta \in \Xi$, the function $G^\theta$ is non-negative and continuous. We let $Z^\theta = (G^\theta)^{-1}(0)$ be the zero set of $G^\theta$.

We can now give our correspondence for semi-infinite networks. We write, say, $(x,\theta) \in \textup{IIa}^\infty$ to mean that the network from $(x,0)$ in direction $\theta$ has network type $\textup{IIa}^\infty$. 

\begin{figure}
    \centering
\begin{center}

\begin{tikzpicture}[transform shape = false, scale = 1.6] 

    \draw[->, line width = 0.5mm] plot [smooth, tension=0.3] coordinates {(-3,0) (-03.02,0.5) (-2.98,1.5) (-3, 1.9) (-3,2)};
    \filldraw[black] (-3,0) circle (1pt);
    \node at (-3,-0.25) {\Large $\text{I}^\infty$};

    \draw[->, line width = 0.5mm] plot [smooth, tension=1.1] coordinates {(-1.5, 1) (-1.15, 1.5) (-1, 2)};
    \draw[->, line width = 0.5mm] plot [smooth, tension=1.1] coordinates {(-1.5, 1) (-1.85, 1.5) (-2, 2)};
    \draw[line width = 0.5mm] plot [smooth, tension=1.1] coordinates {(-1.5, 0) (-1.45, 0.5) (-1.5, 1)};
    \filldraw[black] (-1.5, 0) circle (1pt);
    \node at (-1.5, -0.25) {\Large $\text{IIa}^\infty$};

    \draw[line width = 0.5mm] plot [smooth, tension=1.1] coordinates {(0,0) (-0.15, 0.5) (0, 1)};
    \draw[line width = 0.5mm] plot [smooth, tension=1.1] coordinates {(0,0) (0.15, 0.5) (0, 1)};
    \draw[->, line width = 0.5mm] plot [smooth, tension=0.3] coordinates {(0,1) (-0.02, 1.25) (0, 1.75) (0, 2)};
    \filldraw[black] (0, 0) circle (1pt);
    \node at (0, -0.25) {\Large $\text{IIb}^\infty$};

    \draw[->, line width = 0.5mm] plot [smooth, tension=1.1] coordinates {(1.5,1) (1.85, 1.5)(2,2)};
    \draw[->, line width = 0.5mm] plot [smooth, tension=1.1] coordinates {(1.5,1) (1.15, 1.5)(1,2)};
    \draw[line width = 0.5mm] plot [smooth, tension=1.1] coordinates {(1.5,0) (1.6, 0.25)(1.5,0.5)};
    \draw[line width = 0.5mm] plot [smooth, tension=1.1] coordinates {(1.5,0) (1.4, 0.25)(1.5,0.5)};
    \draw[line width = 0.5mm] plot [smooth, tension = 1.1] coordinates {(1.5,0.5) (1.48,0.75) (1.5,1)};
    \filldraw[black] (1.5,0) circle (1pt);
    \node at (1.5,-0.25) {\Large  $\text{III}^\infty$};

    \draw[->, line width = 0.5mm] plot [smooth, tension=1.1] coordinates {(3,0) (3.35, 1)(3.5,2)};
    \draw[->, line width = 0.5mm] plot [smooth, tension=1.1] coordinates {(3,0) (2.65, 1)(2.5,2)};
    \filldraw[black] (3,0) circle (1pt);
    \node at (3,-0.25) {\Large $\text{IV}^\infty$};

    \draw[->, line width = 0.5mm] plot [smooth, tension=1.1] coordinates {(4.5,0) (4.85, 1)(5,2)};
    \draw[->, line width = 0.5mm] plot [smooth, tension=1.1] coordinates {(4.5,0) (4.15, 1)(4,2)};
    \draw[line width=0.5mm, decoration={markings, mark=at position 0.5 with {\arrow{>}}}, postaction={decorate}] plot [smooth, tension=1.1] coordinates {(4.26,0.6) (4.4, 0.8) (4.7, 1) (4.9,1.2)};
    \filldraw[black] (4.5,0) circle (1pt);
    \node at (4.5,-0.25) {\Large $\text{Va}^\infty$};

    \draw[->, line width = 0.5mm] plot [smooth, tension=1.1] coordinates {(6,0) (6.35, 1)(6.5,2)};
    \draw[->, line width = 0.5mm] plot [smooth, tension=1.1] coordinates {(6,0) (5.65, 1)(5.5,2)};
    \draw[line width=0.5mm, decoration={markings, mark=at position 0.5 with {\arrow{>}}}, postaction={decorate}] plot [smooth, tension=1.1] coordinates {(6.24,0.6) (6.1, 0.8) (5.8, 1) (5.6,1.2)};
    \filldraw[black] (6,0) circle (1pt);
    \node at (6,-0.25) {\Large $\text{Vb}^\infty$};

\end{tikzpicture}
\end{center}
    \caption{The $7$ possible semi-infinite network types. One gets the vertex `at infinity' by connecting the paths at the top.}
    \label{fig:semi-inf-networks}
\end{figure}
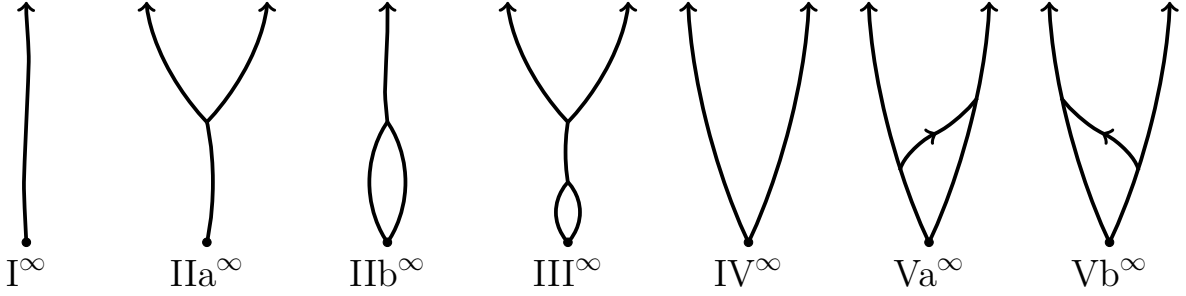

  \begin{theorem} \label{T: semi-inf-global-mins-close-by-intro}
  The following holds almost surely for all $x\in \R,\theta \in \Xi$:
        \begin{enumerate}
        \item $(x, \theta) \notin \textup{I}^\infty \cup \textup{IIb}^\infty$.
            \item $(x, \theta) \in \textup{IIa}^\infty$ if and only if $G^\theta$ does not have a local minimum at $x$.
            \item $(x, \theta) \in \textup{III}^\infty$ if and only if $G^\theta$ has a local minimum at $x$, but $x \notin Z^\theta$.
            \item $(x, \theta)\in\textup{IV}^\infty$ if and only if $x \in Z^\theta$ and $x$ is not an isolated point in both $Z^\theta \cap (-\infty, x]$ and $Z^\theta \cap [x, \infty)$. 
            \item $(x,\theta)\in \textup{Va}^\infty$ if and only if $x \in Z^\theta$, $x$ is an isolated point in $Z^\theta \cap (-\infty, x]$, but $x$ is not an isolated point in $Z^\theta \cap [x, \infty)$.
            \item $(x,\theta)\in \textup{Vb}^\infty$ if and only if $x \in Z^\theta$, $x$ is not an isolated point in $Z^\theta \cap (-\infty, x]$, but $x$ is an isolated point in $Z^\theta \cap [x, \infty)$.
        \end{enumerate}
    \end{theorem}
The major difference between Theorems \ref{T: fin-time-thm} and \ref{T: semi-inf-global-mins-close-by-intro} is that in the latter theorem, all statements refer only to the one-dimensional process $G^\theta$ for fixed $\theta$. Indeed, to view Theorem \ref{T: semi-inf-global-mins-close-by-intro} as a limit of Theorem \ref{T: fin-time-thm}, one approach would be to extend the function $(x, \theta) \mapsto G^\theta(x)$ to all of $\R^2$ by setting it equal to $\infty$ when $\theta \notin \Xi$ (this is what one obtains by naively applying definition \eqref{E:Gtheta} when $\theta \notin \Xi$). With this definition, parts $2$ and $3$ in Theorem \ref{T: semi-inf-global-mins-close-by-intro} are exact analogues of their counterparts in Theorem \ref{T: fin-time-thm}.

One advantage of working with semi-infinite geodesics rather than finite geodesics is that associated distance statistics, i.e.\, \textit{Busemann functions}, have more tractable representations. Indeed, the functions $G^\theta$ can be alternately written as a limit of multi-path Busemann functions, which can be represented in terms of reflected Brownian motions, see Proposition \ref{P: equiv-busemann-def}. For this reason we refer to $G^\theta$ as the \textbf{Busemann gap function} in direction $\theta$. We believe that for an appropriately sampled direction $\theta \in \Xi$, the gap function $G^\theta$ should also have a simple description as a variant of a two-sided Brownian motion reflected off of the line $x = 0$, started from an appropriate positive condition at $0$. While our methods do not provide such an exact description, we can show the following absolute continuity statement.

\begin{theorem}
\label{T: reflected-bm}
Let $\mu$ be any random probability measure on $\R$ such that $\mu(\Xi) = 1$, and let $\theta$ be a direction sampled from $\mu$. Then on any compact interval $I = [a, b]$, the process
$$
G^\theta|_I
$$
is absolutely continuous with respect to the law of a Brownian motion $W:[a, b] \to \R$ of diffusion coefficient $2$ and reflected off of the line $x = 0$, started from the initial condition $W(a) = G^\theta(a)$ (and independent of this initial condition).
\end{theorem}

By combining Theorem \ref{T: reflected-bm} and Theorem \ref{T: semi-inf-global-mins-close-by-intro}, we can better understand the structure of the sets $\textup{IIa}^\infty, \textup{III}^\infty$ etc. For example, for any $\theta \in \Xi$, the set $\textup{IV}^\infty$ has Hausdorff dimension $1/2$, and the set $Z^\theta = \textup{IV}^\infty \cup \textup{Va}^\infty \cup\textup{Vb}^\infty$ is locally absolutely continuous with respect to the zero set of a Brownian motion.

\subsubsection*{A note on the proofs}

The proofs of Theorem \ref{T: fin-time-thm} and \ref{T: semi-inf-global-mins-close-by-intro} are deterministic given a few basic soft geometric properties of the directed landscape which hold almost surely. These properties mainly follow from planarity and strong forms of coalescence, and should extend to other models. The properties used in the proof of Theorem \ref{T: fin-time-thm} are explicitly stated in Section \ref{sec: basic-prop}, and can be viewed as a (non-minimal) set of axioms for the sparse geometry of the directed landscape. In the semi-infinite setting, we require some additional properties for semi-infinite geodesics, stated at the end of Section \ref{sec: poss-types}. 

\subsubsection*{Future directions}
It would be interesting to see if the axiomatic approach here can be used to extend Theorem \ref{T: fin-time-thm} to the setting where the time coordinates can change, and all $27$ geodesic networks from \cite{dauvergne202327geodesicnetworksdirected} occur. At a minimum, one would need to consider a richer object than the gap sheet, which on its own cannot distinguish between a point pair supporting two disjoint geodesics versus three disjoint geodesics. We expect that the methods also have relevance in discrete random metrics in the KPZ class, e.g.\ first and last passage percolation and directed polymers.

More speculatively, one can ask if any of the results in the present paper have analogues beyond KPZ models. A natural candidate here is Kendall's Poisson roads metric \cite{kendall2017random, blanc2025geodesics}, a random fractal metric which should contain a non-trivial gap sheet. A more well-known model is Liouville quantum gravity, a family of random continuum planar metrics which exhibit qualitative similarities to the directed landscape. Much of the geodesic coalescence and network structure in Liouville quantum gravity is identical to that of the directed landscape, see \cite{gwynne2020confluence, bhatia2025strong} for results on strong coalescence, and \cite{le2010geodesics, angel2017stability, miller2025geodesics, gwynne2021geodesic} and references therein for results on geodesic networks. A major difference between Liouville quantum gravity and the directed landscape is that the latter is built from an independent background noise, whereas Liouville quantum gravity is constructed from a Gaussian free field, which is correlated. As a result, it is not clear that the gap sheet is ever non-zero in Liouville quantum gravity and a subtler object may need to be introduced.

\subsubsection*{Outline of the paper}

Section \ref{sec: background} is a preliminary section which gathers all facts about the directed landscape that we need for the proof of Theorem \ref{T: fin-time-thm}. The final part of this section, Section \ref{sec: basic-prop}, defines an almost sure set $\Omega$ on which Theorem \ref{T: fin-time-thm} holds. Beyond this section, our proof of Theorem \ref{T: fin-time-thm} is deterministic. Section \ref{Sec: I-III} establishes the equivalences in Theorem \ref{T: fin-time-thm} related to network types I-III, i.e.\, the equivalences at points where we do not have two disjoint geodesics. This includes the relationship between local minima of the gap sheet in either variable and certain geodesic networks. A small extension of the argument in Section \ref{Sec: I-III} also illuminates the geodesic structure at \textit{one-sided local minima} of the gap sheet (in either variable). We explore the structure of one-sided local minima in a short follow-up section, Section \ref{sec: one-sided-minima}. The equivalences in Theorem \ref{T: fin-time-thm} related to networks IV, V are proven in Section \ref{Sec: IV-V}, completing the proof of the theorem. A brief Section \ref{Sec: applic-proofs} proves the applications of Theorem \ref{T: fin-time-thm} recorded in Corollaries \ref{Cor: network-density} and \ref{Cor: dec-fns}.

The remainder of the paper is devoted to the semi-infinite case. Section \ref{S:prelim-semi-infinite} gathers additional preliminaries about semi-infinite paths and Busemann functions in the directed landscape. Section \ref{sec: semi-inf-cor} contains the proof of Theorem \ref{T: semi-inf-global-mins-close-by-intro}. As in the finite case, the proof is deterministic on an almost sure set $\Omega^\infty \subset \Omega$, which we define at the end of Section \ref{sec: poss-types}. The final Section \ref{sec: buse-gap-prop} contains the proof of Theorem \ref{T: reflected-bm}. An appendix proves the useful property that any weakly disjoint optimizer in the directed landscape is in fact strictly disjoint.

\subsubsection*{Common notation}

Throughout the paper, we use the notation $\pi$ for geodesics and $\tau$ for $2$-optimizers,. In particular, for $(x,s;y,t) \in \R^4_\uparrow$, we write $\pi_{(x,s;y,t),L}, \pi_{(x,s;y,t),R}$ for the leftmost and rightmost geodesics for $(x,s;y,t)$, and
$$
\tau_{(x,s;y,t),L} = (\tau_{(x,s;y,t),L,1}, \tau_{(x,s;y,t),L,2}), \qquad \tau_{(x,s;y,t),R} = (\tau_{(x,s;y,t),R,1}, \tau_{(x,s;y,t),R,2})
$$
for the leftmost and rightmost 2-optimizers. If it is clear what times $s,t$ we are going between (for example, usually $s = 0, t = 1$), we will reduce the above notation to, say, $\pi_{x,y,L}$. Similarly, if $x$ is fixed and $y$ varies, we simply write $\pi_{y,L}$. For $(x,s) \in \R^2, \theta \in \R$, we write $\pi_{(x,s),L}^\theta, \pi_{(x,s),R}^\theta$ for the leftmost and rightmost semi-infinite geodesics in direction $\theta$ starting from $(x,s)$. We use a similar $\tau$-notation for optimizers, and dropping subscripts and superscripts as discussed above when their values are clear.

Finally, for $s < t <r$, if $\pi_1: [s,t] \to \R, \pi_2:[t,r] \to \R$ are functions and $\pi_1(t) = \pi_2(t)$, then we write $\pi_1 \oplus \pi_2:[s, r] \to \R$ for the path concatenation:
$$
        (\pi_1 \oplus \pi_2) (\ell) := \begin{cases} 
      \pi_1(\ell) & \ell \leq t \\
      \pi_2(\ell) & t< \ell \leq r .
   \end{cases}
    $$

\section{Background}
\label{sec: background}

\subsection{The directed landscape}

The directed landscape $\L$ can be defined as a scaling limit of Poisson last passage percolation, which helps give some intuition for the object as a directed metric where paths are weighted by an independent background noise. Let $P$ be a homogeneous rate-$2$ Poisson point process on $\R^2$, and let, for $(p, ; q) = (x, s; y, t) \in \R^4_{\uparrow}$, 
$$
d(p; q) = \sup_{\pi} \#\{(z, r) \in P : r \in [s, t], z = \pi(r)\},
$$
where the supremum is over all $1$-Lipschitz paths $\pi:[s, t] \to \R$ with $\pi(s) = x, \pi(t) = y$.
\begin{theorem}
    \textup{(Corollary 13.12, \cite{dauvergne2022scalinglimitlongestincreasing})}. 
    We have the following convergence in distribution as $n\to \infty$, in the topology of uniform convergence on compact subsets of $\R^4_\uparrow$.
    $$
       \frac{d\left(xn^{2/3}, sn ;yn^{2/3}, tn\right) - 2n(t-s)}{n^{1/3}} \implies \L(x,s;y,t).
    $$
\end{theorem}
This result is especially nice as the commonly used symmetries of $\L$ follow from the symmetries of Poisson last passage percolation.

Alternately, let $C(X,\R)$ be the space of continuous functions from a set $X$ to $\R$ with the topology of uniform convergence on compact sets. The directed landscape is a random function in $C(\R_\uparrow^4,\R)$ built from its two-parameter marginal $\mathcal S \in C(\R^2, \R)$, given by
$$
\mathcal S(x, y) = \mathcal L(x, 0; y, 1),
$$
known as the \textbf{Airy sheet}. For $s > 0$, we say that $\mathcal S_s \in C(\R^2, \R)$ is an Airy sheet of scale $s$ if 
$$
s \mathcal S(x s^{-2}, y s^{-2}) \stackrel{d}{=} \mathcal S_s(x, y),
$$
as random elements of $C(\R^2, \R)$.
For the purposes of this paper, we will take the definition of the Airy sheet as a black box, see \cite{Dauvergne_2022} Definition 1.2 for details.

\begin{definition} \label{D: the-DL}

    The directed landscape is the unique (in law) random function in $C(\R_\uparrow^4,\R)$ satisfying: 
    \begin{enumerate}
        \item (Airy sheet marginals) For any $t \in \R$ and $s > 0$ the increment over the time intervals $[s,s+t^3)$ 
            $$
                (x,y) \mapsto \L(x,s;y,s+t^3)
            $$
            is an Airy sheet of scale s.
        \item (Independent increments) For any disjoint intervals $\{(s_i,t_i): i\in\{1,...,\ell\}\}$, the random functions 
            $$
                \L(\cdot, s_i; \cdot, t_i), \hspace{5mm} i\in\{1,..,\ell\}
            $$
            are independent.
        \item (Metric composition law) Almost surely, for any $r < s < t$ and $x,y \in \R$, we have that
        $$
            \L(x,s;y,t) = \max_{z\in \R} \big(\L(x,s;z,r)+\L(z,r;y,t) \big).
        $$
    \end{enumerate}
\end{definition}
At first order, the directed landscape $\mathcal L(x, s; y, t)$ is close to the parabolic directed metric $-(x-y)^2/(t-s)$. The next lemma gives a precise bound.

\begin{lemma} \label{L: DL-bound}
    \textup{(\cite{Dauvergne_2022} Corollary 10.7).} There exists a random constant $C > 0$ such that for all $u = (x,s;y,t) \in \R_\uparrow^4,$ we have

    $$
        \Bigg|\L(x,s;y,t) + \frac{(x-y)^2}{t-s}\Bigg| \leq C (t-s)^{1/3}\log^{4/3} \Bigg(\frac{2\|u\|_2+2)}{t-s}\Bigg) \log^{2/3} (\|u\|_2 + 2).
    $$
\end{lemma}

\subsection{The extended landscape}
\label{SS:extended}

Let 
$$
\mathfrak{X}_\uparrow := \{(\x,s;\y,t): \x,\y \in \R^k_{\leq}, s<t\},
$$
where $\R^k_\leq = \{\x \in \R^k: x_1 \leq ...\leq x_k\}$. For $(\x,s;\y,t) \in \mathfrak{X}_\uparrow$, define:
        \begin{align} \label{eq: ext-landscape}
            \L(\x,s;\y,t) = \sup_{\gamma_1,...,\gamma_k} \sum_{i=1}^k \|\gamma_i\|_\L,
        \end{align}
 where $\gamma_i$ are mutually disjoint and $\gamma_i(s) = x_i, \gamma_i(t) = y_i$. The extension of $\L$ as a function $\mathfrak{X}_\uparrow \to \R \cup\{\infty\}$ is known as the \textbf{extended (directed) landscape}. We call a collection of ordered disjoint paths (except possibly at their endpoints) $\gamma = (\gamma_1,...\gamma_k):[s,t]\to\R^k_\leq$ a \textbf{$\mathbf{k}$-optimizer} if $\|\gamma\|_\L := \|\gamma_1\|_\L + ... + \|\gamma_k\|_\L = \L(\gamma(s),s;\gamma(t),t)$.
    In the context of this paper, we are mainly interested in the case of $k=1,2$, but some proofs will be shown for general $k$ when it is not too much effort to do so. We will use $\pi$ to refers to geodesics (1-optimizers), and $\tau = (\tau_1.\tau_2)$ for 2-optimizers. 

The \textbf{Airy line ensemble}, $\{\hat \A_i:\R \to \R, i\in \N \},$ is a stationary process of almost surely strictly ordered functions (e.g. $\hat \A_i >\hat \A_{i+1}$). It was first defined in \cite{prahofer2002scale} as the limit of the polynuclear growth (PNG) model. The \textbf{parabolic Airy line ensemble}, defined as $\{\A_i:= \hat \A_i - x^2: i\in \N\}$, was used to construct the directed landscape in \cite{Dauvergne_2022}.
Recall that the extended landscape is the extension of $\L$ to a function $\mathfrak{X}_\uparrow \to \R\cup \{\infty\}$. Defining the \textbf{extended Airy sheet} as 
    \begin{align*}
        \S(\x,\y)=\L(\x,0;\y,1)
    \end{align*}
    enables the following connection. 
   
    \begin{theorem} \label{T: greenes-thm}
       Let $\{\A_i\}_{i\in \N}$ be a parabolic Airy line ensemble. There is a coupling between $\A$ and $\L$ such that almost surely: 
        \begin{itemize}
            \item [1.] \textup{(\cite{dauvergne2022disjointoptimizersdirectedlandscape}, Corollary 1.9)}. For all $y \in \R$ and $k\in \N$, $\sum_{i=1}^k \A_i(y) = \S(0^k,y^k).$
            \item [2.] \textup{(\cite{dauvergne2022disjointoptimizersdirectedlandscape}, Proposition 5.9)}. For all $y_1 \le y_2 \in \R$, we have 
                $$
                    \S((0, 0),\y) = \A_1(y_1) + \A_1(y_2) - \inf_{z\in [y_1,y_2]} \big(\A_1(z) - \A_2(z)\big)
                $$
        \end{itemize}
    \end{theorem}
    In Theorem \ref{T: greenes-thm}, note that the $k =2$ case of $1$ is equivalent to the $y_1 = y_2$ case of $2$.
    The extended landscape enjoys the following symmetries.

    \begin{lemma} \label{L: EDL-symmetries}
        \textup{(\cite{dauvergne2022disjointoptimizersdirectedlandscape} Lemma 6.10)}. Take $q > 0$, $r,c\in \R$, and for $\x \in \R^k$ let $T_c\x$ denote the shifted vector $(x_1+c,...,x_k+c)$. We have the following equalities in distribution for the extended landscape $\L$ as functions in $C(\mathfrak{X}_\uparrow, \R)$.
        \begin{itemize}
            \item [1.] Shift stationarity: $\L(\x,s;\y,t) \overset{d}{=} \L(T_c\x, s+r; T_c\y,t + r).$ 
            \item [2.] Flip symmetry: $\L(\x,s;\y,t) \overset{d}{=} \L(-\y,-t;-\x,-s).$ 
            \item[3.] $1:2:3$ scale invariance: $\L(\x,s;\y,t) \overset{d}{=} q\L(q^{-2} \x, q^{-3}s;q^{-2}\y,q^{-3}t).$ 
            \item[4.] Skew symmetry: $\L(\x,s;\y,t) + (t-s)^{-1}\|\x-\y\|_2^2 \overset{d}{=} \L(\x,s;T_c\y,t)+(t-s)^{-1}\|\x-T_c\y\|_2^2.$
        \end{itemize}
    \end{lemma}

\subsection{Geodesics and optimizers} \label{S: geo and optimizers}

Recall from the introduction that a $k$-tuple of paths $\gamma = (\gamma_1, \dots, \gamma_k)$ from $(\x, s)$ to $(\y, t)$ is a \textbf{(disjoint) $k$-optimizer} if $\gamma_1  < \cdots < \gamma_k$ on $(s, t)$, and if
\begin{equation}
\label{E:optimizer-eql}
   \L(\x, s; \y, t) = \|\gamma_1\|_\L + \cdots + \|\gamma_k\|_\L. 
\end{equation}

   \begin{theorem}
       [\cite{dauvergne2022disjointoptimizersdirectedlandscape}, Theorem 1.7] 
       \label{T:dj-exist}
       Almost surely, \eqref{E:optimizer-eql} is attained for every $(\x,s;\y,t) \in \mathfrak{X}_\uparrow$ by some k-optimizer. 
   \end{theorem}

Proving the existence of disjoint optimizers is the most difficult result in \cite{dauvergne2022disjointoptimizersdirectedlandscape}, and is established at the very end of the paper. For most of that paper, the authors instead work with \textbf{weak $k$-optimizers}. These are $k$-tuples of paths $\gamma = (\gamma_1, \dots, \gamma_k)$ from $(\x, s)$ to $(\y, t)$ satisfying \eqref{E:optimizer-eql} but with the weaker ordering constraint $\gamma_1 \le \cdots \le \gamma_k$. The space of weak optimizers is significantly easier to work with since weak ordering is a closed property in the uniform topology, and so existence and various properties of weak optimizers can be more easily established (i.e., by compactness arguments). In order to translate all the properties of weak optimizers established in \cite{dauvergne2022disjointoptimizersdirectedlandscape} to the disjoint optimizer setting we require here, in Appendix \ref{Appen: no weak opt} we prove that almost surely, all weak optimizers in the directed landscape are disjoint optimizers. Given this fact, in the remainder of the body of the paper we will only refer to disjoint optimizers. In particular, we state the upcoming Lemmas \ref{L:rlm-opt-unique-no-bubble} and \ref{L: geo-opt-rm-coal} for disjoint optimizers, even though they are proven in \cite{dauvergne2022disjointoptimizersdirectedlandscape} for weak optimizers. 

For this next lemma, inequalities are understood coordinate-wise. That is, $\gamma \leq \gamma'$ means $\gamma_1 \leq \gamma'_1, \dots, \gamma_k \leq \gamma'_k$.

    \begin{lemma} \label{L:rlm-opt-unique-no-bubble}
        The following properties hold almost surely for $\L$.
        \begin{enumerate}
            \item \textup{(\cite[Lemma 7.6]{dauvergne2022disjointoptimizersdirectedlandscape})} For every $u = (\x,s;\y,t) \in \mathfrak{X}_\uparrow$, there exist leftmost and rightmost $k$-optimizers $\gamma_{u,L}, \gamma_{u,R}$ for $u$ such that for any other $k$-optimizer $\gamma$ for $u$, we have $\gamma_{u,L} \leq \gamma \leq \gamma_{u,R}$. 
            
            \item \textup{(\cite[\textup{Lemma 7.7}]{dauvergne2022disjointoptimizersdirectedlandscape})}
            The functions $(\x,\y) \mapsto \gamma_{(\x,s;\y,t),R}$ and $(\x,\y) \mapsto \gamma_{(\x,s;\y,t),R}$  are monotone increasing in $\x$ and $\y$. 
            
            \item \textup{(\cite[\textup{Lemma 7.2}]{dauvergne2022disjointoptimizersdirectedlandscape})} In the case where $u = (\x, s; \y, t)$ is a rational point, $\gamma_{u,L}=\gamma_{u,R}$. That is,  there is always a unique $k$-optimizers between rational points. 
            \end{enumerate}
    \end{lemma}
    We will use $\pi_{u,L}$ and $\pi_{u,R}$ to refer to the leftmost and rightmost geodesics, and $\tau_{u,L}$ and $\tau_{u,R}$ for 2-optimizers. If $s,t$ are fixed (for example, $0,1$), we will use $\pi_{x,y,L}$ and $\pi_{x,y,R}$, and $\pi_{y,L}$ and $\pi_{y,R}$ in the case where $x$ is also fixed. The notation for 2-optimizers when $s,t,x$ are fixed is analogous.
    
    In the geodesic case when $k = 1$, the three properties above were all shown in \cite{Dauvergne_2022}. The uniqueness property in Lemma \ref{L:rlm-opt-unique-no-bubble}.$3$ fails to hold simultaneously for all points. However, its failure is essentially limited to exceptional behaviour near the start and endpoints. In the case of geodesics, this is shown by the following \textit{no geodesic bubbles} lemma.

    \begin{lemma}\textup{(\cite[Theorem 1]{bhatia2024dualitydirectedlandscapeapplications} or \cite[Lemma 3.3]{dauvergne202327geodesicnetworksdirected}.)}
       \label{L:no-bubbles-geo}
       The following holds almost surely for $\L$. For any geodesic $\pi:[s, t] \to \R$, and any interval $[s',t'] \subset(s,t)$, the path $\pi\big|_{[s',t']}$ is the unique geodesic between its end points.
    \end{lemma}

    The directed landscape enjoys coalescence of its geodesics and optimizers. Define the \textbf{overlap} of two $k$-optimizers, $\gamma:[s_1,t_t]\to \R^k_\leq,\gamma': [s_2,t_2]\to \R^k_\leq$, denoted by $\Ov(\gamma,\gamma')$, to be the closure of the set 
    $$
        \{r \in (s_1,t_1) \cap (s_2,t_2) : \gamma(r) = \gamma'(r)\}.
    $$
    Since $\L$ has no geodesic bubbles (Lemma \ref{L:no-bubbles-geo}), we have that $\Ov(\pi,\pi')$  is a closed interval if $\pi,\pi'$ are geodesics.

We say a sequence of $k$-tuples $\gamma_{(n)}:[s_n,t_n] \to \R^k_\leq$ converges to $\gamma: [s,t] \to \R^k_\leq$ \textbf{in the overlap topology} if $\Ov(\gamma_{(n)}, \gamma)$ is an interval whose endpoints converge to $s$ and $t$ as $n\to \infty$. We denote this by
    $$
        \gamma_{(n)} \overset{overlap}{\to} \gamma.
    $$
    
    \begin{lemma} (Overlap convergence of rightmost optimizers) \label{L: geo-opt-rm-coal}
        Almost surely, the following holds. Consider $s < t$, $\x,\y \in\R^k_\le$, and sequences $\x^{(n)}, \y^{(n)} \in\R^k_\leq$ such that $\x^{(n)} \to \x^+, \y^{(n)} \to \y^+$ and $\x^{(n)} > \x, \y^{(n)} > \y$. Then 
        $$
            \gamma_{(\x^{(n)},s;\y^{(n)},t),R} \overset{overlap}{\to} \gamma_{(\x,s;\y,t),R}.
        $$
        In the case that $\x^{(n)}=\x$ for all $n\in \N$, $\Ov( \gamma_{(\x^{(n)},s;\y^{(n)},t),R},\gamma_{(\x,s;\y,t),R}) = [s,r_n]$ where $r_n \to t$.
    \end{lemma}
    \begin{proof}
        The first part is Lemma 8.6 from \cite{dauvergne2022disjointoptimizersdirectedlandscape}. For the last part, consider a sequence $\z^{(n)} \to \x^+$ such that $\z^{(n)} > \x$. We know by the first part of the lemma that
        \begin{equation} \label{E:ovl-conv-int}
            \Ov( \gamma_{(\z^{(n)},s;\y^{(n)},t),R},\gamma_{(\x,s;\y,t),R}) = [s_n,r_n],
        \end{equation}

        where $s_n \to s, t_n \to t$. By Monotonicity (Lemma \ref{L:rlm-opt-unique-no-bubble}.2), \eqref{E:ovl-conv-int} also holds with $\z^{(n)}$ replaced by $\x$. Noticing that the rightmost optimizer restricted to a sub-interval of time must be the rightmost optimizer between its endpoints, we have that
        $$
            \gamma_{(\x^{(n)},s;\y^{(n)},t),R}\big|_{[0,s_n]} \text{ and }\gamma_{(\x,s;\y,t),R}\big|_{[0,s_n]}
        $$
        are both rightmost optimizers between their common endpoints, and so they must be equal, finishing the proof.
       
    \end{proof}
    Now, for a geodesic $\pi:[s,t]\to \R$, we write
    $$
        \mathfrak{g}\pi = \{(\pi(r),r): r\in [s,t]\}
    $$
    for the graph of $\pi$ (in space-time coordinates). 
   The next lemma shows that for geodesics, almost surely convergence in overlap is equivalent to the (a priori weaker) notion of convergence of graphs in the Hausdorff topology.

    \begin{lemma} 
    [\cite{dauvergne202327geodesicnetworksdirected} Proposition 3.5.2] \label{L: haus-iff-ov}
        Almost surely, for all sequences of geodesics $(\pi_n)_{n\in \N}$ and geodesics $\pi$, $\pi_n\overset{overlap}{\rightarrow}\pi$ if and only if $\mathfrak{g}\pi_n \to \mathfrak{g}\pi$ in the Hausdorff topology.
    \end{lemma}
    This correspondence between overlap and Hausdorff convergence is made especially useful by the following precompactness result.

    \begin{lemma}
        [\cite{dauvsarvirthreehalvesvariation2022} Lemma 3.1]  \label{L:geo-precompact}
        The following holds almost surely. Let $(x_n,s_n;y_n,t_n) \to (x,s;y,t)$, and let $\pi_n$ be any sequence of geodesics for $(x_n,s_n;y_n,t_n)$. Then the sequence of graphs $\mathfrak{g}\pi_n$ is precompact in the Hausdorff topology, and any subsequential limit of $\mathfrak{g}\pi_n$ is the graph of a geodesic for $(x,s;y,t)$.
    \end{lemma}

\subsection{Geodesic networks}
\label{SS:geodesic-networks}
For every $u = (x,s;y,t) \in \R^4_{\uparrow}$, define the \textbf{geodesic network} for $u$ by 
    $$
        \Gamma(u) = \bigcup \{\mathfrak{g}\pi : \pi \text{ is a geodesic for $u$}\}.
    $$
    The geodesic network $\Gamma(u)$ gives rise to a finite planar directed graph, given by mapping $(x, s), (y, t)$ and all branch points in $\Gamma(u)$ to vertices, and connecting these points together with directed edges according to the topology of $\Gamma(u)$.
    More precisely, following \cite{dauvergne202327geodesicnetworksdirected}, we can construct a directed graph $G(u) = (V, E)$ from $\Gamma(u)$ as follows:
    \begin{enumerate}
        \item The vertex set $V$ of $G(u)$ contains $(x,s)$ and $(y,t)$, and all points $v = (z,r) \in \Gamma(u)$ such that there are at least two distinct geodesics $\pi_1,\pi_2$ for $u$ with $\pi_1(r) = \pi_2(r) = z$ and $\pi_1|_{[r-\ep, r+\ep]} \neq \pi_2|_{[r-\ep, r+\ep]}$ for all $\ep > 0$. 
        \item For vertices $v_1 = (z_1, r_1), v_2 = (z_2,r_2)\in V$ with $r_1 < r_2$, the edge set $E$ contains one copy of the directed edge $(v_1,v_2)$ for every distinct path $\pi: [r_1,r_2]\to \R$ such that $\mathfrak{g}\pi \subset \Gamma(x,s;y,t)$, $\pi$ is a geodesic for $(v_1;v_2)$, and $V \cap \mathfrak{g}\pi = \{v_1, v_2\}$. We write $[e]$ for the embedding of a directed edge $e$ in the network $\Gamma(u)$.
    \end{enumerate}
    While \cite{dauvergne202327geodesicnetworksdirected} focuses on the underlying graph structure of the network, for this paper we will also care about the way the graph is drawn. We will classify geodesic networks up to the following notion of isomorphism, which is similar to a directed graph isomorphism, while also distinguishing between networks Va and Vb (see Figure \ref{fig:fixed-time-networks}). For $u = (x, s; y, t)$ and $u' = (x', s'; y', t')$, we say that $\Gamma(u) \sim \Gamma(v)$ if there is a directed graph isomorphism $\psi:G(u) \to G(u')$ and a homeomorphism $\phi:\Gamma(u) \to  \Gamma(u')$ such that 
    \begin{itemize}
        \item $\phi(v) = \psi(v)$ for all vertices $v \in V(u)$ and $\phi[e] = [\psi(e)]$ for all edges $e \in E(u)$. 
        \item $\phi$ preserves geodesic order, in the sense that if $\pi, \tau$ are any geodesics from $(x, s)$ to $(y, t)$ with $\pi \le \tau$, then $\phi(\pi) \le \phi(\tau)$.
    \end{itemize}
    We call the equivalence class of a network $\Gamma(u)$ under $\sim$ its \textbf{network type}. Our discussion of isomorphism and network type will become much more concrete in light of the upcoming network type classification, Lemma \ref{L: fixed-time-possible-networks}.

    The main result of \cite{dauvergne202327geodesicnetworksdirected} shows that (up to a weaker notion of transpose-isomorphism) there are exactly 27 geodesic network graphs that appear in the directed landscape, and finds their Hausdorff dimensions. As we are restricting our attention to networks between fixed times, we will not need the full power of this result. Rather, the main result we appeal to from \cite{dauvergne202327geodesicnetworksdirected} is a bound on size of the set of geodesic $3$-stars.

    A point $(x,s) \in \R^2$ is a $k$\textbf{-star point} if there exists
    $$
        \mathbf{z} \in \R^k_{<}:= \{\mathbf{z} \in \R^k : z_1 < ...<z_k\}
    $$
    and a time $t \in \R$ and if there are geodesics $\pi_1,...\pi_k$ satisfying $\pi_i(t) = z_i$ and $\pi_i(s) = x$ and such that $\pi_i(r) \ne \pi_j(r)$ for all $i \ne j, r \ne s$. In the case that $t < s$, we call $\pi = (\pi_1,...\pi_k)$ a \textbf{backward $k$-geodesic star}, and if $s < t$, we call $\pi$ a  \textbf{forward $k$-geodesic star}.

    To say something about how common certain $k$-star points are, we introduce the $1:2:3$-metric:
    $$
        d_{1:2:3}((x,s),(y,t)) = |t-s|^{1/3} + |y-x|^{1/2}.
    $$
    For $S\subset \R^2$, $\text{dim}_{1:2:3}(S)$ is the Hausdorff dimension with respect to the above metric. 

    \begin{lemma}
    \label{L:3-star}
        \textup{(\cite{dauvergne202327geodesicnetworksdirected} Theorem 1.5)} Let $\operatorname{Star}_k \subset \R^2$ be the set of all $k$-star points. Then almost surely,
        $$
            \textup{dim}_{1:2:3}(\operatorname{Star}_3) = 2.
        $$
    \end{lemma}

    The Hausdorff dimension in Lemma \ref{L:3-star} is small enough that $3$-stars do not appear at any fixed time almost surely.

    \begin{lemma} \label{L: rare-3-stars}
        Almost surely, for all $x\in\R, s\in\Q$, $(x,s)$ is not a $3$-star point.  
    \end{lemma}

    \begin{proof}
    Suppose that the lemma fails. Then by shift stationarity of $\L$ (Lemma \ref{L: EDL-symmetries}.1), every fixed time $t \in \R$ would have a positive probability $\delta > 0$ of containing a $3$-star. Since the directed landscape has independent time increments (Definition \ref{D: the-DL}.2), by Kolmogorov's $0-1$ law, the set 
    $$
    \{t \in \R : (x, t) \in \operatorname{Star}_3 \text{ for some } x \in \R\}
    $$
    has positive Lebesgue measure almost surely. On the other hand, letting $p_t: \R^2 \to \{0\}\times \R$ be the projection onto the line $\{0\}\times \R$, we see that 
    $$
    3\text{dim}(p_t(\text{Star}_3))  \leq \text{dim}_{1:2:3}(\text{Star}_3) = 2,
    $$
    where the Hausdorff dimension on the left-hand side is the usual Euclidean Hausdorff dimension. This uses that Hausdorff dimension decreases under projections. This is a contradiction, since any set in $\R$ of positive Lebesgue measure must have Euclidean Hausdorff dimension $1$.
    \end{proof}

    \subsection{The almost sure set for Theorem \ref{T: fin-time-thm}} \label{sec: basic-prop}

    We end the preliminary section by defining a set $\Omega$ on which Theorem \ref{T: fin-time-thm} will hold deterministically. We let $\Omega$ be the probability one set where the following properties hold.
    \begin{enumerate}
    [label=\textbf{P.\arabic*},ref=P.\arabic*]
        \item \label{P: cont} \textit{Continuity. }For $(x,s;y,t)\in \R^4_\uparrow$, $(x,s;y,t) \mapsto \L(x^k,s;y^k,t)$ is continuous for all $k\in \N$.  
        \item \label{P: rmlmopt-mon} \textit{Rightmost and leftmost optimizers exist, unique rational optimizers, and monotonicity.} As in Lemma \ref{L:rlm-opt-unique-no-bubble}
        \item \label{P: precomp} \textit{Precompactness of geodesics in the overlap topology.} By Lemmas \ref{L: haus-iff-ov} and \ref{L:geo-precompact} .
        \item \label{P: ov-rm-opt} \textit{Convergence in overlap of right/left most optimizers.} As in Lemma \ref{L: geo-opt-rm-coal}.
         \item \label{P: no-3-stars} \textit{3-stars are rare.} For all $(x,s)\in\R\times\Q$, $(x,s)$ is not a 3-star (Lemma \ref{L: rare-3-stars}.) 
           \item \label{P: no-bubbles} \textit{There are no geodesic bubbles}. As in Lemma \ref{L:no-bubbles-geo}
    \end{enumerate}

    We note that the above set of properties not at all minimal. For example, continuity and existence of leftmost and rightmost optimizers implies monotonicity. We record redundant properties above to make referencing them more clear. We work on $\Omega$ throughout the remaining sections. In particular, in Sections \ref{Sec: I-III} \ref{Sec: IV-V}, and \ref{Sec: applic-proofs}  the proofs are deterministic on $\Omega$. Section \ref{sec: one-sided-minima} on one-sided local minima invokes one extra property of the directed landscape, and Sections \ref{S:prelim-semi-infinite} and beyond use additional axiomatic properties of semi-infinite geodesics and optimizers.

\section{Finite time networks I-III} \label{Sec: I-III}

In this section, we prove all parts of Theorem \ref{T: fin-time-thm} which concern points where the gap sheet is non-zero, or equivalently, points where we do not have two disjoint geodesics. This covers network types \textup{I}, \textup{IIa}, \textup{IIb} and \textup{III}. Throughout the section we work on the almost sure set $\Omega$ where properties \ref{P: cont}-\ref{P: no-bubbles} hold.

We first identify the set of allowable network types from time $0$ to $1$.

\begin{lemma} \label{L: fixed-time-possible-networks}
    On $\Omega$, for all $x,y \in \R$, the network type of $(x,0;y,1)$ must be one of the networks pictured in figure \ref{fig:fixed-time-networks}.
\end{lemma}

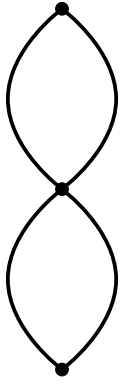
\begin{figure}
    \centering
   
\begin{tikzpicture}[scale=2.4]

  \draw[line width=0.5mm] plot [smooth, tension=1] coordinates {(3,0) (2.7,0.5) (3,1)};
  \draw[line width=0.5mm] plot [smooth, tension=1] coordinates {(3,0) (3.3,0.5) (3,1)};
  \draw[line width=0.5mm] plot [smooth, tension=1] coordinates {(3,1) (2.7,1.5) (3,2)};
  \draw[line width=0.5mm] plot [smooth, tension=1] coordinates {(3,1) (3.3,1.5) (3,2)};

    \filldraw[black] (3,0) circle (1pt);
    \filldraw[black] (3,1) circle (1pt);
    \filldraw[black] (3,2) circle (1pt);

\end{tikzpicture}
    \caption{Possible fixed-time network to be ruled out.}
    \label{fig:back2back-bubble-network}
\end{figure}

\begin{proof}
This can be deduced from the geodesic network classification in \cite{dauvergne202327geodesicnetworksdirected} and \ref{P: no-3-stars}, but it is also fairly straightforward to give an independent proof using the properties in Section \ref{sec: basic-prop}. We outline a proof, and leave some basic enumeration to the reader.

Indeed, let $G = (V, E)$ be the graph of a geodesic network from time $0$ to $1$, and let $p = (x, 0), q = (y, 1)$ be its source and sink vertices. 
First, $\deg(p) \le 2$ and $\deg(q) \le 2$ by \ref{P: no-3-stars}. Moreover, since $G$ has no interior bubbles (\ref{P: no-bubbles}), the number of paths from $p$ to $q$ in $G$ is at most $\deg(p) \cdot \deg(q) \le 4$. That is, $G$ is a finite planar graph. Now, the number of faces $|F|$ in $G$ is at most $\deg(p) + \deg(q) - 1$, since by \ref{P: no-bubbles}, every face is incident to either $p$ or $q$, and the unbounded face in $G$ is incident to both. Therefore by Euler's formula,
\begin{equation} \nonumber
|E| - |V| = |F| - 2 \le \deg(p) + \deg(q) -3.
\end{equation}
Next, any vertex in $V \setminus \{p, q\}$ has degree at least $3$ by construction, and so
$$
2|E| = \deg(p) + \deg(q) + \sum_{v \in V \setminus \{p, q\}} \deg(v) \ge \deg(p) + \deg(q) + 3(|V| - 2).
$$
Combining these two inequalities gives that $|V| \le \deg(p) + \deg(q)$. From this point, it is easy to enumerate finite planar directed graphs satisfying $\deg(p) \le 2, \deg(q) \le 2$ and with all faces incident to either $p$ or $q$. Indeed, the network graph of $(x,0;y,1)$ must be one of the seven pictured in Figure \ref{fig:fixed-time-networks}, or a graph like network III with the middle unique geodesic removed. This is pictured in Figure \ref{fig:back2back-bubble-network}. 

Suppose that this is the network from $(x, 0)$ to $(y, 1)$, and let $(z, r)$ be the interior branch point shared by all geodesics. Observe that 
\begin{align*}
\L(x^2, 0; y^2, 1) &\ge \L(x^2, 0; z^2, r) + \L(z^2, r; y^2, 1) \\
&= 2\L(x, 0; z, r) + 2\L(z, r; y, 1) \\
&=2 \L(x, 0; y, 1).
\end{align*}
Here the first bound is the triangle inequality for the extended landscape, the second equality uses that there are disjoint geodesics from $(x, 0)$ to $(z, r)$, and also from $(z, r)$ to $(y, 1)$, and the final equality uses that $(z, r)$ is on a geodesic from $(x, 0)$ to $(y, 1)$. Therefore $\L(x^2, 0; y^2, 1) = 2 \L(x, 0; y, 1)$, and so because the extended landscape value $\L(x^2, 0; y^2, 1)$ is always attained by disjoint paths by \ref{P: rmlmopt-mon}, there must be at least two disjoint geodesics in the network from $(x, 0)$ to $(y, 1)$, contradicting that our network graph is of the form in Figure \ref{fig:back2back-bubble-network}.

\end{proof}

In order to understand the relationship between geodesic network types and the gap sheet, we need more information on how $2$-optimizers behave at points with non-unique geodesics. The following lemma tells us that, if a network has two geodesics going into one endpoint, any $2$-optimizer must eventually stay on either one of these geodesics.

\begin{lemma}\label{L:opt=geo}
On $\Omega$, the following holds for all $x,y \in \R$.
Suppose that any geodesic $\pi$ for $u = (x,0;y,1)$ has the property that there exists an $r \in [0,1)$ such that there are two disjoint geodesics from $(\pi(r),r)$ to $(\pi(1),1)$. (In other words, $\Gamma(u)$ is not of type \textup{I} or \textup{IIb}). Then there exists a time $\tilde t \in [0, 1)$ such that for any disjoint 2-optimizer $\tau = (\tau_1, \tau_2)$ for $u$, we have 
$$
(\tau_1, \tau_2)\big|_{[\tilde t,1]} = (\pi_{u,L}, \pi_{u,R})\big|_{[\tilde t,1]}.
$$
\end{lemma}
\begin{proof}
    First, by examining the network types \textup{IIa}, \textup{III}, \textup{IV}, and \textup{Va/b}, we see that  there exists an $\epsilon>0$ such that $\Gamma(u) \cap ([1-\epsilon, 1] \times \R)$ contains exactly two paths. 
Letting $\tau$ be any 2-optimizer for $u$, we show that $\tau_1\big|_{[1-\epsilon,1]} \leq \pi_{u,L}\big|_{[1-\epsilon,1]}$. Suppose this is not the case. Then there exists a $t' \in [1-\ep,1]$ such that:
    $$
        \pi_{u, L}(t') < \tau_1(t').
    $$
    Let
    \begin{align*}
        s_1 &= \max\{0 \le s < t': \pi_{u, L}(s) = \tau_1(s)\}, \text{ and}\\
        s_2 &= \min\{1 \ge s > t': \pi_{u, L}(s) = \tau_1(s)\},
    \end{align*}
    By continuity of geodesics and optimizers, the maximum and minimum above exist, $s_1 <  s_2$, and 
    \begin{align}
     \label{E: l-geo-lessthan-opt}   \pi_{u,L}\big|_{(s_1,s_2)} < \tau_1\big|_{(s_1,s_2)}.
    \end{align}
    \begin{description}
        \item [\textbf{Case I:}] $\|  \pi_{u, L}\big|_{[s_1,s_2]}\|_\L > \|\tau_1\big|_{[s_1,s_2]}  \|_\L $.
        
    In this case, we can find a better optimizer by replacing the section of $\tau_1$ between $s_1$ and $s_2$ with $\pi_{u,L}$. Explicitly,
    \begin{align}
      \label{E: new-opt-w-LM-geo}  \|(\tau_1\big|_{[0,s_1]} \oplus \pi_{u,L}\big|_{[s_1,s_2]} \oplus \tau_1\big|_{[s_2,1]}, \tau_2)\|_\L > \|\tau\|_\L.
    \end{align}
    The left hand side of \eqref{E: new-opt-w-LM-geo} still defines a disjoint 2-optimizer by \eqref{E: l-geo-lessthan-opt}.
    
        \item[\textbf{Case II:}] $\|  \pi_{u,L}\big|_{[s_1,s_2]}\|_\L = \|\tau_1\big|_{[s_1,s_2]}  \|_\L $.
        
        In this case, $\tau_1\big|_{[s_1,s_2]}$ is a geodesic. By \eqref{E: l-geo-lessthan-opt} and an analysis of the possible networks \textup{IIa}, \textup{III}, \textup{IV}, and \textup{Va/b}, we see that $s_2 = 1$, and there are exactly two disjoint geodesics for the point $v_1= (\pi_{u,L}(s_1),s_1;0,1)$.
        Further, 
        \begin{align}
          \label{E: opt-takes-rm-mid-geo}  \tau_1\big|_{[s_1,1]} = \pi_{v_1,R}.
        \end{align}
        By \eqref{E: opt-takes-rm-mid-geo} and the fact that $\tau_1 < \tau_2$ on $(0, 1)$, we have that 
        \begin{align}
          \label{E: r-opt-r-of-mid-geo}  \tau_2\big|_{[s_1,1)} > \pi_{v_1,R}.
        \end{align}
        By \eqref{E: r-opt-r-of-mid-geo} and monotonicity (\ref{P: rmlmopt-mon}), we also have
        \begin{align}
         \label{E: geo-for-rm-opt-monot} 
         \pi_{v_1,R} \le \pi_{v_2, R}, \qquad v_2 := (\tau_2(s_1),s_1;0,1).
        \end{align}
        Also by \eqref{E: r-opt-r-of-mid-geo} and the fact that $\pi_{u, L}|_{[s_1, 1]}, \tau_1|_{[s_1, 1]}$ are disjoint geodesics ending at $(y, 1)$, we must have that 
        \begin{align}
           \label{E: rm-opt-strictly-worse} \|\tau_2\big|_{[s_1,1]}\|_\L < \|\pi_{v_2,R}\|_\L,
        \end{align}
        as otherwise, $(y,1)$ would be a backwards  $3$-star, which is not possible by \ref{P: no-3-stars}.
        Putting together \eqref{E: opt-takes-rm-mid-geo}, \eqref{E: r-opt-r-of-mid-geo}, \eqref{E: geo-for-rm-opt-monot}, \eqref{E: rm-opt-strictly-worse}, we see that replacing $\tau_1\big|_{[s_1,1]}$ with $\pi_{u,L}\big|_{[s_1,1]}$ and $\tau_2 \big|_{[s_1,1]}$ with $\pi_{v_2,R} = \pi_{v_1, L}$ gives a better optimizer:
        $$
            \|\tau\|_\L < \|(\tau_1\big|_{[0,s_1]}\oplus \pi_{v_1,L}, \tau_2\big|_{[0,s_1]}\oplus \pi_{v_2,R})\|_\L
        $$
        This is a contradiction.

    \end{description}
In summary, we have shown that $\tau_1\big|_{[1-\epsilon,1]} \leq \pi_{u,L}\big|_{[1-\epsilon,1]}$. A similar argument shows that  $\tau_2\big|_{[1-\ep,1]} \geq \pi_{u,R}\big|_{[1-\ep,1]}$, and so
$$
\tau_1(1-\ep) \le \pi_{u, L}(1-\ep) < \pi_{u, R}(1-\ep) \le \tau_2(1-\ep).
$$
Combining these inequalities with the disjointness of $\pi_{u, L}, \pi_{u, R}$ on $[1-\ep, 1]$ and geodesic ordering (\ref{P: rmlmopt-mon}), we see that there are disjoint geodesics from $(\tau_1(1-\ep), 1-\ep)$ and $(\tau_2(1-\ep), 1-\ep)$ to $(y, 1)$. Therefore $\tau_1, \tau_2$ must be geodesics on $[1-\ep, 1]$. 

Therefore $(\tau_1, \tau_2)$ must equal $(\pi_{u, L}, \pi_{u, R})$ after some time $t_1 < 1$, as otherwise we would have a three-star at $(y, 1)$, which does not happen by \ref{P: no-3-stars}. We get $\tilde t$ by taking the maximum of $t_1$ when $(\tau_1,\tau_2)$ is taking to be the leftmost and rightmost optimizer for $u$ (by \ref{P: rmlmopt-mon}).

\end{proof}
Next, recall that the optimizers in last passage percolation interlace (e.g. see \cite[Lemma 2.3]{dauvergne2022disjointoptimizersdirectedlandscape}, \cite[Section 2]{basu2022interlacing}). We record a version of this fact for the directed landscape in one particular case.

\begin{lemma}  \label{L: geo-opt-ordering}
    On $\Omega$, the following holds almost surely. For all $u = (x,s;y,t) \in \R^4_\uparrow$, and intervals $[r_1,r_2] \subset [s,t]$ such that $\Gamma(u) \cap (\R \times [r_1,r_2])$ contains exactly one path, if $\tau = (\tau_1,\tau_2), \pi$ are any 2-optimizer and geodesic for $u$, we have that
    $$
        \tau_1\big|_{[r_1,r_2]} \leq \pi \big|_{[r_1,r_2]} \leq \tau_2\big|_{[r_1,r_2]}.
    $$
\end{lemma}
\begin{proof}
    Say this were not the case, then, without loss of generality, there is an $r' \in [r_1,r_2]$ such that $\tau_2(r') < \pi(r')$. By continuity, there is an $\ep > 0$ such that $\tau_2\big|_{[r'-\ep,r'+\ep]} < \pi\big|_{[r'-\ep,r'+\ep]}$. Now take 
    \begin{align*}
        s_1 = \sup\{r < r_1: \tau_2(r) = \pi(r)\}, \quad s_2 = \inf\{r > r_2: \tau_2(r) = \pi(r)\}.
    \end{align*}
    We see that $s_1 \in [s,r_1), s_2 \in (r_2,t]$, $\tau_2(s_1) = \pi(s_1), \tau_2(s_2) = \pi(s_2)$ and $\tau_2\big|_{[s_1,s_2]} \leq \pi\big|_{[s_1,s_2]}$ by continuity. However,  $\tau_2\big|_{[s_1,s_2]}$ cannot be a geodesic because $\Gamma(u) \cap (\R \times [r_1,r_2])$ and $\tau_2\big|_{[r'-\ep,r'+\ep]} < \pi\big|_{[r'-\ep,r'+\ep]}$. This means we can replace $\tau_2$ with $\pi$ on the interval $[s_1,s_2]$ to get a path of larger length, while still maintaining continuity and disjointness with $\tau_1$ (since $\tau_1 < \tau_2 \leq \pi$ on $(s,t)$). This contradicts $\tau$ being a 2-optimizer.
\end{proof}

The last lemma we need before proving our correspondence for networks I-III allows us to find sequences of points with non-unique geodesics which coalesce in a particular way.
\begin{lemma} \label{L: seq-disj-geos}
    On $\Omega$, the following holds for all $x,y \in \R$. For any $w > y$, let 
    
    \begin{align*}
        a_w := \sup\{t\in [0,1]: \pi_{x,y,R}(t) = \pi_{x,w,R}(t)  \}
    \end{align*} Then there exist sequences $z_n \downarrow y, a_{z_n} \uparrow 1$, as well as geodesics $\pi_{x,z_n,1}, \pi_{x,z_n,2}$ for $(x,0;z_n,1)$ such that
    \begin{align*}
        \Ov(\pi_{x,y,R}, \pi_{x,z_n,1}) &= [0,a_{z_{n+1}}],\\
        \Ov(\pi_{x,y,R}, \pi_{x,z_n,2}) &= [0,a_{z_{n}}],\text{ and}\\
        \pi_{x,z_n,1}(t) &< \pi_{x,z_n,2}(t) \text{ for all $t\in(a_{z_{n}},1)$}.
    \end{align*}
\end{lemma}

This is illustrated in Figure \ref{fig:geo-bubbles-close}.

\begin{proof}
    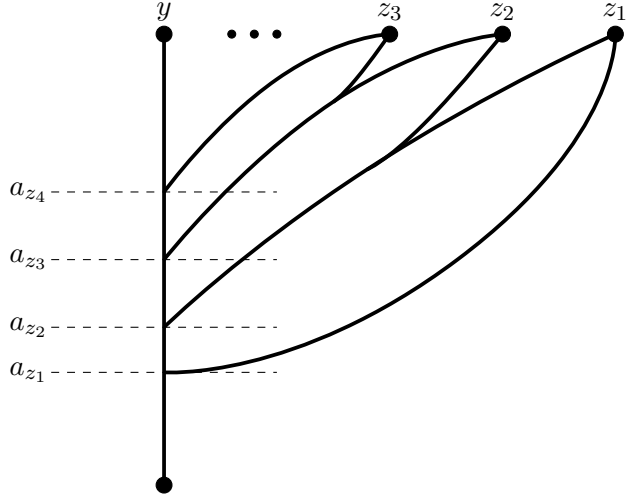
\begin{figure}
        \centering
        \begin{tikzpicture}[scale = 3]
             \draw[line width=0.5mm] plot [smooth, tension=1.3] coordinates {(0,0)  (0,2)};
             \draw[line width=0.5mm] plot [smooth, tension=1.3] coordinates {(0,0.5)  (1.3,1)(2,2)};
             \draw[line width=0.5mm] plot [smooth, tension=1] coordinates {(0,0.7)  (0.9,1.4)(2,2)};
             \draw[line width=0.5mm] plot [smooth, tension=1] coordinates { (0.9,1.4)(1.15,1.6) (1.5,2)};
             \draw[line width=0.5mm] plot [smooth, tension=1] coordinates { (0,1) (0.75,1.7)(1.5,2)};
             \draw[line width=0.5mm] plot [smooth, tension=1] coordinates { (0.75,1.7) (0.85, 1.8)(1,2)};
             \draw[line width=0.5mm] plot [smooth, tension=1] coordinates { (0,1.3) (0.5,1.8)(1,2)};


            \draw[dashed, ] plot [smooth, tension=1.3] coordinates {(-0.5,0.5) (0.5,0.5)};
            \draw[dashed, ] plot [smooth, tension=1.3] coordinates {(-0.5,0.7) (0.5,0.7)};
            \draw[dashed, ] plot [smooth, tension=1.3] coordinates {(-0.5,1) (0.5,1)};
            \draw[dashed, ] plot [smooth, tension=1.3] coordinates {(-0.5,1.3) (0.5,1.3)};

            \filldraw[black] (0,0) circle (1pt); 
             \filldraw[black] (0,2) circle (1pt);
              \filldraw[black] (1,2) circle (1pt);
               \filldraw[black] (1.5,2) circle (1pt);
                \filldraw[black] (2,2) circle (1pt);
             \filldraw[black] (0.3,2) circle (0.5pt);
              \filldraw[black] (0.4,2) circle (0.5pt);
               \filldraw[black] (0.5,2) circle (0.5pt);

            \node at (0,2.1) {$y$};
            \node at (1,2.1) {$z_3$};
            \node at (1.5,2.1) {$z_2$};
            \node at (2,2.1) {$z_1$};

            \node at (-0.6,0.5) {$a_{z_1}$};
            \node at (-0.6,0.7) {$a_{z_2}$};
            \node at (-0.6,1) {$a_{z_3}$};
            \node at (-0.6,1.3) {$a_{z_4}$};
        \end{tikzpicture}
        \caption{Sketch of Lemma \ref{L: seq-disj-geos}. }
        \label{fig:geo-bubbles-close}
    \end{figure}
    We will drop the $x$ from all notation here, since it is fixed.
    First we prove the following:
    \begin{align}
       \label{E: branch-dont-touch} \text{If $y < w_1 < w_2$ and $a_{w_1} > a_{w_2} > 0$, then $\pi_{w_1,R}(t) \neq \pi_{w_2,R}(t)$ for all $t\in(a_{w_2},1] $}.
    \end{align}
    By the definition of $a_w$, we see that
    \begin{align}
     \label{E: geo-touch-point}   \pi_{w_1,R}(a_{w_2}) &= \pi_{y,R}(a_{w_2}) = \pi_{w_2,R}(a_{w_2}), \text{ and}\\
       \label{E: begin-branch-order} \pi_{w_1,R}\big|_{(a_{w_2},a_{w_1})} &= \pi_{y,R}\big|_{(a_{w_2},a_{w_1})} < \pi_{w_2,R}\big|_{(a_{w_2},a_{w_1})}.
    \end{align}
    So if \eqref{E: branch-dont-touch} did not hold for some $t' \in (a_{w_2},1)$, by \eqref{E: geo-touch-point} and \eqref{E: begin-branch-order}, we would get a geodesic bubble between $( \pi_{w_1,R}(a_{w_2}),a_{w_2})$ and $ (\pi_{w_1,R}(t'),t')$, which does not happen by \ref{P: no-bubbles}. Clearly \eqref{E: branch-dont-touch}  holds for $t=1$ since $w_1 < w_2$.

    Next, we show that the set of numbers $\{a_w:w > y\}\cap (1/2,1)$ cannot contain a strictly increasing subsequence $(a_n)_n$ converging to any number other than $1$. Suppose that $a_n \uparrow  b < 1$. Then for each $a_n$, choose a $w_n$ such that $a_n = a_{w_n}$. By monotonicity of geodesics (\ref{P: rmlmopt-mon}), the sequence $w_n$ is decreasing, and $(\pi_{w_{n},R})_n$ is a (weakly) decreasing sequence of geodesics bounded below by $\pi_{y,R}$. Hence $(\pi_{x,w_{n},R})_n$ converges to a geodesic $\pi_{w_\infty}$ for $(x,0;w_\infty,1)$ in overlap  
  (\ref{P: precomp}), where $w_\infty = \lim_{n\to\infty} w_n$. On the other hand, by \eqref{E: branch-dont-touch},
    \begin{align*}
        \pi_{w_\infty}\big|_{(b,1)} \leq \pi_{w_\infty,R}\big|_{(b,1)} < \pi_{w_n,R}\big|_{(b,1)},
    \end{align*}
    contradicting that $(\pi_{w_{n},R})_n$ converges to $\pi_{w_\infty}$ in overlap. Therefore, $a_n$ must converge to $1$. A similar argument shows that $\{a_w: w > y\}\cap (1/2,1)$ cannot contain any decreasing convergent sequences.
    
      From this we see that the set $\{a_w: w > y\} \cap (1/2,1)$ is at most countable and has no accumulation points in $[1/2, 1)$. Moreover, by convergence in overlap of rightmost geodesics, we have that $a_w \uparrow 1$ as $w \downarrow y$.  Therefore, $\{a_w: w > y\}$ can be written as an increasing sequence $a_{n}\uparrow 1$. Now for each $a_n$, pick a $v_n$ such that $a_{v_n} = a_n$. By monotonicity of rightmost geodesics (\ref{P: rmlmopt-mon}), we see that $(v_n)_n$ is decreasing, and $v_n\to y^+$. Furthermore, for all $n$, and all $v\in (v_{n+1},v_{n})$, $a_{v} \in \{a_{v_{n+1}},a_{v_n}\}$. Therefore by convergence of rightmost geodesics in overlap (\ref{P: ov-rm-opt}), for all $n \in \N$ there exists $\ep_n > 0$ such that
    \begin{align}
      \label{E: ep-equals-vn}   a_{v_{n+1}+\ep_n} = a_{v_{n+1}}.
    \end{align}
    Now we define 
    \begin{align*}
        z_n := \sup\{v\in (v_{n+1}, v_{n}):a_{v} = a_{v_{n+1}}  \}.
    \end{align*}
        note $z_n > v_{n+1}$ by \eqref{E: ep-equals-vn}. Now we can take a sequence $t_k \to z_n^-$ such that $a_{t_k} = a_{n+1}$. Taking a convergent subsequence of $(\pi_{x,t_k,R})_k$ in the overlap topology (\ref{P: precomp}) gives a geodesic $\pi_{z_n,1}$  for $(x,0;z_n,1)$  such that 
    $$
        \Ov(\pi_{x,z_n,1},\pi_{x,y,R}) = [0,a_{n+1}].
    $$
    Now, if $z_n = v_{n}$, then taking $\pi_{z_n,2} = \pi_{v_n,R}$, and using the same argument used to show \eqref{E: branch-dont-touch} gives the required result. If $z_n < v_{n}$ then we can take a sequence $u_k \to z_n^+$ such that $a_{u_k} = a_{n}$, which has a convergence subsequence, via the same argument as before we get a geodesic $\pi_{z_n,2}$ for $(x,0;z_n,1)$ such that 
    $$
        \Ov(\pi_{z_n,2},\pi_{y,R}) = [0,a_{n}],
    $$
    and by \eqref{E: branch-dont-touch} and noticing that $a_{z_n} = a_n$, we are done.
\end{proof}

We are now ready to prove the first part of our correspondence, giving an equivalent condition for networks I,II, and III. 

\begin{prop}
\label{P:upperbubble-locmin}
On $\Omega$, the following holds for all $x, y \in \R$.
\begin{enumerate}
\item If there exists an $r \in [0,1)$ such that $\Gamma(x,0;y,1) \cap (\R\times [r,1))$ contains exactly two disjoint paths, then $G_x$ has a \textbf{strict} local minimum at $y$.
    \item If $G_x:\R \to \R$ has a (weak) local minimum at a point $y \in \R$, then there exists an $r \in [0,1)$ such that $\Gamma(x,0;y,1) \cap (\R\times [r,1))$ contains exactly two disjoint paths.
    \item All local minima of $G_x$ are strict local minima.
\end{enumerate}
\end{prop}

\begin{proof}
\textbf{Proof of Part $1$.} \qquad By Lemma \ref{L:opt=geo},  we can find a $t^*>r$ such that the rightmost 2-optimizer $\tau_y$ for $(x,0;y,1)$ satisfies
    \begin{align}
        \label{E: opt=geos_at_end}
        \tau_y\big|_{[t^*,1]} = (\pi_{y,L}, \pi_{y,R})\big|_{[t^*,1]}.
    \end{align}
    Here, $\pi_{y,L},\pi_{y,R}$ are the leftmost and rightmost geodesics for $(x,0;y,1)$.

     Denoting the rightmost 2-optimizer for $(x,0;z,1)$ by $\tau_z = (\tau_{z,1},\tau_{z,2})$, take a point $z > y$ close enough so that, by \eqref{P: ov-rm-opt}, we have
    \begin{align}
        \label{E: geo_coal_close_enough_backwrd_dir_local_min_thm}
        \pi_{y,R} \big|_{[0,t^*]} &= \pi_{z,R} \big|_{[0,t^*]}, \qquad \text{ and }\\
        \label{E: opt_coal_close_enough} \tau_y\big|_{[0,t^*]} &= \tau_z\big|_{[0,t^*]}.
    \end{align}
We will show that $G_x(y) < G_x(z)$, implying that $y$ is a strict right-sided minimum for $G_x$. A symmetric argument gives that $y$ is a strict left-sided minimum for $G_x$, which will complete the proof of $1$. To show $G_x(y) < G_x(z)$, we start by expanding the definition of the gap function. 
     \begin{align}
       \notag G_x(y) &=  \|\pi_{y,L}\|_\L + \|\pi_{y,R}\|_\L - \|\tau_{y,1}\|_\L - \|\tau_{y,2}\|_\L \\
       \notag 
       \begin{split}
            &= \|\pi_{y,L}\big|_{[0,t^*]}\|_\L + \|\pi_{y,L}\big|_{[t^*,1]}\|_\L
           + \|\pi_{y,R}\big|_{[0,t^*]}\|_\L + \|\pi_{y,R}\big|_{[t^*,1]}\|_\L \\ &  \hspace{5mm}-\|\tau_{y,1}\big|_{[0,t^*]}\|_\L -\|\tau_{y,1}\big|_{[t^*,1]}\|_\L  - \|\tau_{y,2}\big|_{[0,t^*]}\|_\L-\|\tau_{y,2}\big|_{[t^*,1]}\|_\L \\
       \end{split}\\
       \label{E: gap_geo_bubble_at_top} &= \|\pi_{y,L}\big|_{[0,t^*]}\|_\L
           + \|\pi_{y,R}\big|_{[0,t^*]}\|_\L - \|\tau_{y,1}\big|_{[0,t^*]}\|_\L - \|\tau_{y,2}\big|_{[0,t^*]}\|_\L.
    \end{align}
    The last line follows from  (\ref{E: opt=geos_at_end}). We also expand $G_x(z)$:
    \begin{align}
       \label{E: gap_rational_close_geo} 
       \begin{split}
           G_x(z) &= 2\|\pi_{z,R}\big|_{[0,t^*]}\|_\L + 2\|\pi_{z,R}\big|_{[t^*,1]}\|_\L \\
       & \quad - \|\tau_{z,1}\big|_{[0,t^*]}\|_\L -\|\tau_{z,1}\big|_{[t^*,1]}\|_\L  - \|\tau_{z,2}\big|_{[0,t^*]}\|_\L-\|\tau_{z,2}\big|_{[t^*,1]}\|_\L
       \end{split} 
    \end{align}
    Combining \eqref{E: geo_coal_close_enough_backwrd_dir_local_min_thm}, \eqref{E: opt_coal_close_enough}, \eqref{E: gap_geo_bubble_at_top},and \eqref{E: gap_rational_close_geo}, reduces the inequality $G_x(y)< G_x(z)$ to the inequality

    \begin{align}
        \label{E: bub-imp-loc-min-int-step} \|\pi_{y,L} \big|_{[0,t^*]}\|_\L  +\|\tau_{z,1} \big|_{[t^*,1]}\|_\L+\|\tau_{z,2} \big|_{[t^*,1]}\|_\L <   \|\pi_{z,R} \big|_{[0,t^*]}\|_\L +2\|\pi_{z,R} \big|_{[t^*,1]}\|_\L.
    \end{align}
    Now, by \eqref{E: opt=geos_at_end} and \eqref{E: opt_coal_close_enough}, we have that
    \begin{align*}
        \pi_{y,L}(t^*) = \tau_{z,1}(t^*).
    \end{align*}
    Hence we have that $\pi_{y,L}\big|_{[0, t^*]}\oplus\tau_{z,1}\big|_{[t^*,1]}$ gives a path from $(x,0)$ to $(z,1)$. This tells us that
    \begin{align}
       \label{E: backwrd-dir-worse-path} \|\pi_{y,L}\big|_{[0, t^*]}\|_\L + \|\tau_{z,1}\big|_{[t^*,1]}\|_\L \leq \|\pi_{z,R}\|_\L = \|\pi_{z,R} \big|_{[0,t^*]}\|_\L +\|\pi_{z,R} \big|_{[t^*,1]}\|_\L.
    \end{align}
    In fact, the inequality in \eqref{E: backwrd-dir-worse-path} must be strict. To see why this is true, first note that if it was not the case, then $\pi_{y,L}\big|_{[0, t^*]}\oplus\tau_{z,1}\big|_{[t^*,1]}$ would be a geodesic from $(x,0)$ to $(z,1)$. However, $\tau_{z,1}(t^*) < \pi_{y,R}(t^*)$ since $\Gamma(x,0;y,1) \cap (\R \times [r,1))$ contains two disjoint paths and $r < t^* < 1$. Now, define
    \begin{align*} 
    s^{**} &:=     \sup\{s\in [t^*,1): \tau_{z,1}(s) = \pi_{y,L}(s)\},                                \\
        t^{**}&:= \inf\{t\in (t^*,1): \tau_{z,1}(t) = \pi_{y,R}(t)\}.
    \end{align*}
    Note that the two sets above are non-empty: the first contains $t^*$, and the second contains some point by continuity since $z > y$.
   By continuity, $\tau_{z,1}(s^{**}) = \pi_{y,L}(s^{**})$, $\tau_{z,1}(t^{**}) = \pi_{y,R}(t^{**}),$ and $s^{**} < t^{**}$. This means that the three paths
    $$
\pi_{y,L}\big|_{(s^{**}, t^{**})} < \tau_{z,1}\big|_{(s^{**}, t^{**})} < \pi_{y,R}\big|_{(s^{**}, t^{**})},
    $$
    are all disjoint geodesics. 
    This contradicts that $\Gamma(x,0;y,1) \cap (\R \times [r,1])$ contains exactly two paths. The strictness of the inequality in  \eqref{E: backwrd-dir-worse-path}
    reduces \eqref{E: bub-imp-loc-min-int-step} to the weak inequality
    \begin{align}
      \notag \|\tau_{z,2}\big|_{[t^*,1]}\|_\L \leq \|\pi_{z,R}\big|_{[t^*,1]}\|_\L,
    \end{align}
    which is clear since $\tau_{z,2}(t^*) = \pi_{z,R}(t^*)$ and $\pi_{z,R}\big|_{[t^*,1]}$ is a geodesic.

  \textbf{Proof of Part $2$.} \qquad  Assume $y$ is a local minimum of $G_x(\cdot)$. Suppose that the implication is false, that is there is no $r \in [0,1)$ such that there are $\Gamma(x,0;y,t) \cap (\R\times [r,1))$ contains exactly two disjoint paths.
  By examining the possible networks appearing for a fixed time (Lemma \ref{L: fixed-time-possible-networks}), we see there exists a $t^*$ such that $\Gamma(x,0;y,1) \cap (\R \times [t^*,1])$ contains precisely one path. Denote the rightmost geodesic for $(x,0;y,1)$ by $\pi_y$.      
    
    Now, let $\tau_y$ be any 2-optimizer for $(x,0;y,1)$, and define
   
    \begin{align*}
        \mathfrak{l} := \sup\{t \in [0,1): \tau_{y,1}(t) = \pi_{y}(t)\}, \qquad \mathfrak{r}  := \sup\{t \in [0,1): \tau_{y,2}(t) = \pi_{y}(t)\}.
    \end{align*}
    We claim that $\mathfrak{l}=1$ or $\mathfrak{r}=1$. Indeed, suppose without loss of generality that $\mathfrak{l}<\mathfrak{r}<1$. Then, considering the path
    $$    \tilde\tau_{y,2}:=\tau_{y,2}\big|_{[0,\mathfrak{r}]} \oplus \pi_{y}\big|_{[\mathfrak{r},1]},
    $$
    we see that $\|\tilde \tau_{y,2}\|_\L > \|\tau_{y,2}\|_\L$ since $\tau_{y,2}\big|_{[\mathfrak{r},1]}$ cannot be a geodesic due to the fact that $\mathfrak{r}<1$, $\tau_{y,2}(\mathfrak{r}) = \pi_{y}(\mathfrak{r})$ and $\Gamma(x,0; y,1) \cap (\R \times [t^*,1])$ contains once path. Therefore
     $$
        \|\tau_{y,1}\|_{\L} + \|\tilde \tau_{y,2}\|_{\L} > \|\tau_{y,1}\|_{\L} + \| \tau_{y,2}\|_{\L}.
    $$
    Moreover, since $\mathfrak{l}<\mathfrak{r}$, the definition of $\mathfrak{l}$ implies that $(\tau_{y,1}, \tilde \tau_{y,2})$ is a pair of disjoint paths, which together with the above inequality contradicts that $\tau_y$ is a 2-optimizer. 

    Now, without loss of generality, we assume that $\mathfrak{l}=1$ as the  $\mathfrak{r}=1$ case has a symmetric proof (i.e.\ using an approximation from the left with leftmost instead of rightmost geodesics). By Lemma \ref{L: seq-disj-geos}, we have sequences $z_n \downarrow {y}, a_{z_n}\uparrow 1$, and geodesics $\pi_{z_n,1}$ and $\pi_{z_n,2}$ for $(x,0;z_n,1)$  such that \begin{align}
       \label{E:bubble-params} \begin{split}
            \Ov(\pi_y, \pi_{z_n,1}) &= [0,a_{z_{n+1}}]\\
            \Ov(\pi_{y}, \pi_{z_n,2}) &= [0,a_{z_{n}}]\\
            \pi_{z_n,1}(t) &< \pi_{z_n,2}(t) \text{ for all $t\in(a_{z_{n}},1)$}.
        \end{split}
    \end{align}
    For the remainder of the proof, we denote $a_{z_n}$ as $a_n$.

    Now because $\mathfrak{l}=1$ we can find a subsequence $\{a_{n_k}\}_k$ of $\{a_n\}_n$, and a corresponding subsequence $\{z_{n_k}\}_k$ of $\{z_n\}_n$ such that for every $n_k$, there exists a $t_k \in [a_{n_k}, a_{n_k+1}]$ such that 
    \begin{align*}
        \tau_{y,1}(t_k) = \pi_{y}(t_k).
    \end{align*}
    We will show for any $k$ that 
    \begin{align}
        G_x(z_{n_k}) < G_x(y) \label{E: forward_wts},
    \end{align}
    contradicting that $y$ is a local minimum. By Lemma \ref{L: geo-opt-ordering}, we have that
    \begin{align}
       \notag \tau_{y,2}\big|_{[t^*,1]} \geq \pi_{y}\big|_{[t^*,1]}, \label{E: rightmost-opt-ordering}
    \end{align}
    so, as long as $k$ is large enough, $\tau_{y,2}\big|_{[a_{n_k},1]}$ must intersect $\pi_{z_{n_k },2}\big|_{[a_{n_k},1]}$ at least once. This is because both are continuous functions and 
    $$
        \tau_{y,2}(a_{n_k}) \geq \pi_{y}(a_{n_k}) = \pi_{z_{{n_k}},2}(a_{n_k}),
    $$
    while
    $$\tau_{y,2}(1) = y < z_{n_k} = \pi_{z_{n_k},2}(1).$$
    Now, define
    $$
        s_k = \inf\{s \in [a_{n_k},1] : \tau_{y,2}(s) = \pi_{z_{n_k},2}(s)\},
    $$
    and notice that $(\tau_{y,1}\big|_{[0,{t_k}]} \oplus \pi_{z_{n_k},1}\big|_{[t_k,1]}, \tau_{y,2}\big|_{[0,{s_k}]}\oplus \pi_{z_{n_k},2}\big|_{[s_k,1]} )$ is a 2-tuple of continuous disjoint paths from $(x,0)$ to $(z_{n_k},1)$ by (\ref{E:bubble-params}) and the definition of 
    $t_k,s_k$. Hence,
    \begin{equation}
       \L(x^2,0;z_{n_k}^2,1) \geq \|(\tau_{y,1}\big|_{[0,{t_k}]} \oplus \pi_{z_{n_k},1}\big|_{[t_k,1]}, \tau_{y,2}\big|_{[0,{s_k}]}\oplus \pi_{z_{n_k},2}\big|_{[s_k,1]} )\|_\L.
       \label{E: new-cand-2-opt} 
    \end{equation}
   Therefore, by \eqref{E: new-cand-2-opt} and the definition of $G_x$, we have that
    \begin{align}
     G_x(z_{n_k}) &\leq
        \|\pi_{z_{n_k},1}\|_\L + \|\pi_{z_{n_k},2}\|_\L - \|\tau_{y,1}\big|_{[0,t_k]}\|_\L - \|\pi_{z_{n_k},1}\big|_{[t_k,1]}\|_\L - \|\tau_{y,2}\big|_{[0,s_k]}\|_\L - \|\pi_{z_{n_k},2}\big|_{[s_k,1]} \|_\L  \notag        \\ 
     &= \|\pi_{z_{n_k},1}\big|_{[0,t_k]}\|_\L + \|\pi_{z_{n_k},2}\big|_{[0,s_k]}\|_\L - \|\tau_{y,1}\big|_{[0,t_k]}\|_\L -\|\tau_{y,2}\big|_{[0,s_k]}\|_\L .    \label{E: non_opt_length_sum} 
    \end{align}
    Meanwhile, 
    \begin{align}
        G_x(y) = 2\|\pi_{y}\|_\L -\|\tau_{y,1}\big|_{[0,t_k]}\|_\L - \|\tau_{y,1}\big|_{[t_k,1]}\|_\L -\|\tau_{y,2}\big|_{[0,s_k]}\|_\L - \|\tau_{y,2}\big|_{[s_k,1]}\|_\L,  \label{E: y_gap_sum}
    \end{align}
     and so to show \eqref{E: forward_wts}, it is enough to show \eqref{E: non_opt_length_sum} $<$ \eqref{E: y_gap_sum}. Now, by \eqref{E:bubble-params}, we have that $\pi_y|_{[0, t_k]} = \pi_{z_{n_k}, 1}|_{[0, t_k]}$. Using this, we have that the inequality \eqref{E: non_opt_length_sum} $<$ \eqref{E: y_gap_sum} is equivalent to the following:
    \begin{align}
        \|\pi_{z_{n_k},2}\big|_{[0,s_k]}\|_\L +\|\tau_{y,1}\big|_{[t_k,1]}\|_\L +\|\tau_{y,2}\big|_{[s_k,1]}\|_\L         < \|\pi_{y}\|_\L +  \|\pi_{y}\big|_{[t_k, 1]}\|_\L.    \label{E: frwrd_dir_final_wts}
    \end{align}
    Now because $\pi_{z_{n_k},2}(s_k) =  \tau_{y,2}(s_k)$, we have that $\pi_{z_{n_k},2}\big|_{[0,s_k]}\oplus  \tau_{y,2}\big|_{[s_k,1]}$ is a continuous path for $(x,0;y,1)$. Therefore, since $\pi_y$ is a geodesic for $(x,0;y,1)$, it follows that
    \begin{align}
       \|\pi_{z_{n_k},2}\big|_{[0,s_k]}\|_\L + \|\tau_{y,2}\big|_{[s_k,1]}\|_\L \leq \|\pi_{y}\|_\L. \label{E: rightmost-non-opt-path-ineq}
    \end{align}
   Moreover,  \begin{align}\|\tau_{y,1}\big|_{[t_k,1]}\|_\L \leq \|\pi_{y}\big|_{[t_k,1]}\|_\L, \label{E: leftmost-opt-non-opt}
    \end{align}
    since $\tau_{y,1}$ and $\pi_y$ are equal at the times $t_k, 1$, and $\pi_y$ is a geodesic.  Now because $\Gamma(x,0;y,1) \cap (\R \times [t^*,1])$ contains exactly one path, at least one of \eqref{E: rightmost-non-opt-path-ineq} or \eqref{E: leftmost-opt-non-opt} must be strict.
    Using this, \eqref{E: frwrd_dir_final_wts}, \eqref{E: rightmost-non-opt-path-ineq}, and \eqref{E: leftmost-opt-non-opt} implies \eqref{E: forward_wts}.

    \textbf{Proof of Part $3$.} \qquad This is immediate by combining parts $1$ and $2$.
\end{proof}

Note that by the time-reversal symmetry of the directed landscape (i.e.\ Lemma \ref{L: EDL-symmetries}.2) the same results of Proposition \ref{P:upperbubble-locmin} also hold if we shift the starting point of the geodesic. We record this in the following corollary.

\begin{corollary} \label{C: lowerbubble-locmin}
The following holds on $\Omega$ for all $y \in \R$. The function $G(\cdot, y)$ has a strict local minimum at a point $x$ if and only if there exists an $r \in (0,1]$ such that $\Gamma(x, 0; y, 1) \cap (\R \times (0, r])$ contains exactly two disjoint paths. All local minima of $G(\cdot,y)$ are strict.
\end{corollary}

We can now conclude parts $1-4$ of Theorem \ref{T: fin-time-thm}. This is the content of the following theorem.

\begin{theorem}
\label{T:casei-iii}
Almost surely, the following holds for all $x, y \in \R$. Let $u = (x, 0; y, 1)$. Then $u$ has network type \textup{I}, \textup{IIa}, \textup{IIb}, or \textup{III} if and only if $G(x, y) > 0$. Moreover, 
\begin{itemize}
    \item $u$ has network type \textup{IIa} if and only if $G_x$ has a strict local minimum at $y$, but $\hat G_y$ does not have a local minimum at $x$.
    \item $u$ has network type \textup{IIb} if and only if $\hat G_y$ has a strict local minimum at $x$, but $G_x$ does not have a local minimum at $y$. 
     \item $u$ has network type \textup{III} if and only if $G_x$ has a strict local minimum at $y$, and $\hat G_y$ has a strict local minimum at $x$.
    \item $u$ has network type \textup{I} if and only if $G_x$ does not have a local minimum at $y$, and $\hat G_y$ does not have a local minimum at $x$.
\end{itemize}
\end{theorem}

\begin{proof}
  We have that $G(x, y) = 0$ if and only if there are two disjoint geodesics from $(x, 0)$ to $(y, 1)$. This uses the existence of disjoint optimizers, \ref{P: rmlmopt-mon}. In other words, $G(x, y) > 0$ if and only if $u$ has network type \textup{I}, \textup{IIa}, \textup{IIb}, or \textup{III}. The remaining part of the classification is Proposition \ref{P:upperbubble-locmin} and Corollary \ref{C: lowerbubble-locmin}.  
\end{proof}

\section{One-sided local minima}
\label{sec: one-sided-minima}

    From the proof of Proposition \ref{P:upperbubble-locmin}.2, we can extract the following relationship between one-sided local minima and the behaviour of 2-optimizers. We state this corollary for right-sided local minima; a symmetric statement holds for left-sided local minima.
    
    \begin{corollary}
        \label{C: one-side-local-min}
      On $\Omega$, the following holds for all $x, y \in \R$. Let $\tau_{x, y}$ be any $2$-optimizer from $(x, 0)$ to $(y, 1)$, and let $\pi_{x, y, R}$ be the rightmost geodesic.
      
    If $y$ is a right-sided local minimum for $G_x(\cdot)$ (meaning for all small enough $\ep > 0$, $G_x(y) \leq G_x(y + \ep)$), then there exists an $r > 0$ such that $\tau_{x,y,2}\big|_{[r,1]} = \pi_{x,y,R}\big|_{[r,1]}$. In other words, at a right-sided local minimum, the rightmost path in any $2$-optimizer eventually stays on the rightmost geodesic.
      \end{corollary}

     \begin{proof}
It suffices to prove this in case when the network type for $(x, 0; y, 1)$ is \textup{I} or \textup{IIb}, as the others are covered by Lemma \ref{L:opt=geo}. With notation as in the proof of Proposition \ref{P:upperbubble-locmin}.2, if Corollary \ref{C: one-side-local-min} fails, then $\mathfrak{l} = 1$. Following that proof, this gives rise to a sequence of points $z_{n_k} \downarrow y$ such that $G_x(z_{n_k}) < G_x(y)$, contradicting $y$ being a right-sided local minimum.
     \end{proof} 

It is natural to ask whether the two conditions in Corollary \ref{C: one-side-local-min} are equivalent. We are not sure whether this is true. However, we can show that the existence of a right-sided local minimum for $G_x$ is equivalent to a stronger geometric condition.

\begin{prop}
 The following holds almost surely for all $x,y \in \R$. Setting some notation, fix $x \in \R$, let $\pi_y$ be the rightmost geodesic for $(x, 0; y, 1)$ and $\tau_y$ the rightmost $2$-optimizer.
    The following conditions are equivalent (where condition i) means condition ia) + ib)):
    \begin{itemize}
        \item [i)] \begin{itemize}
            \item [a)] $(y,1)$ is a backwards 2-star, that is, there exists an $r\in (0,1)$ and geodesics $\pi_1,\pi_2:[r,1]\to \R$ such that $\pi_1(1) =\pi_2(1) = y$, and $\pi_1(t)<\pi_2(t)$ for all $t\in (r,1]$.
            \item [b)] $\pi_y \big|_{[r,1]} = \pi_2\big|_{[r,1]}$ and $\tau_y\big|_{[r,1]} = (\pi_1,\pi_2)$. \\
            \end{itemize}
        \item[ii)] For small enough $\ep > 0$,
            \begin{align*}
                \L(x^2,0;(y,y+\ep),1) -\L(x^2, 0;y^2, 1) = \L(x,0;y+\ep,1) - \L(x,0;y,1).
            \end{align*}
        \item[iii)] $G_x(y)$ is a right-sided local minimum.
    \end{itemize}
\end{prop}
\begin{proof} 
    \begin{description}
        \item[$i) \implies ii)$:] \qquad Take $\ep>0$ small enough so that, by \ref{P: ov-rm-opt},
        \begin{align*}
            \Ov(\pi_y, \pi_{y+\ep}) &\supset [0,r] \quad \text{ and } \\
            \Ov(\tau_y, \tau_{(y,y+\ep)}) &\supset [0,r].
        \end{align*}
        Here $\tau_{(y,y+\ep)}$ is the rightmost $2$-optimizer from $(x^2, 0)$ to $((y, y + \ep), 0)$.
        By the above, monotonicity, and condition $i)$ we see that:
        \begin{align*}
            \tau_y\big|_{[0,r]} \oplus (\pi_1, \pi_{y+\ep})\big|_{[r,1]}
        \end{align*}
        must also be a disjoint 2-optimizer from $(x^2, 0)$ to $((y, y + \ep), 1)$. Now using this disjoint 2-optimizer and condition $i)b)$, we must have that:
        \begin{align*}
             \L(x^2,0;(y,y+\ep),1)) -\L(x^2, 0;y^2, 1) &= \|\pi_{y+\ep}\big|_{[r,1]}\|_\L - \|\pi_y\big|_{[r,1]}\|_\L \\
            &=\L(x,0;y+\ep,1) - \L(x,0;y,1),
        \end{align*}
        where the final equality uses that  $\Ov(\pi_y, \pi_{y+\ep}) \supset [0,r]$.
           
        \item[$ii) \iff iii)$:] \qquad 
        We claim that the following is true, simultaneously for all $x, y, z \in \R$ with $y < z$:
        \begin{equation}
        \label{E:minimum-formula}
        \L(x^2,0; (y,z),1) = \L(x, 0; y, 1) + \L(x, 0; z, 1) - \min_{w\in[y, z]} G(x, w).
        \end{equation}
To see why \eqref{E:minimum-formula} holds, first observe that all terms above are continuous functions of $x, y, z$ so it suffices to prove that \eqref{E:minimum-formula} holds for any fixed choice of $x, y, z$ almost surely. Next, by the translation invariance of the directed landscape (Lemma \ref{L: EDL-symmetries}.1), we may set $x = 0$. In this case, the equation is implied by Theorem \ref{T: greenes-thm} and the definition of $G$.
            Using \eqref{E:minimum-formula}, we see the statement in $ii)$ reduces to the fact that for small enough $\ep > 0$:
            \begin{align*}
                G(x, y) = \min_{w\in[y, y + \ep]} G(x, w),
            \end{align*}
            which is exactly the condition for $G_x$ being a weak right-sided local minimum at $y$.
            
        \item[$ii) + iii) \implies i)$:] 
          \qquad By Corollary \ref{C: one-side-local-min}, there exists $r'\in (0,1)$ such that 
            \begin{align}
              \label{E: rm-opt-eq-geod} \tau_{y,2}\big|_{[r',1]} = \pi_y\big|_{[r',1]}.
            \end{align}
            Letting $\ep > 0$ be small enough so that by \ref{P: ov-rm-opt},
            \begin{align}
               \label{E: overlap-cond-rm-min-prop} \Ov(\pi_y, \pi_{y+\ep}) \supset [0,r'],
            \end{align}
            we get:
            \begin{align*}
                \|\tau_{y, 1}\|_\L + \|\tau_{y, 2}\big|_{[0,r']} \oplus \pi_{y+\ep}\big|_{[r',1]}\|_\L \leq \L(x^2, 0;(y,y+\ep),1) ,
            \end{align*}
            since the paths on the left hand side are disjoint. Hence, assuming condition $ii)$ is true, we have that:
            \begin{align*}
                \L(x,0;y+\ep,1) - \L(x,0;y,1) &= \L(x^2,0;(y,y+\ep),1) -\L(x^2, 0;y^2, 1) \\
                    &\geq \|\tau_{y, 1}\|_\L + \|\tau_{y, 2}\big|_{[0,r']} \oplus \pi_{y+\ep}\big|_{[r',1]}\|_\L - \L(x^2, 0;y^2, 1)  \\
                    & = \|\tau_{y,1}\|_\L + \|\tau_{y,2}\big|_{[0,r']} \oplus \pi_{y+\ep}\big|_{[r',1]}\|_\L - \|\tau_{y,1}\|_\L -\|\tau_{y,2}\|_\L \\
                    & = \|\pi_{y+\ep} \big|_{[r',1]}\|_\L - \|\tau_{y,2}\big|_{[r',1]}\|_\L \\
                    &=  \|\pi_{y+\ep} \big|_{[r',1]}\|_\L - \|\pi_y\big|_{[r',1]}\|_\L.
            \end{align*}
            The last line follows by \eqref{E: rm-opt-eq-geod}. Now by \eqref{E: overlap-cond-rm-min-prop}, we have that 
            $$ 
                \|\pi_{y+\ep} \big|_{[r',1]} - \|\pi_y\big|_{[r',1]}\|_\L = \L(x,0;y+\ep,1) - \L(x,0;y,1).
            $$ 
            This tells us that the above inequality is an equality, or in other words, the pair of disjoint paths $(\tau_{y, 1},\tau_{y, 2}\big|_{[0,r']} \oplus \pi_{y+\ep}|_{[r', 1]})$ is optimal. Now letting $r= \sup \{t\in [0,1]: \pi_y(t) = \pi_{y+\ep}(t)\}$, we have that $r \ge r'$, and, by monotonicity of geodesics and \eqref{E: rm-opt-eq-geod}, 
            $$
            \tau_{y,1}(s) < \pi_y(s)<\pi_{y+\ep}(s) \qquad \text{ for all } s \in (r, 1).
            $$
            Therefore by monotonicity of geodesics, the rightmost geodesic from $\tau_{y,1}(r), r)$ to $(y, 1)$ is disjoint from $\pi_{y+\ep}|_{[r, 1]}$, and so since $(\tau_{y, 1},\tau_{y, 2}\big|_{[0,r']} \oplus \pi_{y+\ep}|_{[r', 1]})$ is an optimizer, $\tau_{y,1}\big|_{[r,1]}$ must also be a geodesic. Combining this with \eqref{E: rm-opt-eq-geod} gives
        condition $i)$. \qedhere
           \end{description}

\end{proof}

\section{Finite time networks IV-V} \label{Sec: IV-V}

In this section, we are again working on the almost sure set $\Omega$. Our aim is to prove a correspondence for network types \textup{IV}, \textup{Va}, and \textup{Vb} in terms of the gap sheet. This will complete the proof of Theorem \ref{T: fin-time-thm}.

\begin{lemma} \label{L: uses-l-or-r-bubble}
    On $\Omega$, the following holds for all $x,y\in \R$. Suppose that $(x,0;y,1)$ has network type \textup{IV} or \textup{Vb}. Then there exists a compact set $K_\eta = K_{\eta(x, y)} = [-\eta,0]\times[0,\eta]\subset \R^2$ such that for all $(-\ep, \de) \in K_\eta$
    \begin{align}
        \label{E: uses-r-or-l-bubble} \begin{split}
            \Ov(\pi_{x-\ep,y+\de,L}, \pi_{x,y,L}) &\supset [1/4,3/4] \quad \text{ or } \quad\Ov(\pi_{x-\ep,y+\de,L},  \pi_{x,y,R}) \supset [1/4,3/4],\quad \text{ and} \\
         \Ov(\pi_{x-\ep,y+\de,R}, \pi_{x,y,L}) &\supset [1/4,3/4] \quad \text{ or }\quad \Ov(\pi_{x-\ep,y+\de,R}, \pi_{x,y,R}) \supset [1/4,3/4]
       \end{split}
    \end{align}
    If $(x,0;y,1)$ has network type \textup{IV or Va}, then there exists a compact set $K_\eta' = [0,\eta']\times[-\eta',0]$ such that for $(\ep, -\de)\in K_\eta'$ the analogous statement to \eqref{E: uses-r-or-l-bubble} holds. 
\end{lemma}
Note that in the case of network type IV, \eqref{E: uses-r-or-l-bubble} holds for all $(\ep,\de) \in [-\eta,\eta]^2$ (note here $\ep$ and $\de$ can be positive or negative).

\begin{proof}
    We will just prove this for leftmost geodesics $\pi_{x-\ep,y+\de,L}$, as the proof for rightmost geodesics follows in the same way. It suffices to prove that \eqref{E: uses-r-or-l-bubble} holds only when $\ep, \delta > 0$, since the case when at least one of these parameters equals $0$ is covered by the overlap convergence of rightmost and leftmost optimizers on monotone sequences (\ref{P: ov-rm-opt}).
    
    Now, suppose \eqref{E: uses-r-or-l-bubble} was not true, meaning we can find positive sequences $\ep_n,\de_n \to 0^+$ such that
    \begin{align}
      \label{E: no-ov-between-1/4-3/4}  \Ov(\pi_{x-\ep_n,y+\de_n,L}, \pi_{x,y,L})\not\supset [1/4,3/4] \quad \text{ and } \quad \Ov(\pi_{x-\ep_n,y+\de_n,L}, \pi_{x,y,R})\not\supset [1/4,3/4].
    \end{align}
    By precompactness of geodesics in the overlap topology (\ref{P: precomp}), we can find a subsequence $n_i$ such that 
    \begin{align}
      \notag  \pi_{x-\ep_{n_i},y+\de_{n_i},L} \overset{overlap}{\to} \pi_{x,y},
    \end{align}
    where $\pi_{x,y}$ is a geodesic for $(x,0;y,1)$.
    In the case that $(x,0;y,1)$ has network type \textup{IV}, then $\pi_{x,y} = \pi_{x,y,L}$ or $\pi_{x,y,R}$, contradicting \eqref{E: no-ov-between-1/4-3/4}.

    For the case of network \textup{Vb}, then if $\pi_{x,y}$ is equal to $\pi_{x,y,L}$ or $\pi_{x,y,R}$, we again have a contradiction. Therefore we may restrict to the case when $\pi_{x,y}$ is the middle geodesic in the network (see Figure \ref{fig:typeVa-Bubble-create}). More precisely, we can write $\pi_{x,y}$ as:
    \begin{align}
      \nonumber \pi_{x,y} = \pi_{x,y,R}\big|_{[0,r]}\oplus \pi \oplus \pi_{x,y,L}\big|_{[r', 1]},
    \end{align}
where $p = (\pi_{x, y, R}(r), r), q = (\pi_{x, y, L}(r'), r')$ are the two branch points in the networks $\Gamma(x, 0; y, 1)$, and $\pi$ is the unique geodesic from $p$ to $q$. 
    Now, for large enough enough $i$ we have that
    \begin{align}
      \nonumber  \Ov(\pi_{x-\ep_{n_i}, y+\de_{n_i},L}, \pi_{x,y}) \supset[r,r'].
    \end{align}
    However, because $x-\ep_{n_k} < x$ and $\pi_{x,y}(r) > \pi_{x,y,L}(r)$, there must exist $\tilde r \in (0,r)$ such that $\pi_{x-\ep_{n_k}, y+\de_{n_k},L}(\tilde r) = \pi_{x,y,L}(\tilde r)$. However, we see now that 
    $$
        \pi_{x,y} \big|_{[\tilde r,r']} \quad \text{ and } \quad \pi_{x,y,L}\big|_{[\tilde r,r']}
    $$
    are two different geodesics between $(\pi_{x,y}(\tilde r),\tilde r)$ and $(\pi_{x,y}(r'),r')$, giving a bubble in the network for $\pi_{x-\ep_{n_k},y+\de_{n_k}}$ (impossible by \ref{P: no-bubbles}), see the left panel in Figure \ref{fig:typeVa-Bubble-create}. The proof for network \textup{Va} is analogous. 
\end{proof}

\begin{lemma} \label{L: network-V-zero-restriction}
    On $\Omega$, the following holds for all $x,y\in \R$. Suppose $(x,0;y,1)$ has network type $\textup{Va}$. Then there exists $\eta > 0$ such that for any $\ep, \de \in [0, \eta]$, the network $\Gamma(x -\ep, 0; y + \de, 1)$ has type \textup{I}, \textup{II}, or \textup{III}.
    The corresponding statement for $\Gamma(x +\ep, 0; y - \de, 1)$ holds for network type \textup{Vb}.
\end{lemma}
\begin{figure}
    \centering
        
    \begin{tikzpicture}[transform shape = false, scale = 3]

         \draw[line width = 0.5mm ] plot [smooth, tension=1.1] coordinates {(-1,0) (-0.6, 1)(-1,2)};
        \draw[line width = 0.5mm] plot [smooth, tension=1.1] coordinates {(-1,0) (-1.4, 1)(-1,2)};
        \draw[line width = 0.5mm] plot [smooth, tension=1.1] coordinates {(-0.65,0.65) (-1.34,1.4)};
        
      
         \draw[line width = 0.5mm, blue] plot [smooth, tension=1.1] coordinates {(-1.35,0) (-0.78,0.3)};
         \draw[line width = 0.5mm, blue] plot [smooth, tension=1.1] coordinates {(-0.78,0.3) (-0.7,0.475)(-0.65,0.65)};
        \draw[line width = 0.5mm, blue] plot [smooth, tension=1.1] coordinates {(-0.65,0.65) (-1.34,1.4)};
        \draw[line width = 0.5mm, blue] plot [smooth, tension=1.1] coordinates {(-1.34,1.4) (-1.29, 1.55)(-1.215,1.7)};
        
         \draw[line width = 0.5mm, blue] plot [smooth, tension=1.1] coordinates {(-0.65,2) (-1.215,1.7)};

         \draw[line width = 0.5mm, red, dashed] plot [smooth, tension=1.1] coordinates {(-1.1,0.17) (-1.35, 0.75) (-1.32,1.37)};
         \draw[line width = 0.5mm, red, dashed] plot [smooth, tension=1.1] coordinates {(-1.1,0.17)  (-0.79,0.33) };
         \draw[line width = 0.5mm, red, dashed] plot [smooth, tension=1.1] coordinates {(-0.79,0.33) (-0.67,0.65)};
         \draw[line width = 0.5mm, red, dashed] plot [smooth, tension=1.1] coordinates {(-0.68,0.63) (-1.32,1.33)};

         \filldraw[black] (-1,2) circle (1pt);
         \filldraw[black] (-1.35,0) circle (1pt);
         \filldraw[black] (-0.65,2) circle (1pt);
         \filldraw[black] (-1,0) circle (1pt);

          \node at (-1,-0.1)  {$x$};
          \node at (-1,2.1)  {$y$};
          \node at (-1.35,-0.1)  {$x-\ep_{n_k}$};
          \node at (-0.6,2.1)  {$y+\de_{n_k}$};

           \node at (-1.57,0.15)  {$\tilde r$};
           \node at (-1.57,0.66)  {$r$};
           \node at (-1.55,1.43)  {$r'$};


        \draw[line width = 0.2mm, dashed] plot [smooth, tension=1.1] coordinates { (-1.5,0.13)(-0.5,0.13)};
         \draw[line width = 0.2mm, dashed] plot [smooth, tension=1.1] coordinates { (-1.5,0.65)(-0.5,0.65)};
        \draw[line width = 0.2mm, dashed] plot [smooth, tension=1.1] coordinates { (-1.5,1.4)(-0.5,1.4)};
        

      \draw[line width = 0.5mm ] plot [smooth, tension=1.1] coordinates {(1,0) (1.4, 1)(1,2)};
      \draw[line width = 0.5mm] plot [smooth, tension=1.1] coordinates {(1,0) (0.6, 1)(1,2)};
      \draw[line width = 0.5mm] plot [smooth, tension=1.1] coordinates {(0.74,0.4) (1.37,1.3)};

      \draw[line width = 0.5mm, blue] plot [smooth, tension=1.1] coordinates {(0.65,0) (0.78,0.3)};
        \draw[line width = 0.5mm, blue] plot [smooth, tension=1.1] coordinates {(0.65,0) (1.22,0.3)};

        \draw[line width = 0.5mm, blue] plot [smooth, tension=1.1] coordinates {(1.35,2) (1.22,1.7)};
        \draw[line width = 0.5mm, blue] plot [smooth, tension=1.1] coordinates {(1.35,2) (0.78,1.7)};

        \draw[line width = 0.5mm, blue] plot [smooth, tension=1.1] coordinates {(0.78,0.3) (0.6,1) (0.78,1.7)};
        \draw[line width = 0.5mm, blue] plot [smooth, tension=1.1] coordinates {(1.22,0.3) (1.4,1) (1.22,1.7)};

        \draw[line width = 0.5mm, red, dashed] plot [smooth, tension=1.1] coordinates {(0.9,0.17)  (1.2,0.32)};
        \draw[line width = 0.5mm, red, dashed] plot [smooth, tension=1.1] coordinates {(1.2,0.32) (1.35, 0.75)(1.36, 1.23)};
        \draw[line width = 0.5mm, red, dashed] plot [smooth, tension=1.1] coordinates {(0.9,0.17)  (0.76, 0.4)};
        \draw[line width = 0.5mm, red, dashed] plot [smooth, tension=1.1] coordinates {(0.76, 0.4) (1.36, 1.23)};

        \filldraw[black] (1,2) circle (1pt);
        \filldraw[black] (0.65,0) circle (1pt);
        \filldraw[black] (1.35,2) circle (1pt);
        \filldraw[black] (1,0) circle (1pt);

        \node at (1,-0.1)  {$x$};
        \node at (1,2.1)  {$y$};
        \node at (0.65,-0.1)  {$x-\ep_{n_k}$};
        \node at (1.4,2.1)  {$y+\de_{n_k}$};

           \node at (1.57,0.15)  {$\tilde r$};
           \node at (1.57,0.4)  {$r$};
           \node at (1.55,1.33)  {$r'$};


        \draw[line width = 0.2mm, dashed] plot [smooth, tension=1.1] coordinates { (0.5,0.13)(1.5,0.13)};
         \draw[line width = 0.2mm, dashed] plot [smooth, tension=1.1] coordinates { (1.5,0.4)(0.5,0.4)};
        \draw[line width = 0.2mm, dashed] plot [smooth, tension=1.1] coordinates { (1.5,1.3)(0.5,1.3)};
        
    \end{tikzpicture}
    \caption{An illustration of the creation of two bubbles. The left pictures is from the proof of Lemma \ref{L: uses-l-or-r-bubble}, and the right from the proof of Lemma \ref{L: network-V-zero-restriction}. The blue lines are the geodesic(s) for $(x-\ep_{n_k}, y+\de_{n_k})$, and the black lines geodesics for $(x,0;y,1)$. The bubbles are indicated by the red dashed lines.}
    \label{fig:typeVa-Bubble-create}
\end{figure}
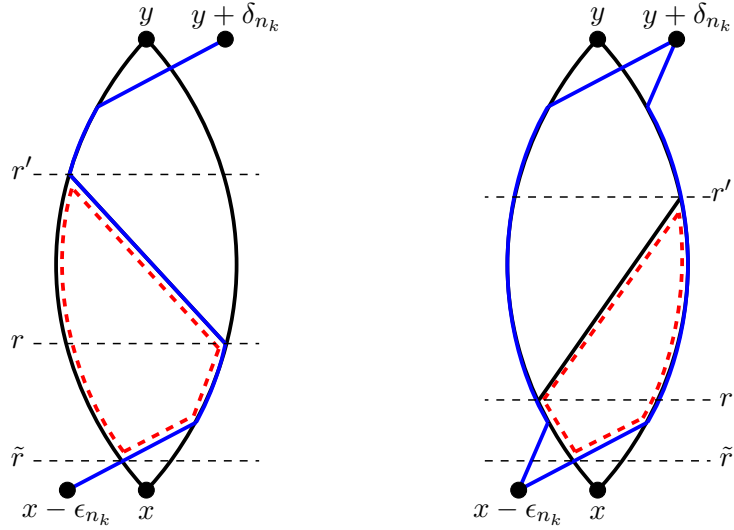
\begin{center}
\end{center}
\begin{proof}
We only prove the lemma for networks of type \textup{Va}, as the lemma for type \textup{Vb} follows by symmetry. Suppose the lemma is false. Then we can find sequences $\ep_n \downarrow 0, \delta_n \downarrow 0$ such that $\Ov(\pi_{x-\ep_n,y+\ep_n,L},\pi_{x-\ep_n,y+\ep_n,R}) =\emptyset$. By precompactness of the overlap topology (\ref{P: precomp}), by possibly passing to a subsequence we may assume that both sequences $\pi_{x-\ep_n,y+\ep_n,L}$ and $\pi_{x-\ep_n,y+\ep_n,R}$ converge in overlap to geodesics $\pi_{x,y,1} \leq \pi_{x,y,2}$ for $(x,0;y,1)$. 
These geodesics are disjoint, since convergence in overlap preserves disjointness. Looking at network type \textup{Va}, we see that $\pi_{x,y,1} = \pi_{x,y,L}$ and $\pi_{x,y,2} = \pi_{x,y,R}$. Moreover, we may also assume that either $\ep_n > 0$ for all $n$, or else $\delta_n > 0$ for all $n$. By symmetry, it suffices to consider the case when $\ep_n > 0$ for all $n$.

Let $\pi_{x,y}$ be the third geodesic for $(x,0;y,1)$, which (similarly to the previous proof) can be written as 
\begin{align}
      \nonumber  \pi_{x,y} = \pi_{x,y,L}\big|_{[0,r]}\oplus \pi \oplus \pi_{x,y,R}\big|_{[r', 1]},
    \end{align}
where $p = (\pi_{x, y, R}(r), r), q = (\pi_{x, y, L}(r'), r')$ are the two branch points in the networks $\Gamma(x, 0; y, 1)$, and $\tau$ is the unique geodesic from $p$ to $q$. 
    Now take $n$ large enough so that, for $\square \in\{L,R\}$,
    $$
        \Ov(\pi_{x-\ep_{n},y+\de_{n},\square}, \pi_{x,y,\square}) \supset [r,r'].
    $$
    Then we must have a time $\tilde r<r$ such that $\pi_{x-\ep_{n_k},y+\de_{n_k},R}(\tilde r) = \pi_{x,y,L}(\tilde r)$, because $\pi_{x-\ep_{n_k},y+\de_{n_k},R}(0) = x - \ep_n < x = \pi_{x,y,L}(0)$ and $\pi_{x-\ep_{n_k},y+\de_{n_k},R}(r) = \pi_{x,y,R}(r)> \pi_{x,y,L}(r)$. However, this means that both
    \begin{align*}
        \pi_{x,y,L}\big|_{[\tilde r,r]}\oplus\pi_{x,y}\big|_{[r,r']} \qquad \text{and} \qquad
        \pi_{x-\ep_{n_k},y+\de_{n_k},R}\big|_{[\tilde r,r]}\oplus \pi_{x,y,R}\big|_{[r,r']}
    \end{align*}
    are different geodesics starting and ending at the same points, creating a bubble. This contradicts \ref{P: no-bubbles}.
\end{proof}

Next, letting $K_\eta$ be as in Lemma \ref{L: uses-l-or-r-bubble}, define the following sets.
\begin{align*}
    \mathscr{L} &:= \{(-\ep, \de) \in K_\eta: \text{ There is a geodesic $\pi_{x-\ep,y+\de}$ where $[1/4,3/4] \subset \Ov(\pi_{x-\ep,y+\de}, \pi_{x,y,L})$}\}, \\
    \mathscr{R} &:= \{(-\ep, \de) \in  K_\eta: \text{ There is a geodesic $\pi_{x-\ep,y+\de}$ where $[1/4,3/4] \subset \Ov(\pi_{x-\ep,y+\de}, \pi_{x,y,R})$}\}.
\end{align*}
In both definitions, naturally, $\pi_{x-\ep, y+\de}$ is a geodesic for $(x-\ep,0;y+\de,1)$. 
We can similarly define $\mathscr{L'}, \mathscr{R'}$ where $(\ep, -\de) \in  K_\eta'$, as in Lemma \ref{L: uses-l-or-r-bubble}.

\begin{lemma} \label{L: L/R-geo-sets-properties} On $\Omega$, the following holds for all $x, y \in \R$. If $(x,0;y,1)$ has network type \textup{IV} or \textup{Vb}, the following hold for all $(-\ep, \de) \in K_\eta = [-\eta,0]\times [0,\eta].$
\begin{itemize}
    \item [1.]  $(-\ep,0) \in \mathscr{L} \backslash \mathscr{R}$ and $(0,\de) \in \mathscr{R}\backslash \mathscr{L}.$
    
    \item[2.] For small enough $\de$, there exists $-\ep_L^\de < -\ep_R^\de < 0$ such that $(-\ep_L^\de,\de)\in \mathscr{L}$ and $(-\ep_R^\de,\de) \in \mathscr{R}$. 

    \item[3.] Consider $-\ep_1, -\ep_2 \in [-\eta,0], \de_1, \de_2 \in [0,\eta]$ where $-\ep_1< -\ep_2, \de_1 < \de_2$. The following claims are true for $i\in \{1,2\}$:
    \begin{itemize}
        \item [i)] If $(-\ep_2, \de_i) \in \mathscr{L}$, then $(-\ep_1, \de_i) \in \mathscr{L}$.
        \item [ii)]  If $(-\ep_1, \de_i) \in \mathscr{R}$, then $(-\ep_2, \de_i) \in \mathscr{R}$.
        \item [iii)] If $(-\ep_i, \de_2) \in \mathscr{L}$, then $(-\ep_i, \de_1) \in \mathscr{L}$.
        \item [iv)]  If $(-\ep_i, \de_1) \in \mathscr{R}$, then $(-\ep_i, \de_2) \in \mathscr{R}$. 
    \end{itemize}
    \item[4.] The sets $\mathscr L$ and $\mathscr R$ are closed.
\end{itemize}
\end{lemma}
\begin{proof}\quad

    \textit{Proof of 1.} \qquad
         If $-\ep\in [-\eta,0)$, then, by Lemma \ref{L: uses-l-or-r-bubble}, $\pi_{x-\ep,y,R}(1/4)\in\{\pi_{x,y,L}(1/4), \pi_{x,y,R}(1/4)\}$, hence there must exist an $r \in [0,1/4]$ such that $\pi_{x-\ep,y,R}(r) = \pi_{x,y,L}(r)$. Now we see that
        $$
            \pi_{x-\ep,y,R}\big|_{[0,r]}\oplus\pi_{x,y,L}\big|_{[r,1]}
        $$
        must be a geodesic, therefore, $(-\ep,0)\in \mathscr{L}$. In fact, because $(x,0;y,1)$ has network type IV or Vb, we see that $\pi_{x,y,L}\big|_{[r,1]}$ is the only geodesic for $(\pi_{x,y,L}(r),r;,y,1)$. Therefore it must also be that $(-\ep,0) \notin \mathscr{R}$. The proof that $(0, \delta) \in \mathscr R \setminus \mathscr L$ is symmetric.\\
    
    \textit{Proof of 2.} \qquad
    First fix $\ep \in (0, \eta)$. Then for small enough $\de \in (0,\eta]$ we claim that $(-\ep,\de) \in \mathscr{L}$.  Indeed, if not then we could find a sequence $\de_n\to 0^+$ where $(-\ep, \de_n) \in \mathscr{R}$ for all $n$. This means there is a sequence of geodesics $\pi_{x-\ep,y+\de_n}$, such that 
    $$
    \Ov(\pi_{x-\ep,y+\de_n}, \pi_{x,y,R}) \supset [1/4,3/4].
    $$
    By taking an overlap-convergent subsequence of $\pi_{x-\ep,y+\de_n}$, by \ref{P: precomp}, we see that we must have a geodesic $\pi_{x-\ep,y}$ for $(x-\ep,0;y,1)$ such that 
    $$\Ov(\pi_{x-\ep,y}, \pi_{x,y,R}) \supset [1/4,3/4],$$
    meaning that $(-\ep,0)\in \mathscr{R}$, contradicting \textit{1.} 
    Taking $-\ep_L^\de:=-\ep$ gives us the first part of the claim. By repeating the analogous argument, we get that $(0,\de) \in \mathscr{R}$ and that there exists small $\ep_R^\de\in (0, \ep^\delta_L]$ such that $(-\ep_R^\de,\de)\in \mathscr{R}$, proving \textit{2.}\\

    \textit{Proof of 3.} \qquad
    We will only show \textit{i)} as the other parts have analogous arguments. Assume $(x-\ep_2,y+\de_i) \in \mathscr L$. By Lemma \ref{L: uses-l-or-r-bubble} and geodesic ordering (\ref{P: rmlmopt-mon}),
    $$
        \Ov(\pi_{x-\ep_2,y+\de_i, L}, \pi_{x,y,L}) \supset [1/4,3/4],
    $$
    and so $\pi_{x-\ep_2,y+\de_i, L} < \pi_{x-\ep_2,y+\de_i, R}$ on $[1/4, 3/4]$. Again using geodesic ordering, $\pi_{x-\ep_1,y+\de, L} \le \pi_{x-\ep_2,y+\de, L}$, and so $\pi_{x-\ep_1,y+\de_i, L} < \pi_{x-\ep_2,y+\de_i, R}$ on $[1/4, 3/4]$ as well. Since $(\ep_1, \delta_i) \in K_\eta$, Lemma \ref{L: uses-l-or-r-bubble} then forces 
    $$
        \Ov(\pi_{x-\ep_1,y+\de_i, L}, \pi_{x,y,L}) \supset [1/4,3/4],
    $$
    and so $(-\ep_1, \delta_i) \in \mathscr L$.

    \textit{Proof of 4.} Suppose we have a sequence of points $(-\ep_n, \delta_n) \in \mathscr L$, converging to a limit $(-\ep, \delta)$. Then there exists a sequence of geodesics $\pi_{x-\ep_n, y + \delta_n}$ with $[1/4,3/4] \subset \Ov(\pi_{x-\ep_n,y+\de_n}, \pi_{x,y,L})$ for all $n$. By precompactness (\ref{P: precomp}), this sequence has an overlap-convergent subsequence converging to a limiting geodesics $\pi_{x-\ep, y + \de}$ for $(x-\ep, 0; y+\de, 1)$. This geodesic must also satisfy $[1/4,3/4] \subset \Ov(\pi_{x-\ep,y+\de}, \pi_{x,y,L})$, and so $(x-\ep, y + \de) \in \mathscr L$. The proof that $\mathscr R$ is closed is analogous.
    \end{proof}

We are now ready to prove the main result of this section, from which we will quickly conclude the remainder of Theorem \ref{T: fin-time-thm}.

\begin{prop} \label{P:typeIV-has-close-zeroes} On $\Omega$, the following holds for all $x, y \in \R$.
\hspace{5mm}
    \begin{itemize}
        \item [1.] If $(x,s;y,t)$ has network type \textup{IV} or \textup{Vb}, then $G(x,y) = 0$ and there exists sequences $-\ep_n, \de_n \to 0$ where $-\ep_n < 0, \de_n > 0$, such that $G(x-\ep_n, y+\de_n) = 0$. 

        \item[2.] If $(x,s;y,t)$ has network type \textup{IV} or \textup{Va}, then $G(x,y) = 0$ and there exists sequences $\ep_n, -\de_n \to 0$ where $\ep_n > 0, -\de_n < 0$, such that $G(x+\ep_n, y-\de_n) = 0$.  
    \end{itemize}
\end{prop}

\begin{proof}
    We will only prove \textit{1.} as the proof of \textit{2.} is analogous. The fact that $G(x,y) = 0$ follows since networks of type \textup{IV} or \textup{Vb} contain two disjoint geodesics. The remainder of the proof is devoted to finding the sequences $\ep_n, \de_n$, which requires more work. 
    
   Let $\eta, K_\eta$ be as in Lemma \ref{L: uses-l-or-r-bubble}, and take $\de$ small enough so that Lemma \ref{L: L/R-geo-sets-properties}.2 holds. Define
    \begin{align}
        \label{e:eq_LR_def}
        -\ep^\de := \sup\{-\ep \in [-\eta, 0): (x-\ep, y+\de) \in \mathscr{L}\}.
    \end{align}
    We have $-\ep^\de > -\eta$ by Lemma \ref{L: L/R-geo-sets-properties}.2, and so by definition the point $(-\ep^\de, \de)$ is in the closure of $\mathscr{L}$. Hence $(-\ep^\de, \de) \in \mathscr L$ by Lemma \ref{L: L/R-geo-sets-properties}.4, and by Lemma \ref{L: L/R-geo-sets-properties}.1, we have $-\ep^\de < 0$. Using this with the definition \eqref{e:eq_LR_def} and the fact that $K_\eta \subset \mathscr L \cup \mathscr R$ then implies that $(-\ep^\de, \de)$ is also in the closure of $\mathscr R$, and so $(-\ep^\de, \de) \in \mathscr R$ by Lemma \ref{L: L/R-geo-sets-properties}.4.  Next, observe that since $(-\ep^\delta, \delta) \in \mathscr L \cap \mathscr R$, we have
    \begin{align}
       \nonumber - \ep^\de &\to 0^{-} \text{ as } \de \to 0^+.
    \end{align}
    Indeed, if not, then since $\mathscr L \cap \mathscr R$ is closed (Lemma \ref{L: L/R-geo-sets-properties}.4) it would contain a point of the form $(-\ep, 0)$ for some $\ep \in [-\eta, 0)$, contradicting Lemma \ref{L: L/R-geo-sets-properties}.1. Finally, we define
    $$
    \beta^\de = \inf \{\delta' \in (0, \delta] : (-\ep^\delta, \delta') \in \mathscr L \cap \mathscr R\}.
    $$
  Using Lemma \ref{L: L/R-geo-sets-properties}.4 again, we have that $(\ep^\delta, \beta^\delta) \in \mathscr L \cap \mathscr R$ and by Lemma \ref{L: L/R-geo-sets-properties}.1, $\beta^\delta \ne 0$. To complete the proof, we will show that  $G(x-\ep^\delta,y+\beta^\delta) = 0$. 

  If this is not the case, then since $(-\ep^\de, \beta^\de) \in \mathscr L \cap \mathscr R$ then there exists an $r \in (0, 1/4) \cup (3/4, 1)$ such that $\pi_{x-\ep^\de,y+\beta^\de,L}(r) = \pi_{x-\ep^\de,y+\beta^\de,R}(r)$. First suppose that this holds for some $r > 3/4$. Then by convergence in overlap of leftmost geodesics (\ref{P: ov-rm-opt}), we can find $\beta\in (0, \beta^\delta)$ with 
  $$
        \Ov(\pi_{x-\ep^\de,y+\beta,L}, \pi_{x-\ep^\de,y+\beta^\de,L}) \supset [0, r].
    $$
    Therefore the path $\pi_{x-\ep^\de,y+\beta^\de,R}|_{[0, r]} \oplus \pi_{x-\ep^\de,y+\beta,L}|_{[r, 1]}$ is another geodesic for $(x - \ep^\de, 0; y + \beta, 1)$, and so $(-\ep^\de, \beta) \in \mathscr L \cap \mathscr R$. This contradicts the definition of $\beta^\de$.

    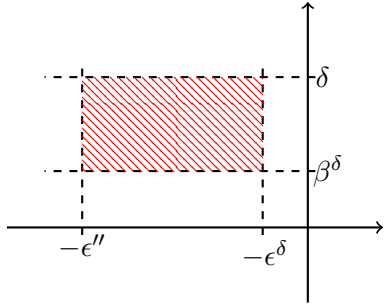
\begin{figure}
        \centering

                \begin{tikzpicture}
                      \draw[thick, ->] (-3, -1) -- (2, -1);

                      \draw[thick, ->] (1, -2) -- (1, 2);

                      \draw[thick, dashed] (1.1, 1) -- (-2.5, 1);
                      \draw[thick, dashed] (1.1, -0.25) -- (-2.5, -0.25);

                        \draw[thick, dashed] (0.4,-1.1) -- (0.4,1.2);
                        \draw[thick, dashed] (-2,-1.1) -- (-2,1.2);

                        \node at (1.2,1) {$\de$};
                        \node at (1.27,-0.25) {$\beta^\de$};
                        \node at (0.43,-1.35) {$-\ep^\de$};
                        \node at (-2,-1.25) {$-\ep''$};

                        \fill[pattern=north west lines, pattern color=red] (-2, -0.25) rectangle (0.4, 1);
                \end{tikzpicture}
        \caption{Illustration of the contradiction arrived at in the proof of proposition \ref{P:typeIV-has-close-zeroes}. The red region is entirely contained in $\mathscr{L} \cap \mathscr{R}$.}
        \label{F:de-ep-rat-point-pic}
    \end{figure}

 Therefore $\pi_{x-\ep^\de,y+\beta^\de,L}(r) = \pi_{x-\ep^\de,y+\beta^\de,R}(r)$ for some $r < 1/4$.  By convergence in overlap of leftmost geodesics (\ref{P: ov-rm-opt}), we can find an $-\ep'' < -\ep^\de$ such that 
    $$
        \Ov(\pi_{x-\ep'', y+\beta^\de,L}, \pi_{x-\ep^\de, y+\beta^\de,L}) \supset [r,1],
    $$
    implying that $(-\ep'', \beta^\de) \in \mathscr{L}\cap\mathscr{R}$ by the same reasoning as before.
    
    Now by Lemma \ref{L: L/R-geo-sets-properties}.3, for all $-\hat\ep \in [-\ep'', -\ep^\de]$, we  have that $(x-\hat\ep, y+\beta^\de) \in \mathscr{L} \cap \mathscr{R}$, Similarly, for all $\hat\de \in [\beta^\de, \de]$, $(x-\ep^\de, y+\hat\de) \in \mathscr{L}\cap\mathscr{R}$. Again by Lemma \ref{L: L/R-geo-sets-properties}.3, we can conclude that
    $$
    [-\ep'', -\ep^\de] \times [\beta^\de, \de] \subset \mathscr{L} \cap \mathscr{R},
    $$
    and so no geodesics in $[-\ep'', -\ep^\de] \times [\beta^\de, \de]$ have type \textup{I}. This contradicts uniqueness of rational geodesics, \ref{P: rmlmopt-mon}.
\end{proof}

The next theorem gives the second half of Theorem \ref{T: fin-time-thm}.
\begin{theorem} \label{T:if-close-zeroes-network-type}
    The following holds almost surely for all $(x,0;y,1)$. Suppose $G(x,0,y,1) = 0$. Then:
    \begin{itemize}
        \item [(i)] The point $(x,0,y,1)$ has network type \textup{IV} if and only if there exist sequences  $-\ep_n \to 0^{-}, \de_n \to 0^+, \ep_n' \to 0^+, -\de_n'\to 0^-$ such that $G(x-\ep_n, y+\de_n) = G(x+\ep_n', y-\de_n') = 0$.
        \item [(ii)] The point $(x,0,y,1)$ has network type \textup{Va} if and only if there exist sequences  $\ep_n \to 0^{+}, -\de_n \to 0^-$ such that $G(x+\ep_n, y-\de_n) = 0$ and there do not exist sequences $-\ep_n' \to 0^-, \de_n'\to 0^+$ such that $G(x-\ep_n, y+\de_n) = 0$.

        \item[(iii)] The point $(x,0,y,1)$ has network type \textup{Vb} if and only if there exist sequences  $-\ep_n' \to 0^{-}, \de_n' \to 0^+$ such that $G(x+\ep_n', y-\de_n') = 0$ and there do not exist sequences $\ep_n \to 0^-, -\de_n\to 0^+$ such that $G(x-\ep_n, y+\de_n) = 0$.

    \end{itemize}
\end{theorem}

\begin{proof}
The forward implications all follow from Proposition \ref{P:typeIV-has-close-zeroes}. The reverse implications follow from the forward implications, Lemma \ref{L: network-V-zero-restriction} and the fact that if $G(x, y) = 0$, then the network type must be IV, Va or Vb.\end{proof}

\begin{remark}
\label{R:comments-for-semi-inf}
We end by making a few comments about the proofs that will be used in extending to the semi-infinite case. First, while all our results are stated on $\Omega$ for geodesics from time $0$ to $1$, all statements hold on $\Omega$ for any rational start and end times, since this is the only property we used of $0, 1$. Second, in Proposition \ref{P:typeIV-has-close-zeroes} we identified sequences of leftmost and rightmost geodesics associated to the points $(x \pm \ep_n, y \pm \delta_n)$ which overlap with $\pi_{x, y, L}, \pi_{x, y, R}$ on the interval $[1/4, 3/4]$. Of course, the interval $[1/4, 3/4]$ was arbitrary and any interval $[a, b] \subset (0,1)$ with $a < b$ will produce the same result.
\end{remark}

\section{Proofs of applications} \label{Sec: applic-proofs}

In this section, we prove the corollaries stated in Section \ref{sec: applications}. Recall that $Z:= G^{-1}(0)$. We start with a simple lemma.
\begin{lemma}  \label{L: bowtie}
   On the almost sure event $\Omega$, if $(x,y) \in Z$, then there exists an $\ep > 0$ such that the bow-tie set $$
   \left([x-\ep,x]\times[y-\ep,y]\right) \cup \left([x,x+\ep]\times [y,y+\ep]\right)
   $$
    only intersects $Z$ at $(x,y)$.
\end{lemma}
\begin{proof}
    We will only prove this for $\left([x,x+\ep]\times [y,y+\ep]\right)$. If the statement were false, then we could find sequences $x_n \downarrow x, y_n \downarrow y$ with $(x_n, y_n) \in Z$. First consider the case where $y_n >y$ for all $y$ (the $x_n > x$ case is similar). Let $r \in (0, 1)$ be chosen large enough so that for $t \in (r, 1)$, all geodesics for $(x, 0, \pi_{x, y, R}(t), t)$ go through the point $(\pi_{x, y, R}(r), r)$. We can find such an $r$ by examining the three possible network types IV, Va, Vb.

    By \ref{P: ov-rm-opt}, for $y_n$ close enough to $y$ we have that $r \in \Ov(\pi_{x,y,R}, \pi_{x_n,y_n,R})$. Since $(x_n,y_n) \in Z$, $\pi_{x_n,y_n,L}(r) < \pi_{x_n,y_n,R}(r)$, which means that we can find times $s\in [0, r), t \in (r, 1)$ with $\pi_{x_n,y_n,L}(s) = \pi_{x,y,R}(s)$ and $\pi_{x_n,y_n,L}(t) = \pi_{x,y,R}(t)$. If $s = 0$, then this is contradiction to our definition of $r$. If $s > 0$, then this creates a geodesic bubble, contradicting \ref{P: no-bubbles}.
\end{proof}

\begin{proof}[Proof of Corollary \ref{Cor: network-density}]   \quad

    1. This follows since we can write $G_x(y) = \mathcal A_1(x) - \mathcal A_2(x)$, where $\mathcal A$ is an Airy line ensemble (see Theorem \ref{T: greenes-thm}. The pair of functions $(\mathcal A_1, \mathcal A_2) - (\mathcal A_1(0), \mathcal A_2(0)$, when restricted to any compact set is absolutely continuous with respect to a pair of independent Brownian motions, e.g. see \cite[Section 4.1]{corwin2014brownian}.

    2. This follows from Lemma \ref{L: bowtie}.
    
    3. This follows from Theorem \ref{T: fin-time-thm}.5, 6, 7.

    4. For $x,y\in \R$ and $\ep > 0$, define the set
     $$
     V_{x,y,\ep} :=[x-\ep,x]  \times [y-\ep,y+\ep].
    $$ 
    By Lemma \ref{L: bowtie} and Theorem \ref{T: fin-time-thm}.6, if $(x,0;y,1)$ has network type Va, there is an 
    $\ep_{x,y} > 0$ such that 
    $V_{x,y,\ep_{x, y}}$
    only intersects $Z$ at $(x,y)$. Now consider the collection $\{V_{x,y,\ep_{x, y}/2}: (x,y)\in Z\}$. Noting that this is a disjoint collection of sets of positive Lebesgue measure proves countability.
    
    Next, take $(x,0;y,1) \in \R^4_{\uparrow}$ with network type \textup{IV}, and for $\mu > 0$ consider the box $U :=[x-\mu, x+\mu] \times [y-\mu, y+\mu]$.

Let $p: \R^2 \to \R$ be the projection onto the first coordinate. Noting that $U\cap Z$ is a closed set (because it is the preimage of the closed set $\{0\}$ under the continuous function $G$), we see that $p(U\cap Z)\subset \R$ is compact. 

Now, we claim that $U_x :=p(U \cap Z)$ is nowhere dense. To see this, observe that $U_x$ does not contain rational points almost surely, since we can use that for any fixed $x$, the law of $G_x$ is given by the Airy gap $\mathcal A_1 - \mathcal A_2$ (Theorem \ref{T: greenes-thm}), which is strictly positive almost surely for all $y$. This and Corollary \ref{Cor: network-density}.3 tells us that we can find $\nu > 0, x_0 \in (x-\mu,x+\mu)$ such that 
$$
     \left((x_0 - \nu, x_0 + \nu)\times [y-\mu,y+\mu]\right) \cap (U\cap Z) = \emptyset,
$$
and 
$$
    \left(\{x_0 + \nu\} \times [y-\mu, y+\mu]\right) \cap (U\cap Z) \neq \emptyset.
$$
Take $(x',y')\in  \left(\{x_0 + \nu\} \times [y-\mu, y+\mu]\right) \cap (U \cap Z)$, then by Theorem \ref{T: fin-time-thm}.5, $(x',0;y',1)$ cannot have network type \textup{IV or Vb}, and hence must have type Va. Since $\mu > 0$ was arbitrary, this completes the proof.
\end{proof}

\begin{proof}[Proof of Corollary \ref{Cor: dec-fns}]   For every $n \in \N$, the set 
   $$
   A_n = \{\pi(1/2): \pi \text{ is a geodesic from time 0 to 1}, |\pi(0)| < n, |\pi(1)| < n \}
   $$
   is finite. This follows since the set of all geodesics above is compact in the overlap topology, \ref{P: ov-rm-opt}. This tells us that the collection of sets $\{Z_{a_1,a_2}: a_i \in A\}$ defined in Section \ref{sec: applications} is countable. Next, we show that each of these sets are closed. Take any converging sequence $u_n$ in $Z_{a_1,a_2}$ converging to $u \in \R^4_\uparrow$. By \ref{P: precomp}, we can find a subsequence such that $\pi_{u_{n_k},L}$ and $\pi_{u_{n_k},R}$ converge in overlap to disjoint geodesics for $u$. Noting the possible networks at a fixed time (Lemma \ref{L: fixed-time-possible-networks}), we must have that $\pi_{u_{n_k},L}$ and $\pi_{u_{n_k},R}$ are converging in overlap to $\pi_{u,L}$ and $\pi_{u,R}$ respectively (and these geodesics are disjoint). This gives that $\pi_{u,L}(1/2) = a_1$ and $\pi_{u,R}(1/2) = a_2$ and so $u \in Z_{a_1,a_2}$.   

Now, $Z_{a_1, a_2}\cap [-n, n]^2 = \emptyset$ if $\{a_1, a_2\} \not \subset A_n$, which implies the first bullet point since $A_n$ is finite.
For the second bullet point, observe that since the set $Z$ has no isolated points, and for any $n$, $Z_{a_1, a_2}\cap [-n, n]^2 = \emptyset$ for all but finitely many pairs $(a_1, a_2)$, that each of the sets $Z_{a_1, a_2}$ contains no isolated points. Therefore by the countability of networks of type $\textup{Va/b}$ (Corollary \ref{Cor: network-density}.4), we can write $Z_{a_1, a_2}$ as the closure of $Z_{a_1, a_2} \cap \textup{IV}$. Therefore to complete the proof, it suffices to show that if $u = (x_1, y_1), v = (x_2, y_2) \in Z_{a_1, a_2} \cap \textup{IV}$ and $x_1 \le x_2$, then in fact $x_1 < x_2$ and $y_1 > y_2$. We proceed by contradiction.

First suppose $x_1 = x_2$, and without loss of generality assume $y_1 < y_2$. By the nonexistence of $3$-stars (\ref{P: no-3-stars}), the leftmost and rightmost geodesics from $x_1$ to $y_i$ (denoted by $\pi_{y_i,L}, \pi_{y_i,R}$), there is an $s\in (0,1)$ such that 
    $$
        (\pi_{x_1,y_1,L}, \pi_{x_1, y_1,R})\big|_{[0, s]} = (\pi_{x_1, y_2,L}, \pi_{x_1, y_2,R})\big|_{[0, s]}.
    $$
     Using this and the fact that $\pi_{x_1, y_2,L}(1) > \pi_{x_1, y_1, R}(1)$, we must have an $s'\in (s, 1)$ such that $\pi_{x_1, y_2,L}(s') = \pi_{x_1, y_2, R}(s')$. Together this gives that $(x_1,0;y_1,1)$ must have network type \textup{Vb}, which is a contradiction. 
     
Now suppose $x_1 < x_2$. By a symmetric argument, we can rule out the case when $y_1 = y_2$, so we may assume $y_1 < y_2$. In this case, we can take rationals $w \in (x_1, x_2), z \in (y_1, y_2)$. There is a unique geodesic $\pi_{w, z}$ for $(w, 0; z, 1)$ by \ref{P: rmlmopt-mon}, so by geodesic ordering (again \ref{P: rmlmopt-mon}), we have
$$
\pi_{x_1, y_1, R}(1/2) \le \pi_{w, z}(1/2) \le \pi_{x_2, y_2, L}(1/2),
$$
contradicting that $\pi_{x_1, y_1, R}(1/2) = a_2> \pi_{x_2, y_2, L}(1/2) = a_1$.  
\end{proof}

    \section{Preliminaries in the semi-infinite setting}
\label{S:prelim-semi-infinite}
 In the remainder of the paper, we prove the results from Section \ref{SS:semi-infinite} on semi-infinite geodesics, semi-infinite optimizers, and the Busemann gap function. This section gives background on some these objects. 
    
    Recall that a path $\pi:[s,\infty) \to \R$ is a \textbf{semi-infinite geodesic} if, for any $t \geq s$, $\pi\big|_{[s,t]}$ is a geodesic. A semi-infinite geodesic has direction $\theta \in \R$ if
    $$
        \lim_{r\to\infty} \pi(r)/r = \theta.
    $$
    We say $\pi:[s, \infty) \to \R$ is a semi-infinite geodesic for $(x, s; \theta)$ if $\pi(s) = x$ and $\pi$ has direction $\theta$.
    Semi-infinite geodesics in the directed landscape were first studied in \cite{rahman2023infinitegeodesicscompetitioninterfaces}, with their behaviour being further clarified in \cite{busani2024stationary}, \cite{busani2024nonexistencenoncoalescinginfinitegeodesics}, \cite{bhatia2024dualitydirectedlandscapeapplications}, and \cite{busani2024nonexistencenoncoalescinginfinitegeodesics}. The next theorem gathers results on semi-infinite geodesics from these papers that we will need. For this theorem, we say that two semi-infinite geodesics $\pi:[s, \infty) \to \R, \tau:[t, \infty)\to \R$ \textbf{eventually coalesce} if there exists $r_0 \ge \max(s, t)$ such that $\pi(r) = \tau(r)$ for all $r \ge r_0$. If we choose $r_0$ minimally, we call $r_0$ the \textbf{coalescence time}.

    \begin{theorem} \label{T: semi-inf-geos}
        The following statements hold almost surely.
        \begin{itemize}
            \item [1.] (\cite{busani2024stationary} Theorem 6.5) For all $(x,s)\in \R^2, \theta \in \R,$ there are leftmost and rightmost semi-infinite geodesics for $(x,s;\theta)$, denoted by $\pi^{\theta}_{x,s,L}$ and $\pi_{x,s,R}^\theta$ respectively. For all other semi-infinite geodesics $\pi_{x,s}^\theta$ for $(x,s;\theta)$, we have
            $$
                \pi^\theta_{x,s,L}\leq \pi^\theta_{x,s} \leq \pi^\theta_{x,s,R}.
            $$
            \textup{We will usually only consider the geodesics with $s = 0$, in which case we simply drop the $s$ from the notation and denote them by $\pi_{x,L}^\theta$ and $\pi_{x,R}^\theta$.}
            \item[2.] (\cite{busani2024stationary} Theorem 7.1 (i)) For all $(x,s),(y,t)\in \R^2, \theta \in \R$, $\pi^\theta_{x,s,L}$ eventually coalesces with $\pi^\theta_{y,t,L}$ and $\pi^\theta_{x,s,R}$ eventually coalesces with $\pi^\theta_{y,t,R}$. 
            \item[3.] (\cite{busani2024stationary} Theorem 6.3 (i)) (Monotonicity) For all $s \in \R$, the functions $(\theta,x) \mapsto \pi^\theta_{x,s,L}$ and $(\theta,x)\mapsto \pi^\theta_{x,s,R}$ are non-decreasing in $\theta,s$. 
            \item[4.] (\cite{busani2024stationary} Theorem 7.3) There is a countably infinite dense subset of $\R$ denoted by $\Xi$ such that $\theta \notin \Xi$ if and only if, for all $(x, s) \in \R^2$, the geodesics $\pi^\theta_{(x,s),L}$ and $\pi^\theta_{(x,s),R}$ eventually coalesce. 
            \item[5.] (\cite{busani2024nonexistencenoncoalescinginfinitegeodesics} Theorem 1.5) For all $\theta \in \R$ and $(x, s) \in \R^2$, all semi-infinite geodesics in direction $\theta$ eventually coalesce with either $\pi_{x,s,L}^\theta$ or $\pi_{x,s,R}^\theta$.

            \item[6.] For $\theta \in \Xi$, we have $\pi^\theta_{(x,s),L}(r) < \pi^\theta_{(x,s),R}(r)$ for all large enough $r$.
        \end{itemize}
    \end{theorem}

   \begin{proof}[Proof of 6.] This result is essentially contained in \cite{busani2024stationary}, but is not stated explicitly there. By part $4$, if $\theta \in \Xi$ then exist two geodesics $\gamma, \pi$ in direction $\theta$ which do not eventually coalesce. Since there are no geodesic bubbles in $\L$ (Lemma \ref{L:no-bubbles-geo}) this means that $\gamma(r) \ne \pi(r)$ for all large enough $r$. By part $5$, for any $(x, s) \in \R$, both $\gamma$ and $\pi$ either eventually coalesce with $\pi^\theta_{(x, s), L}$ or $\pi^\theta_{(x, s), R}$. This can only happen if $\pi^\theta_{(x, s), L}(r) < \pi^\theta_{(x, s), R}(r)$ for all large enough $r$.    
   \end{proof} 
Now, while we cannot directly define distances to $\infty$ in the directed landscape, we can define \textit{relative distances} as we move to infinity along a particular slope, known as Busemann functions. For $(x, s; y, t) \in \R^2$ and $\theta \in \R$, define
$$
B^\theta(x, s; y, t) = \lim_{r \to \infty} \L(x, s; r \theta, r) - \L(y, t; r \theta, r).
$$
Almost surely, for any $\theta \notin \Xi$ and $(x, s; y, t) \in \R^4$, the above limit exists, see \cite[Theorem 5.1]{busani2024stationary}. Moreover, rather than taking the limit to $\infty$ along the exact line $(r \theta, \theta)$, we can follow any path whose asymptotic slope is $\theta$. In particular, we can replace $(r \theta, r)$ with $\pi(r)$ for some geodesic $\pi$ in direction $\theta$. This allows us the following alternate definition.
    \begin{definition}
        \label{D: busemann}
       For $\square \in \{L, R\}$ and $(x, s), (y, t) \in \R^2$, let $T_\square = T_\square(x, s; y, t)$ be the coalescence time of $\pi_{x, s,\square}^\theta$ and $\pi_{y, t,\square}^\theta$. Define
       $$B_\square^\theta(x, s; y, t) = \|\pi_{x, s,\square}^\theta|_{[s, T_\square]}\|_\L - \|\pi_{y, t,\square}^\theta|_{[t, T_\square]}\|_\L = \lim_{r \to \infty} \|\pi_{x, s,\square}^\theta|_{[s, r]}\|_\L - \|\pi_{y, t,\square}^\theta|_{[t,r]}\|_\L.
       $$
    \end{definition}
    Almost surely, for all $\theta \notin \Xi$ and $u = (x, s; y, t) \in \R^4$, we have $B(u) = B^\theta_L(u) = B^\theta_R(u)$ (again, this is part of \cite[Theorem 5.1]{busani2024stationary}). The processes $B_L, B_R$ are the left-continuous and right-continuous (in $\theta$) extensions of $B$ from $\R \setminus \Xi$ to all of $\R$, and are defined in this way in \cite{busani2024stationary}. The equivalent definition above is part of \cite[Theorem 6.5(i)]{busani2024stationary}. Moving forward, we write $B^\theta_\square(x) = B^\theta_\square(x, 0; 0,0)$ for $\square \in \{L, R\}$ since these are the only Busemann functions we require.

    We note a change in notation from \cite{busani2024stationary}, where the process was originally defined. Instead of indexing the process with $\{-,+\}$, we use $\{L,R\}$ to emphasize the connection to the left and right semi-infinite geodesics. 
    
  The processes $(x, \theta) \mapsto B^\theta_\square(x)$ can be viewed as the semi-infinite analogue of the Airy sheet. It is known as the \textbf{stationary horizon}, introduced in \cite{busani2024diffusive, seppalainen2023global}, and connected to the directed landscape in \cite{busani2024stationary}, see also \cite{busani2024scaling}. While the law of the Airy sheet has a somewhat complicated description, the law of the stationary horizon is relatively simple. Its marginals when $\theta$ is restricted to a finite set $F$ of size $n$ can be elegantly described through a collection of $n$ reflected Brownian motions. In the next theorem, we recall the description of this law when $n = 1, 2$, and the quadrangle inequality for the stationary horizon.

\begin{theorem}
    \label{T:stat-horizon}
    \begin{enumerate}
        \item (\cite{rahman2023infinitegeodesicscompetitioninterfaces}, Corollary 3.22) For any fixed $\theta \in \R$ and $\square \in \{L, R\}$, the process $x \mapsto B^\theta_\square(x)$ is a Brownian motion with diffusivity $\sqrt{2}$ and drift $2 \theta$.
        \item (Theorem 5.1(iii), \cite{busani2024stationary}) Almost surely, for any $\theta_1 < \theta_2$ and $x_1 < x_2$, we have
        $$
B^{\theta_1}_\square(x_2) - B^{\theta_1}_\square(x_1) \le B^{\theta_2}_\square(x_2) - B^{\theta_2}_\square(x_1).
        $$
    \item (\cite{busani2024diffusive}, Theorem 1.2.5) Almost surely, for all $\theta \in \R$ and all $a > 0$, there exists an $\ep > 0$ such that for all $0 < \delta \le \epsilon$ and $x \in [-a, a]$ we have
    $$
    B^{\theta-\de}_L(x) = B^\theta_L(x), \qquad    B^{\theta+\de}_R(x) = B^\theta_R(x).
    $$
    \end{enumerate}
\end{theorem}

We end this preliminary section with a simple anti-coalescence estimate.
    
\begin{lemma} \label{L: coal-semi-geo-bound}
    Almost surely, for any point $(p,s)$ in $\R\times \Q$, direction $\theta \in \R$ and $m \in \R$, there exists a $y < p$ such that $\pi^{\theta}_{p,s,L}(t) \neq \pi^{\theta}_{y,s,L}(t)$ for all $t \in [s, m]$. That is, $\pi^{\theta}_{p,s,L}, \pi^{\theta}_{y,s,L}$ coalesce after time $m$. A similar statement holds for rightmost semi-infinite geodesics.
\end{lemma}

\begin{proof}
We first prove the lemma when $(p, s) = (0,0)$. If the lemma failed for some $\theta \in \R, m > 0$, then for all $y < 0$,
   \begin{align*}
       B_L^\theta (y) = \L(y,0;\pi_{0,0,L}^\theta(m),m) - \L(0,0;\pi_{0,0,L}^\theta(m),m),
   \end{align*}
   since by assumption, $\pi_{y,L}^\theta$ coalesces with $\pi_{0,L}^\theta$ before time $m$. By Theorem \ref{T:stat-horizon}.1, for fixed $y$ the left-hand side above is a Brownian motion of drift $2 \theta$, so $B_L^\theta (y) = 2\theta y + o(y)$ as $y \to -\infty$. By the quadrangle inequality (Theorem \ref{T:stat-horizon}.2), these asymptotics holds almost surely for all $\theta \in \R$.
   
   On the other hand, by Lemma \ref{L: DL-bound}, for all $\theta \in \R$, the right-hand side above equals $-y^2/m + o(y^2)$ as $y \to -\infty$, which gives a contradiction. Therefore the lemma holds for $(p, s) = (0,0)$, for all $\theta, m$.

   We can extend this to all $(x,s)\in \Q^2$ by the translation invariance of $\L$. The general statement then follows by monotonicity of semi-infinite geodesics (Theorem \ref{T: semi-inf-geos}.3).
\end{proof}

\section{A correspondence for semi-infinite networks} \label{sec: semi-inf-cor}

\subsection{Possible semi-infinite network types}
\label{sec: poss-types}

In Section \ref{SS:geodesic-networks}, we described the network type for a point $(x,s;y,t)\in \R^4_\uparrow$. We extend this idea to semi-infinite geodesics in a direction $\theta$. We define the \textbf{semi-infinite geodesic network} for $(x,s;\theta)$ by
$$
    \Gamma^\infty(x,s;\theta) := \bigcup \{\mathfrak{g}\pi^\theta: \pi^\theta \text{ is a semi-infinite geodesic from $(x,s)$ in direction $\theta$}\}.
$$
After adding a vertex at $\infty$, we can define (semi-infinite) network types in the same way that we defined finite network types in Section \ref{SS:geodesic-networks}. More formally, if we let $f:\R^2 \to \R \times [-1, 1]$ be given by
$$
f(z, r) = \left(\frac{z}{\max(1, r)}, \frac{r}{|r| + 1}\right),
$$
we say that $\Gamma^\infty(x,s;\theta) \sim \Gamma^\infty(y, t;\theta')$ if the closures of  $f(\Gamma^\infty(x,s;\theta)), f(\Gamma^\infty(y, t;\theta'))$ are isomorphic as finite-time geodesic networks.

\begin{lemma}
\label{L:semi7}
    Almost surely, for all $(x,\theta) \in \R^2$, $(x, 0; \theta)$ has one of $7$ semi-infinite network types, pictured in Figure \ref{fig:semi-inf-networks}.
\end{lemma}
\begin{proof}
    For any $(y,q)\in \Gamma^\infty(x,0;\theta) \cap(\R\times \Q)$, $\Gamma(x,0;y,q) \subset \Gamma^\infty(x,0;\theta)$, and the network type of $(x,0;y,q)$ must be one of the seven networks pictured in Figure \ref{fig:fixed-time-networks} (Lemma \ref{L: fixed-time-possible-networks}). This, the nonexistence of 3-stars at fixed times (\ref{P: no-3-stars}), and the eventual coalescence of any geodesic from $(x, 0)$ with either the rightmost or leftmost geodesic (Theorem \ref{T: semi-inf-geos}.5) gives the result. 
\end{proof}
We ask a similar question to the one answered for the case of finite geodesic networks: How can we determine exactly when each of these networks occurs? In the finite case, Theorem \ref{T: fin-time-thm} gives an answer in terms of the gap sheet. In the semi-infinite case, the corresponding object to study is a \textbf{Busemann gap function}, to be defined in Section \ref{SS:generalized-busemann}. 

Certain features of the semi-infinite correspondence are essentially the same as in the finite correspondence. Because of this (and for brevity) we will not aim to give a complete classification of when each of the seven network types in Figure \ref{fig:semi-inf-networks} occurs. Instead, we will focus only on situations in the which the Busemann setting is qualitatively different from the finite setting. 

To set things up, first observe that by parts $4$ and $6$ of Theorem \ref{T: semi-inf-geos}, then $\theta \notin \Xi$ if and only if the network type of $(x, 0; \theta)$ is type $\text{I}^\infty$ or $\text{IIb}^\infty$. We are interested in studying the more interesting case when $\theta$ is in the exceptional set $\Xi$, where we will aim to differentiate between the cases $\text{IIa}^\infty,\text{III}^\infty, \text{IV}^\infty, \text{Va}^\infty,$ and $\text{Vb}^\infty$.

As in the finite setting, all argument will be deterministic on a set of probability $1$. Define $\Omega^\infty$ to be the probability one set where \ref{P: cont}-\ref{P: no-bubbles} hold, along with the following:
 \begin{enumerate}
[label=\textbf{P.\arabic*},ref=P.\arabic*]
    \setcounter{enumi}{6}
        \item \label{P: semi-inf-prop} \textit{Basic semi-infinite geodesic properties.} As in Theorem \ref{T: semi-inf-geos}. 
        \item \label{P: coal-semi-geo-bound}\textit{Coalescent scaling.}  As in Lemma \ref{L: coal-semi-geo-bound}.
    \end{enumerate}

For the remainder of Section \ref{sec: semi-inf-cor}, we work on $\Omega^\infty$.

\subsection{A generalized Busemann function}
\label{SS:generalized-busemann}

A \textbf{semi-infinite $2$-optimizer} is a function $\tau:[s, \infty) \to \R^2_\le$ such that for all $t \ge s$, $\tau|_{[s, t]}$ is a 2-optimizer. We say that $\tau:[s, \infty) \to \R$ is a $2$-optimizer for $(\x, s; \boldsymbol{\theta}) $ if $\tau(s) = \x$ and $\tau$ has direction $\boldsymbol{\theta} \in \R^2_\leq$. If $x_1 = x_2 =x$ and/or $\theta_1 = \theta_2 = \theta$ we say $\tau$ is a $2$-optimizer for $(x, s; \theta)$.

To show that the Busemann gap function from \eqref{E:Gtheta} is well-defined, we first show the existence of a specific kind of semi-infinite optimizer for directions in the exceptional set $\Xi$ where the optimizer eventually uses the leftmost rightmost semi-infinite geodesic.

\begin{prop} \label{P: opt-that-uses-disj-geos}
    On $\Omega^\infty$, for all $\theta \in \Xi, x \in \R$, there exists a semi-infinite 2-optimizer, $\tau_{x}^\theta$ for $(x, 0; \theta)$ and a time $t^* > 0$ such that $(\tau_{x,1}^\theta, \tau_{x,2}^\theta)\big|_{[t^*,\infty)} = (\pi^\theta_{x,L}, \pi_{x,R}^\theta)\big|_{[t^*,\infty)}$. Furthermore, for all semi-infinite 2-optimizers $\tau_{x}$ which satisfy $\tau_x\big|_{[t',\infty)} = (\pi^\theta_{x,L}, \pi_{x,R}^\theta)\big|_{[t',\infty)}$ for some $t' > 0$, we have that $\tau_x\big|_{[t^*,\infty)} = (\pi^\theta_{x,L}, \pi_{x,R}^\theta)\big|_{[t^*,\infty)}$. Finally, for any compact interval $[a, b]$ we can find a $t^*$ that works for all $x \in [a, b]$.
\end{prop}

\begin{proof}

    Since $\theta \in \Xi$, there is a time $\hat t$ such that $\pi^\theta_{x,L}(t) \neq \pi^\theta_{x,R}(t)$ for all $t > \hat t$. We assume $\hat t$ is the smallest such time. If $\hat t = 0$, meaning we are looking at a network of type $\textup{IV}^\infty, \textup{Va}^\infty$ or $\textup{Vb}^\infty$, then $(\pi_{x,L}^\infty, \pi_{x,R}^\infty)$ is such a semi-infinite 2-optimizer. So, going forward, assume $\hat t > 0$. By \ref{P: coal-semi-geo-bound}, we have a $y_1 < x$ and a $y_2 > x$ such that
    \begin{align}
       \label{E: sem-inf-opt-spaced-out} \pi^{\theta}_{y_1,L}(t) \neq \pi^{\theta}_{x,L}(t) \text{ and }  \pi^{\theta}_{y_2,R}(t) \neq \pi^{\theta}_{x,R}(t) \text{ for all } t\in [0,\hat t].
    \end{align}
     Now define the following:
    \begin{align*}
        t^* := \max \Big(\inf \{t \in [0,\infty): \pi^\theta_{y_1,L}(t) = \pi^\theta_{x,L}(t)\}  , \inf \{t \in [0,\infty): \pi^\theta_{y_2,R}(t) = \pi^\theta_{x,R}(t)\}        \Big).
    \end{align*}

    Because leftmost and rightmost semi-infinite geodesics coalesce in every direction, $t^* < \infty$ (\ref{P: semi-inf-prop}). Further, by $\eqref{E: sem-inf-opt-spaced-out}$, $t^* > \hat t$. Now we consider the following pair of paths:
    \begin{align*}
       \tau^{\theta}_{x} =  \tau_{\left(x^2,0; (\pi_{x,L}^\theta(t^*),\pi_{x,R}^\theta(t^*)), t^*\right),R} \oplus (\pi^\theta_{x,L}, \pi^{\theta}_{x,R})\big|_{[t^*,\infty)}.
    \end{align*}
    In words, $\tau_x^\theta$ follows the rightmost 2-optimizer to $(\pi^\theta_{x,L}(t^*), \pi^{\theta}_{x,R}(t^*)),t^*)$, then follows the respective semi-infinite geodesics after. We claim that $\tau^{\theta}_{x}$ is a semi-infinite 2-optimizer. It is clear that $\tau^{\theta}_{x}\big|_{[s,t]}$ is a 2-optimizer for $t \in [0, t^*]$, it remains to be shown for $t > t^*$. For this, it is enough to show that for all $t > t^*$, there is a $2$-optimizer $\tau$ from $(x^2, 0)$ to $(\pi^\theta_{x,L}(t), \pi^{\theta}_{x,R}(t)),t)$ with $\tau|_{[t^*, t]} = \tau^\theta_x|_{[t^*, t]}$.
    
    \begin{figure}
    \centering

\begin{tikzpicture}[transform shape = false, scale = 3]

    \draw[line width = 0.5mm, dashed, green] plot [smooth, tension=1.1] coordinates {(0,0) (0.4, 0.5)(0.565,1.1)};
    \draw[line width = 0.5mm, dashed, green] plot [smooth, tension=1.1] coordinates {(0,0) (-0.1, 0.6)(-0.31,1.3)};

    \draw[->, line width = 0.5mm] plot [smooth, tension=1.1] coordinates {(0,1) (0.35, 1.5)(0.5,2)};
    \draw[->, line width = 0.5mm] plot [smooth, tension=1.1] coordinates {(0,1) (-0.35, 1.5)(-0.5,2)};
    \draw[line width = 0.5mm] plot [smooth, tension=1.1] coordinates {(0,0) (0.05, 0.5)(0,1)};
    \filldraw[black] (0,0) circle (1pt);
    \node at (0,-0.15) {$x$};

    \draw[dashed, line width =0.3mm] (-1.3,0) -- (1.3,0);
    \node at (1.35,0) {$0$};
    \draw[dashed, line width =0.3mm] (-1.3,1) -- (1.3,1);
    \node at (1.35,1) {$\hat t$};
    \draw[dashed, line width =0.3mm] (-1.3,1.5) -- (1.3,1.5);
    \node at (1.35,1.5) {$t^*$};

    \draw[line width = 0.5mm] plot [smooth, tension=1.1] coordinates {(-1,0) (-0.5, 0.8)(-0.3,1.4)};
    \filldraw[black] (-1,0) circle (1pt);
    \node at (-1,-0.15) {$y_1$};

    \draw[line width = 0.5mm] plot [smooth, tension=1.1] coordinates {(1,0) (0.7, 0.8)(0.35,1.5)};
    \filldraw[black] (1,0) circle (1pt);
    \node at (1,-0.15) {$y_2$};

    \begin{scope}
        \clip(0,1.497) rectangle (2,2);
        \draw[->, line width = 0.5mm, green] plot [smooth, tension=1.1] coordinates {(0,1) (0.35, 1.5)(0.5,2)};
    \end{scope}

      \begin{scope}
        \clip(0,1.07) rectangle (2,2);
        \draw[line width = 0.5mm, green, trim left = 0.5] plot [smooth, tension=1.1, domain = 0:0.8] coordinates {(1,0) (0.7, 0.8)(0.35,1.5)};
    \end{scope}

    \begin{scope}
        \clip(0,1.26) rectangle (-2,2);
        
      \draw[line width = 0.5mm, green] plot [smooth, tension=1.1] coordinates {(-1,0) (-0.5, 0.8)(-0.3,1.4)};
    \end{scope}

    \begin{scope}
        \clip(0,1.4) rectangle (-2,2);
        \draw[->, line width = 0.5mm, green] plot [smooth, tension=1.1] coordinates {(0,1) (-0.35, 1.5)(-0.5,2)};
    \end{scope}

\end{tikzpicture}

    \caption{Illustration of the proof of Proposition \ref{P: opt-that-uses-disj-geos}. The green dashed lines indicate a 2-optimizer which eventually uses the leftmost and rightmost semi-infinite geodesics.}
    \label{fig:opt-use-semi-inf-geo-in-Xi}
\end{figure}
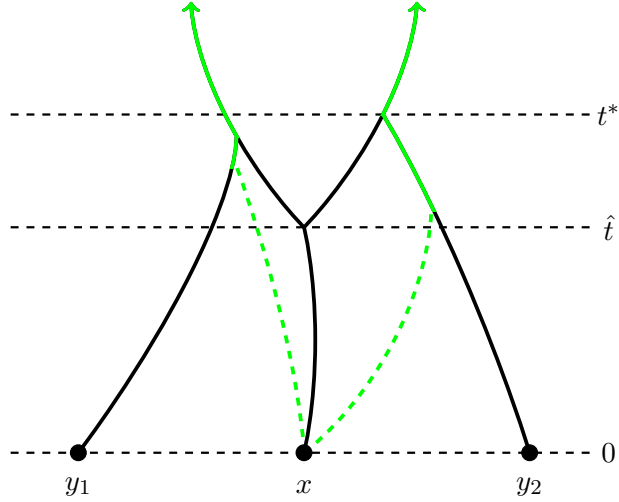

Let $\tau$ be any optimizer from $(x^2,0)$ to  $(\pi_{x,L}^\theta(t), \pi_{x,R}^\theta(t)),t)$. We start by showing there is a time $r_1 \in [\hat t,t^*]$ such that $\tau_1(r_1) \in \{ \pi_{y_1,L}^{\theta}(r_1),\pi_{x,L}^{\theta}(r_1)\}$. First, for $r < \hat t$ close enough to $\hat t$, $\pi^\theta_{x,L}(r) = \pi^{\theta}_{x,R}(r)$ (note that we cannot say this is true for all $r < \hat t$ because of network type $\textup{III}^\infty$). Now, using the same strategy as the proof of Lemma \ref{L: geo-opt-ordering}, we have that 
    \begin{align*}
       \pi^\theta_{y_1, L} (r) \leq \tau_1(r) \leq \pi^\theta_{x,L}(r).
    \end{align*}
    Using continuity, we get that
    \begin{align}
        \nonumber \pi^\theta_{y_1, L} (\hat t) &\leq \tau_1(\hat t) \leq \pi^\theta_{x,L}(\hat t).
    \end{align}
    Using this with the fact that $\pi^\theta_{y_1, L}(t^*) = \pi^\theta_{x,L}(t^*)$ gives the existence of such an $r_1$. 
    The same argument gives an $r_2 \in [\hat t,t^*]$ such that $\tau_2(r_2) \in \{\pi_{y_2,R}^{\theta}(r_2), \pi_{x,R}^{\theta}(r_2)\}$.
    \\\\
    Without loss of generality, assume that $\tau_1(r_1) = \pi_{y_1,L}^{\theta}(r_1)$ and $\tau_2(r_2) = \pi_{y_1,R}^{\theta}(r_2)$. This gives that:
    \begin{align}
    \label{E: cand-opt}\begin{split}
        \|\tau_1 \big|_{[0,r_1]} \oplus \pi^\theta_{y_1,L} \big|_{[r_1,t]}\|_\L &\geq \|\tau_1\|_\L, \\
         \|\tau_2 \big|_{[0,r_2]} \oplus \pi^\theta_{y_2,R} \big|_{[r_2,t]}\|_\L &\geq \|\tau_2\|_\L.
    \end{split}
    \end{align}
      We also see that $\tilde \tau =(\tau_1 \big|_{[0,r_1]} \oplus \pi^\theta_{y_1,L} \big|_{[r_1,t]}, \tau_2 \big|_{[0,r_2]} \oplus \pi^\theta_{y_2,R} \big|_{[r_2,t]})$ is a candidate 2-optimizer from $(x^2,0)$ to $\left((\pi^{\theta}_{x,L}(t^*),\pi^{\theta}_{x,R}(t^*)), t^*\right)$. Therefore $\tilde \tau$ is an optimizer and the inequalities \eqref{E: cand-opt} are actually equalities. Finally, the geodesics $\pi^\theta_{y_1,L} \big|_{[r_1,t]}, \pi^\theta_{y_1,R} \big|_{[r_1,t]}$ are the unique geodesics connecting up their endpoints by \ref{P: no-bubbles} (no geodesic bubbles), so the only way for the inequalities \eqref{E: cand-opt} to be equalities is if $\tau = \tilde \tau$. Observing that $\tau\big|_{[t^*, t]} = \tilde \tau\big|_{[t^*, t]} = \tau^\theta_x\big|_{[t^*, t]}$ then gives the result.
    
    For the `furthermore', observe that if $\tau'$ were another semi-infinite optimizer in direction $\theta$ satisfying $\tau'\big|_{[t',\infty)} = (\pi^\theta_{x,L}, \pi_{x,R}^\theta)\big|_{[t',\infty)}$, then for $t > t' \wedge t^*$, in the above proof we could take $\tau = \tau'\big|_{[0, t]}$. This gives that $\tau'\big|_{[t^*, t]} = \tau^\theta_x\big|_{[t^*, t]}$. Taking $t \to \infty$ gives the result.

    For the `finally', we can take $t^*_a$ to be the $t^*$ value for $x = a$, similarly for $t^*_b$. Then let $T$ be the earliest time such that 
    \begin{equation*}
         \pi_{a,\square}^\theta = \pi_{b,\square}^\theta,
    \end{equation*}
    for $\square \in \{L,R\}$. Taking 
    $$
        t^* = \max(t^*_a,t^*_b, T)
    $$
    finishes things by monotonicity and \ref{P: semi-inf-prop}. 
    
\end{proof}

The above proposition allows us to make a consistent choice of a semi-infinite optimizer for $\theta \in \Xi$.
Note that we could use the right-most optimizer, but a priori there is no guarantee that it eventually uses both sides of the non-unique geodesics (it should be the case, but we do not prove it). We can now show that the Busemann gap function is well-defined.

\begin{prop} \label{P: busemann-gap-form}
For $\theta \in \Xi$, recall from \eqref{E:Gtheta} that
\begin{align}
    \label{e: buseman-gap-function} G^{\theta}(x) &:= \lim_{t\to \infty} \L(x,0; \pi_{x,L}^\theta(t),t) +\L(x,0; \pi_{x,R}^\theta(t),t) - \L(x^2,0; (\pi_{x,L}^\theta(t), \pi_{x,R}^\theta(t)),t). 
\end{align}
Then on $\Omega^\infty$, for all $\theta\in \Xi, x\in\R$ the limit in \eqref{e: buseman-gap-function} exists, and there exists a $T_x > 0$ such that for any $t \ge T_x$:
    $$
        G^\theta(x) = \|\pi_{x,L}^\theta\big|_{[0,t]}\|_\L + \|\pi_{x,R}^\theta\big|_{[0,t]}\|_\L - \|\tau_{x}^\theta\big|_{[0,t]}\|_\L.
    $$
\end{prop}
\begin{proof}
    The limit existing follows by Proposition \ref{P: opt-that-uses-disj-geos}, and $T_x$ can be taken to be $t^*$ from that same proposition.
\end{proof}
The above form should be reminiscent of the our definition of the gap sheet earlier. Indeed, this form will allow us to apply Theorem \ref{T: fin-time-thm} to classify semi-infinite networks and conclude Theorem \ref{T: semi-inf-global-mins-close-by-intro}.

\subsection{The proof of Theorem \ref{T: semi-inf-global-mins-close-by-intro}}

We start with two easy lemmas, which essentially translate overlap properties of finite geodesics to the semi-infinite setting.

\begin{lemma} \label{L: rm-semi-inf-geo-coal}
    On $\Omega^\infty$, for all $(x,s)\in \R^2, \theta \in \R, \de > 0$, there exists an $\ep >0$ such that if $0<z - x < \ep$, $\Ov(\pi_{(z,s),R}^\theta, \pi_{(x,s),R}^\theta) \supset [s+\de,\infty)$. A similar statement holds for leftmost semi-infinite geodesics when $-\ep < z-x < 0$.
\end{lemma}
\begin{proof}
    To simplify notation in the proof, we write $\pi_z := \pi_{z,s,R}^\theta$. By \ref{P: semi-inf-prop} we know that $\pi_x$ coalesces with $\pi_{z'}$ eventually for any $z' > x$. That is, there is some $T > 0$ such that
    $$
        \Ov(\pi_x, \pi_{z'}) \supset [T,\infty).
    $$
    Next, by monotonicity, for all $z\in (x,z')$, we have that $\pi_x\big|_{[T, \infty)} = \pi_z\big|_{[T, \infty)} = \pi_{z'}\big|_{[T, \infty)}$. Therefore the lemma follows from convergence in overlap of the rightmost geodesics $\pi_z|_{[0, T]}$ to $\pi_x|_{[0, T]}$ (\ref{P: ov-rm-opt}).
\end{proof}

\begin{lemma} \label{L: excp-dir-coal-geos}
    On $\Omega^\infty$, for any $[a, b] \subset \R, \theta \in \Xi$, there exists a $T>0$ such that for any $x \in [a, b]$, and any semi-infinite geodesic  $\pi_{x}^\theta$ for $(x, 0; \theta)$, we have
    \begin{align*}
        \pi_{x}^\theta \big|_{[T,\infty)} \in \left\{\pi_{0,L}^\theta\big|_{[T,\infty)}, \pi_{0,R}^\theta\big|_{[T,\infty)} \right\}.
    \end{align*}
    Additionally, for large enough $M > T$, there exists an $N_M > T$ such that for any $x \in [a, b]$, $y\in [ \pi_{0,L}^\theta(M), \pi_{0,R}^\theta(M)]$, and any geodesic $\pi_{x,(y,M)}$ for $(x,0;y,M)$, we have that
    \begin{align}
        \label{E: y_M equal-one-of-geos}\pi_{x, (y,M)} \big|_{[T,N_M]} \in \left\{\pi_{0,L}^\theta\big|_{[T,N_M]},\pi_{0,R}^\theta\big|_{[T,N_M]} \right\}.
    \end{align}
    Moreover, $N_M \to \infty$ as $M\to \infty$.
\end{lemma}

\begin{proof}
When $[a,b] = \{x\}$ is a single point, the first statement follows by examining the possible semi-infinite geodesic networks (Lemma \ref{L:semi7}) and using the eventual coalescence of the rightmost and leftmost semi-infinite geodesics $\pi_{0,L}^\theta, \pi_{0,R}^\theta$ with $\pi_{x,L}^\theta, \pi_{x,R}^\theta$ (Theorem \ref{T: semi-inf-geos}.2). Let $T(x)$ denote the time found in this way.
For the case of general $[a, b]$, setting $T = \max(T(a), T(b))$, the claim follows from monotonicity of geodesics.

For the second statement (together with the `Moreover'), we proceed by contradiction. If the claim is false, then we can find a fixed time $t \ge T$ and sequences $x_n \in [a, b], M_n \to \infty, y_n \in [\pi_{x,L}^\theta(M_n), \pi_{x,R}^\theta(M_n)]$ such that for all $n \in \N$,
    \begin{equation}
    \label{E:contra-bontra}
        \pi_{x_n , (y_n, M_n)}(t) \notin \{\pi^\theta_{0, L}(t), \pi^\theta_{0, R}(t)\}.
    \end{equation}
    By geodesic ordering, $\pi^\theta_{a, L}\big|_{[0, M_n]} \le \pi_{x_n, (y_n, M_n)} \le \pi^\theta_{b, R}\big|_{[0, M_n]}$, so by precompactness of geodesics  in overlap (\ref{P: precomp}) we can find a subsequence $I \subset \N$ such that $\{\pi_{x_n, (y_n, M_n)}, n \in I\}$ converges to a limiting semi-infinite geodesic $\pi$, in the sense of overlap convergence on every compact interval. Here a
    standard diagonalization argument allows us to pass from precompactness in overlap of finite geodesics to precompactness of semi-infinite geodesics in a convergence-on-compact topology. The limit $\pi$ is a semi-infinite geodesic for $(x, 0; \theta)$ for some $x \in [a, b]$, but by \eqref{E:contra-bontra} and overlap convergence, $\pi(t) \notin \{\pi^\theta_{x, L}(t), \pi^\theta_{x, R}(t)\}$. This contradicts the first part of the lemma. 
\end{proof}

We now begin the proof of Theorem \ref{T: semi-inf-global-mins-close-by-intro}. The next two propositions prove parts $4$, $5$, and $6$. We return to prove parts $1$, $2$, and $3$ after.

\begin{prop} \label{P: semi-inf-global-mins-close-by}
   On $\Omega^\infty$, for all $x\in \R,\theta \in \Xi$:
    \begin{itemize}
        \item [1.] If $(x,0;\theta)$ has network type $\textup{IV}^\infty$ or $\textup{Va}^\infty$, then $G^\theta(x) = 0$ (is a global minimum), and there exists a positive sequence $\ep_n \to 0^+$ such that $G^\theta(x+\ep_n) = 0$ for all $n$.
         \item [2.] If $(x,0;\theta)$ has network type $\textup{IV}^\infty$ or $\textup{Vb}^\infty$, then $G^\theta(x) = 0$, and there exists a negative sequence $-\ep_n \to 0^-$ such that $G^\theta(x-\ep_n) = 0$ for all $n$. 
    \end{itemize}

\end{prop}

\begin{figure}
        \centering
        \begin{tikzpicture} [scale = 3]
                    \draw[->, line width = 0.5mm] plot [smooth, tension=1.1] coordinates {(0,0) (0.35, 0.8)(0.5,2)};
                    \draw[->, line width = 0.5mm] plot [smooth, tension=1.1] coordinates {(0,0) (-0.35, 0.8)(-0.5,2)};
                    \draw[line width=0.5mm, decoration={markings, mark=at position 0.5 with {\arrow{>}}}, postaction={decorate}] plot [smooth, tension=1.1] coordinates {(-0.29,0.6) (-0.1, 0.8) (0.2, 1) (0.43,1.2)};
                    \filldraw[black] (0,0) circle (1pt);
                    \node at (0,-0.15) {$x$};
                \draw[ line width = 0.5mm, blue] plot [smooth, tension=1.1] coordinates {(1,0) (0.6,0.25)(-0.39,1)};
                \draw[ line width = 0.5mm, blue] plot [smooth, tension=1.1] coordinates {(0.67,0.2) (0.5,0.6)(0.35,0.8)};

                \filldraw[black] (1,0) circle (1pt);
                \node at (1,-0.15) {$z$};

                \draw[dashed, line width = 0.3mm] plot [smooth, tension=1.1] coordinates {(-0.5,0.8) (0.5,0.8)};
                \node at (0.6,0.8) {$T_R$};

                \draw[dashed, line width = 0.3mm] plot [smooth, tension=1.1] coordinates {(-0.5,1) (0.5,1)};
                \node at (0.6,1) {$T_L$};
                
                \draw[dashed, line width = 0.3mm] plot [smooth, tension=1.1] coordinates {(-0.5,1.2) (0.5,1.2)};
                \node at (0.57,1.22) {$T^*$};

                \draw[ line width = 0.5mm, orange] plot [smooth, tension=1.1] coordinates {(-0.44,1.3) (-0.2, 1.7) (0,1.9)};
                \draw[line width = 0.5mm, orange] plot [smooth, tension=1.1] coordinates {(0.46,1.4) (0.2, 1.75) (0,1.9)};
                \node at (0,1.99) {$y_M$};
                \filldraw[black] (0,1.9) circle (1pt);


                \draw[ line width = 0.5mm, green] plot [smooth, tension=1.1] coordinates {(-0.2, 1.7) (-0.25,1.75) (-0.3,1.9)};
                 
                \draw[ line width = 0.5mm, green] plot [smooth, tension=1.1] coordinates {(0.23, 1.7) (0.1,1.77) (-0.3,1.9)};
                \node at (-0.3,2.12) {$y_M-\de_n$};
                \filldraw[black] (-0.3,1.9) circle (1pt);

                \draw[ line width = 0.5mm, green] plot [smooth, tension=1.1] coordinates {(0.2,0)(0,0.1)(-0.13,0.2)};
                \draw[ line width = 0.5mm, green] plot [smooth, tension=1.1] coordinates {(0.2,0)(0.18,0.1)(0.13,0.2)};

                \node at (0.3,-0.15) {$x+\ep_n$};
                \filldraw[black] (0.2,0) circle (1pt);

        \end{tikzpicture}   
        \caption{Illustrations of the proof of Proposition \ref{P: semi-inf-global-mins-close-by}. The leftmost and rightmost semi-infinite geodesics from $x+\ep_n$ follow the green lines until coalescing with the black semi-infinite geodesics from $x$.} 
        \label{fig:semi-inf-RL-typeIV}
    \end{figure}

\begin{proof}

    We prove the first case, as the argument for the second case is symmetric.  The fact that $G^\theta(x)= 0$ is clear, since using Proposition \ref{P: busemann-gap-form}, we have that 
    $$
        \|\pi_{x,L}^\theta \big|_{[0,T_x]}\| + \|\pi_{x,R}^\theta \big|_{[0,T_x]}\| - \|\tau_{x}^\theta \big|_{[0,T_x]}\| = 0.
    $$
    This is the case, because given that $(x,0;\theta)$ has network type $\textup{IV}^\infty$ or $\textup{Va}^\infty$, we have $\tau^\theta_x = (\pi_{x,L}^\theta, \pi_{x,R}^\theta)$. Now we take some $z > x$ and consider 
    \begin{align*}
        T_\square = \inf\{t > 0: \pi_{x,\square}^\theta(t) = \pi_{z,\square}^\theta(t)\},
    \end{align*}
    where $\square\in\{L,R\}$. This will be of use later when bounding leftmost and rightmost semi-infinite geodesics close to $x$. Using Lemma \ref{L: excp-dir-coal-geos}, take an $M\in \Q$ large enough to make $N_M > T^*:=\max(T_L,T_R,T)$ where $T$ is as in Lemma \ref{L: excp-dir-coal-geos}. Now consider
\begin{align}
\label{E:YM-def}
    y_M
        &= \inf\left\{y \in [\pi_{x,L}^\theta(M), \pi_{x,R}^\theta(M)]: \pi_{(x,0;y,M),R}\big|_{[T^*,N_M]} = \pi_{x,R}^\theta\big|_{[T^*, N_M]}\right\}.
\end{align}
First, note that the above infimum is actually a minimum, i.e.\ the rightmost geodesic from $(x, 0)$ to $(y_M, M)$ must use $\pi_{x,R}^\theta\big|_{[T^*, N_M]}$. If $y_M = \pi_{x,R}^\theta\big|_{[T^*, N_M]}$ then this is clear, and if not, then this follows from overlap convergence of rightmost geodesics. The fact that \eqref{E:YM-def} is a minimum also implies $y_M > \pi_{x,L}^\theta(M)$. Indeed, because $(x, 0; \theta)$ has network type $\textup{IV}^\infty$ or $\textup{Va}^\infty$, there is a unique geodesic for $(x, 0; \pi_{x,L}^\theta(M), M)$, and this geodesic simply follows $\pi^\theta_{x, L}$.

Next, by the construction of $y_M$ and \eqref{E: y_M equal-one-of-geos}, for any $y \in [\pi^\theta_{x, L}(M), y_M)$, any geodesic $\pi$ from $(x, 0)$ to $(y, M)$ must use $\pi^\theta_{0, L}|_{[T^*, N_M]}$. By overlap convergence of geodesics, this implies that:
\begin{itemize}[nosep]
    \item The leftmost geodesic $\pi_{(x,0;y_M,M),L}$ from $(x, 0)$ to $(y_M, M)$ uses $\pi^\theta_{0, L}|_{[T^*, N_M]}$.
    \item Any geodesic which equals $\pi_{(x,0;y_M,M),L}$ on an interval of the form $[M-\ep, M]$  for small $\ep > 0$ is disjoint from $\pi_{(x,0;y_M,M),R}$ on the interval $[T^*, M]$. Indeed, if not, then for some $y < y_M$, there is a geodesic $\pi$ from $(x, 0)$ to $(y, M)$ which equals $\pi_{(x,0;y_M,M),R}$ at some time $M - \de \in (N_M, M)$. The concatenation $\pi_{(x,0;y_M,M),R}|_{[0, M -\de]} \oplus \pi|_{[M - \de, M]}$ is then a geodesic, which equals $\pi^\theta_{x, R}$ on $[T^*, N_M]$ (since $y_M$ is a minimum). This in turn implies $y_M \le y$, which is a contradiction.
\end{itemize}
Combining these points with the fact that $(x, 0; \theta)$ has network type $\textup{IV}^\infty$ or $\textup{Va}^\infty$ and the construction of $T^*$ implies that $(x,0;y_M,M)$ must have type $\textup{IV}$ or $\textup{Va}$.

    Now, recall from Remark \ref{R:comments-for-semi-inf} that Proposition \ref{P:typeIV-has-close-zeroes}.2 also holds for the rational interval $[0, M]$, and an overlap interval $[a, b]$ used in place of $[1/4, 3/4]$ containing $T_L, T_R$. That is, we can find sequence $\ep_n \to 0^+, -\de_n\to 0^-$ such that $(x+\ep_n,0; y_M-\de_n,M)$ has network type $\textup{IV, Va, or Vb}$, and 
    \begin{align}
     \label{e: T_L,T_R-in-overlap}   \Ov(\pi_{(x+\ep_n,0;y_M-\de_n, M),\square}, \pi_{(x,0;y_M,M),\square}) \ni T_L,T_R,
    \end{align}
    where $\square \in \{L,R\}$.
    
    We now want to show that $G^\theta(x+\ep_n) = 0$ for $n$ large enough. It is sufficient to show that $(x+\ep_n,0;\theta)$ has two disjoint semi-infinite geodesics. By monotonicity of geodesics, if we take $n$ large enough so that $x+\ep_n < z$, we have that
    \begin{align*}
        \pi_{x+\ep_n,L}^\theta(T_L) &= \pi_{x,L}^\theta(T_L) \text{ and}\\
        \pi_{x+\ep_n,R}^\theta(T_R) &= \pi_{x,R}^\theta(T_R).
    \end{align*}
    On the other hand, by \eqref{e: T_L,T_R-in-overlap}, the leftmost geodesic from $(x+\ep_n, 0)$ to $(\pi_{x,L}^\theta(T_L) ,T_L)$ is equal to $\pi_{(x+\ep_n, 0;y_M-\de_n,M),L}$ restricted to the interval $[0, T_L]$. Combining this with the same assertion with $R$ in place of $L$, we have that 
    \begin{align*}
        \pi_{(x+\ep_n, 0;y_M-\de_n,M),L}\big|_{[0,T_L]} \oplus \pi_{x,L}^\theta\big|_{[T_L,\infty)} \text{ and } \pi_{(x+\ep_n, 0;y_M-\de_n,M),R}\big|_{[0,T_R]} \oplus \pi_{x,R}^\theta\big|_{[T_R,\infty)}
    \end{align*}
    are the leftmost and rightmost semi-infinite geodesics for $(x+\ep_n,0;\theta)$, and do not overlap. See figure \ref{fig:semi-inf-RL-typeIV} for a picture.
\end{proof}

\begin{prop} \label{P: no-close-mins-to-bridge}
    On $\Omega^\infty$, for all $x\in \R, \theta\in \Xi$:
    \begin{itemize}
        \item [1.] If $(x,0;\theta)$ has network type $\textup{Va}^\infty$, then there does not exist a negative sequence $-\ep_n \to 0^-$ such that $(x-\ep_n,0;\theta)$ has network type $\textup{IV}^\infty,\textup{Va}^\infty$, or $\textup{Vb}^\infty$.
        \item [2.] If $(x,0;\theta)$ has network type $\textup{Vb}^\infty$, then there does not exist a positive sequence $\ep_n \to 0^+$ such that $(x+\ep_n,0;\theta)$ has network type $\textup{IV}^\infty,\textup{Va}^\infty$, or $\textup{Vb}^\infty$.
    \end{itemize}
\end{prop}
\begin{proof}
    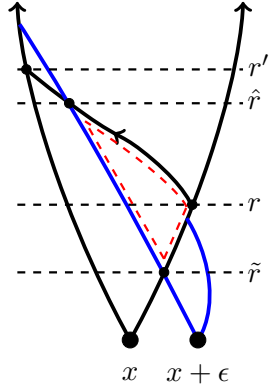
\begin{figure}
        \centering
        \begin{tikzpicture} [scale = 3]
                    \draw[->, line width = 0.5mm] plot [smooth, tension=1.1] coordinates {(0,0) (0.35, 0.8)(0.5,1.5)};
                    \draw[->, line width = 0.5mm] plot [smooth, tension=1.1] coordinates {(0,0) (-0.35, 0.8)(-0.5,1.5)};
                    \draw[line width=0.5mm, decoration={markings, mark=at position 0.5 with {\arrow{>}}}, postaction={decorate}] plot [smooth, tension=1.1] coordinates {(0.27,0.6) (0.1, 0.8) (-0.2, 1) (-0.46,1.2)};
                    \filldraw[black] (0,0) circle (1pt);
                    \node at (0,-0.15) {$x$};
                \draw[ line width = 0.5mm, blue] plot [smooth, tension=1.1] coordinates {(0.3,0) (0.35,0.25)(0.25,0.54)};
                \draw[ line width = 0.5mm, blue] plot [smooth, tension=1.1] coordinates {(0.3,0) (-0.1, 0.75)(-0.49,1.4)};

                \filldraw[black] (0.3,0) circle (1pt);
                \node at (0.3,-0.15) {$x+\ep$};

                \draw[dashed, line width = 0.3mm] plot [smooth, tension=1.1] coordinates {(-0.5,0.6) (0.5,0.6)};
                \node at (0.55,0.6) {$r$};
                \draw[dashed, line width = 0.3mm] plot [smooth, tension=1.1] coordinates {(-0.5,0.3) (0.5,0.3)};
                \node at (0.55,0.3) {$\tilde r$};
                \draw[dashed, line width = 0.3mm] plot [smooth, tension=1.1] coordinates {(-0.5,1.05) (0.5,1.05)};
                \node at (0.55,1.07) {$\hat r$};
                \draw[dashed, line width = 0.3mm] plot [smooth, tension=1.1] coordinates {(-0.5,1.2) (0.5,1.2)};
                \node at (0.57,1.22) {$r'$};

                \filldraw[black] (-0.27,1.05) circle (0.6pt);
                \filldraw[black] (-0.46,1.2) circle (0.6pt);
                \filldraw[black] (0.15,0.3) circle (0.6pt);
                \filldraw[black] (0.275,0.6) circle (0.6pt);

                \draw[dashed, line width = 0.3mm, red] plot [smooth, tension=1.1] coordinates {(0.15,0.36) (0.25, 0.59)};
                \draw[dashed, line width = 0.3mm, red] plot [smooth, tension=1.1] coordinates {(0.25, 0.59) (0.1,0.75)(-0.2,0.97)};
                \draw[dashed, line width = 0.3mm, red] plot [smooth, tension=1.1] coordinates {(0.15,0.36) (-0.2,0.97)};
        \end{tikzpicture}   
        \caption{Illustration of the proof of Proposition \ref{P: no-close-mins-to-bridge}. The bubble is outlined in red}
        \label{fig:97}
    \end{figure}
We only prove part $2$ as part $1$ is done in a similar way. See Figure \ref{fig:97} for a picture. Let $[r,r']$ be the time interval of the bridge between the leftmost and rightmost semi-infinite geodesics for $(x,0;\theta)$. That is, letting $\pi_{x}^\theta$ be the semi-infinite geodesic which uses this middle bridge, we have
    \begin{align*}
         \pi_{x}^\theta = \pi_{x,R}^\theta \big|_{[0,r]}&\oplus \pi_{(\pi_{x,R}^\theta(r),r;\pi_{x,L}^\theta(r'),r')}\oplus \pi_{x,L}^\theta \big|_{[r',\infty)} \text{, and } \\
         \pi_{x,L}^\theta \big|_{(r,r')} &< \pi_{x}^\theta \big|_{(r,r')} <\pi_{x,R}^\theta \big|_{(r,r')}.
    \end{align*}

    By Lemma \ref{L: rm-semi-inf-geo-coal}, take $\ep > 0$ small enough so that $\Ov(\pi_{x,R}, \pi_{x+\ep,R}) \supset [r,\infty)$. Now, if $(x+\ep,0;\theta)$ had network type $\textup{IV}^\infty,\textup{Va}^\infty$, or $\textup{Vb}^\infty$, it follows, along with monotonicity that
    \begin{align*}
        \pi_{x,L}^\theta(r) &\leq \pi_{x+\ep,L}^\theta(r) < \pi_{x+\ep,R}^\theta(r) = \pi_{x}^\theta (r), \text{ and}\\
         \pi_{x}^\theta (r') = \pi_{x,L}^\theta(r') &\leq \pi_{x+\ep,L}^\theta(r') < \pi_{x+\ep,R}^\theta(r').
    \end{align*}
    By continuity, we get an $\hat r \in (r,r']$ such that $\pi_{x+\ep,L}^\theta (\hat r) = \pi_{x}^\theta (\hat r)$. Furthermore, since $\pi_{x+\ep,L}^\theta (0) > \pi_{x,R}^\theta(0)$ and $\pi_{x+\ep,L}^\theta(r) < \pi_{x+\ep,R}^\theta(r) = \pi_{x,R}^\theta(r)$, we have an $\tilde r \in (0,r)$ such that $\pi_{x+\ep,L}^\theta (\tilde r) = \pi_{x,R}^\theta(\tilde r) = \pi_{x}^\theta(\tilde r)$. Putting these together, we have two geodesics $\pi_{x+\ep,L}^\theta \big|_{[\tilde r, \hat r]}$ and $\pi_{x}^\theta \big|_{[\tilde r, \hat r]}$ for $(\pi_{x+\ep}^\theta(\tilde r), r; \pi_{x+\ep}^\theta(\hat r),r)$ where $\pi_{x+\ep,L}^\theta (r) < \pi_{x}^\theta (r)$. This forms a bubble, which is a contradiction by \ref{P: no-bubbles}.
\end{proof}
We prove the remainder of Theorem \ref{T: semi-inf-global-mins-close-by-intro}, which ends up being a bit simpler than the later parts.

\begin{proof}[Proof Of Theorem \ref{T: semi-inf-global-mins-close-by-intro} Parts $1$-$3$.]
The first part is obvious. We will prove $2.$ and $3.$ together by reducing to the finite case in a similar way to the proof of Proposition \ref{P: semi-inf-global-mins-close-by}. Consider a small rational interval $(a,b) \ni x$ centered around $x$. Let $t^*$ be as in Proposition \ref{P: opt-that-uses-disj-geos} for the interval $(a,b)$. By \ref{P: coal-semi-geo-bound}, take rational $y_1 < x < y_2$ such that 
\begin{align}
    \nonumber \pi_{y_1,L}^\theta(t) < \pi_{a,L}^\theta (t)\text{ and }  \pi_{b,R}^\theta(t) < \pi_{y_2,R}^\theta(t) \text{ for all $t\in [0,t^*]$}.
\end{align}
Now define 
$$
    T^* = \max(\inf\{t > t^*: \pi_{y_1,L}^\theta\big|_{[t,\infty)} = \pi_{x,L}^\theta\big|_{[t,\infty)} \text{ and } \pi_{y_2,R}^\theta\big|_{[t,\infty)} = \pi_{x,R}^\theta\big|_{[t,\infty)}\}, T),
$$
where $T$ is as in Lemma \ref{L: excp-dir-coal-geos}.
Such a $T^*$ exists by \ref{P: semi-inf-prop}. Note also that for any $z\in (a,b)$, $\pi_{z, \square}^\theta$ has coalesced with $\pi_{x,\square}^\theta$ by time $T^*$ for $\square \in \{L,R\}$. Now, using Lemma \ref{L: excp-dir-coal-geos}, take $M \in \Q$ large enough so that $N_M > T^*$, and define $y_M$ exactly as in \eqref{E:YM-def}. By a similar argument as was done in the proof of Proposition \ref{P: semi-inf-global-mins-close-by}, we see that $(x,0;y_M,M)$ has network type IIa or III and both $(a, 0; y_M, M)$ and $(b, 0; y_M, M)$ have type IIa. Precisely, we have that 
\begin{align}
    \label{E: fin-semi-inf-corr}\begin{split}
        (x,0;y_M,M) \text{ has network type IIa $\iff$} (x,0;\theta)\text{ has network type $\textup{IIa}^\infty$}, \\
         (x,0;y_M,M) \text{ has network type III $\iff$} (x,0;\theta)\text{ has network type $\textup{III}^\infty$}.
    \end{split}
\end{align}
Next, for any $z \in (a, b)$ we claim that 
\begin{align}
    \tau := \tau_z^\theta\big|_{[0,T^*]} \oplus(\pi_{(x,0;y_M,M),L}\big|_{[T^*,M]}, \pi_{(x,0;y_M,M),R}\big|_{[T^*,M]}) \label{e: potential-opt}
\end{align}
is a 2-optimizer for $(z,0;y_M,M)$, where $\tau_z^\theta$ is the $2$-optimizer from $(z, 0)$ in direction $\theta$ identified by Proposition \ref{P: opt-that-uses-disj-geos}. Noting that $\tau$ is an optimizer when restricted to both of the intervals $[0, T^*]$ and $[T^*, M]$, it suffices to show that there is a $2$-optimizer for $(z,0;y_M,M)$ which goes through the point $\tau(T^*) = (\pi^\theta_{y_1, L}(T^*), \pi^\theta_{y_2, R}(T^*))$ at time $T^*$. For this, observe that by geodesic monotonicity, all leftmost geodesics from the set $[y_1, b] \times \{0\}$ to the point $(y_M, M)$ go through the point $\pi^\theta_{y_1, L}(T^*)$ at time $T^*$ and all rightmost geodesics from the set $[a, y_2] \times \{0\}$ to the point $(y_M, M)$ go through the point $\pi^\theta_{y_1, R}(T^*)$ at time $T^*$. In particular, both of the geodesic pairs
$$
\tau^1 := (\pi_{(y_1, 0; y_M, M), L}, \pi_{(a, 0; y_M, M), R}), \qquad \tau^2:=(\pi_{(b, 0; y_M, M), L}, \pi_{(y_2, 0; y_M, M), R})
$$
go through $\tau(T^*)$ at time $T^*$, and are disjoint by construction and hence optimizers. Moreover, $\tau^1$ must be a leftmost optimizer and $\tau^2$ a rightmost optimizer since the geodesics $\pi_{(a, 0; y_M, M), L}, \pi_{(y_1, 0; y_M, M), L}$ overlap, as do the geodesics $\pi_{(b, 0; y_M, M), R}, \pi_{(y_2, 0; y_M, M), R}$. Therefore we may apply monotonicity of optimizers to conclude that any optimizer for $(z,0;y_M,M)$ goes through the point $\tau(T^*)$ at time $T^*$.

Now, letting $G$ be the gap sheet between times $0$ and $M$, we have, by \eqref{e: potential-opt} that, for all $z\in (a,b)$
\begin{align*}
    G(z,y_M) &= \|\pi_{(z,0;y_M,M),L}\|_\L +  \|\pi_{(z,0;y_M,M),R}\|_\L  -\L(z^2,0;y_M^2,M) \\
    &=\|\pi_{(z,0;y_M,M),L}\big|_{[0,T^*]}\|_\L +  \|\pi_{(z,0;y_M,M),R}\big|_{[0,T^*]}\|_\L - \|\tau_z\big|_{[0,T^*]}\|_\L \\
    & = \|\pi_{z,L}^\theta\big|_{[0,T^*]}\|_\L +  \|\pi_{z,R}^\theta\big|_{[0,T^*]}\|_\L - \|\tau_z^\theta\big|_{[0,T^*]}\|_\L \\
    &= G^\theta(z).
\end{align*}
The third line follows by \eqref{E: y_M equal-one-of-geos} and the definition of $y_M$. Parts $2.$ and $3.$ then follow by the above, \eqref{E: fin-semi-inf-corr}, and Theorem \ref{T: fin-time-thm}.2/4
\end{proof}

\section{The Busemann gap function} \label{sec: buse-gap-prop}

This section will look more closely at the function $G^\theta$ defined in the previous section, and culminate in the proof of Theorem \ref{T: reflected-bm}. We first rewrite $G^\theta$ in terms of Busemann functions, justifying its name.

First, for $\theta \in \Xi$, define the \textit{two-path Busemann function} 
\begin{align*}
    \B^\theta(x) := \lim_{t\to \infty} \|\tau^\theta_x\big|_{[0,t]} \|_\L - \|\pi^\theta_{0, L}\big|_{[0,t]}\|_\L - \|\pi^\theta_{0, R}\big|_{[0,t]}\|_\L, 
\end{align*}
where $\tau_x^\theta$ is the any semi-infinite 2-optimizer in direction $\theta$ emanating from $(x,0)$ which eventually coalesces with the leftmost and rightmost geodesics in direction $\theta$. This function exists and is well defined by Proposition $\ref{P: opt-that-uses-disj-geos}$, and the fact that leftmost and rightmost semi-infinite geodesics coalesce for all directions (\ref{P: semi-inf-prop}). 

Next, define the following process for $\theta \in \Xi, x\in \R$:
\begin{align*}
    \mathfrak{G}^\theta(x) := B_L^\theta(x) + B_R^\theta(x) - \B^\theta(x).
\end{align*}

\begin{lemma} \label{l: buse-gap-length-formula}
    On $\Omega^\infty$, for all $x \in \R, \theta \in \Xi$, we have that
    $
    \mathfrak{G}^\theta(x) =  G^\theta(x).
    $
\end{lemma}

\begin{proof}
By the eventual coalescence of leftmost and rightmost semi-infinite geodesics (\ref{P: semi-inf-prop}), we have that for large enough $t$:
\begin{equation}
\label{E:eventual-coal}
\begin{split}
      B^\theta_\square(x) &= \|\pi^\theta_{x,\square}|_{[0, t]}\|_\L - \|\pi^\theta_{0,\square}|_{[0, t]}\|_\L, \;\square \in \{L, R\};\\
    \quad\mathcal{B}^\theta(x) &= \|\tau^\theta_x\big|_{[0,t]} \|_\L - \|\pi^\theta_{0, L}\big|_{[0,t]}\|_\L - \|\pi^\theta_{0, R}\big|_{[0,t]}\|_\L.
\end{split}
\end{equation}

The formula then follows from Proposition \ref{P: busemann-gap-form}.
\end{proof}

Our current Busemann function description for $G^\theta$ does not shed any light on its exact distribution. For this, we need to connect it to two-path Busemann functions in \textit{distinct} directions, whose structure was elucidated in \cite{dauvergne2024directedlandscapebrownianmotion}. 

\begin{lemma} \label{L: multi-theta-lem}
Let $\mathcal{I}_<^2 := \{(\theta_1,\theta_2) \in (\R \backslash \Xi)^2: \theta_1 < \theta_2\}$.
    \begin{enumerate}
        \item \textup{(\cite{dauvergne2024directedlandscapebrownianmotion},  Proposition 4.3.2/3)}.   Almost surely, for all $(\theta_1,\theta_2)\in \mathcal{I}^2_<$ and every $x \in \R$, there exists a rightmost semi-infinite $2$-optimizer, $\tau_{x,R}^{\theta_1,\theta_2},$ emanating from $(x,0)$ in direction $(\theta_1,\theta_2)$ which eventually coalesces with the (any) pair of semi-infinite geodesics in directions $\theta_1, \theta_2$.
        \item \textup{(\cite{dauvergne2024directedlandscapebrownianmotion}, Proposition 5.2.3)} Define
       $$
       \mathfrak{B}^{\theta_1,
        \theta_2}(x) := \lim_{t\to \infty} \L(x^2,0;(\theta_1,\theta_2)t,t) - \L(0,0;\theta_1t,t) - \L(0,0;\theta_2t,t).
    $$
    Almost surely, for all $(\theta_1,\theta_2)\in \mathcal{I}^2_<$ and every $x \in \R$, the limit above exists and equals
        $$
            \lim_{t\to \infty} \L(x^2,0;\tau_{x,R}^{\theta_1,\theta_2}(t),t) - \L(0,0; \pi_{x,R}^{\theta_1}(t),t) - \L(0,0; \pi_{x,R}^{\theta_2}(t),t).
        $$
        \item \textup{(\cite{dauvergne2024directedlandscapebrownianmotion}, Theorems 1.5 and 1.7.2)} Fix $\theta_1 < \theta_2 \in \R$ and $\square \in \{L, R\}$. Let $B_1, B_2$ be independent Brownian motions of drift $2\theta_1 < 2 \theta_2$ and diffusivity $\sqrt{2}$ and let $M(x) = \max_{y \le x} B_2(y) - B_1(y)$ be the running maximum of the difference. Then
    \begin{align*}
    (B^{\theta_1}_\square(x_1), B^{\theta_2}_\square(x_2), \mathfrak{B}^{\theta_1, \theta_2}(x_3)) \stackrel{d}{=}(B_1(x_1), B_1(x_2) + M(x_2) - M(0), B_1(x_3) + B_2(x_3) - M(0)).
    \end{align*}
Here the equality in law is as continuous functions of $(x_1, x_2, x_3) \in \R^3$.
    \end{enumerate}
\end{lemma}

Note that Lemma \ref{L: multi-theta-lem}.1 is stated explicitly only for Brownian last passage percolation in \cite[Proposition 4.3.2/3]{dauvergne2024directedlandscapebrownianmotion}. The proofs pass verbatim to the directed landscape, as discussed in the proof of \cite[Proposition 5.4.4]{dauvergne2024directedlandscapebrownianmotion}. Lemma \ref{L: multi-theta-lem}.3 is an extension of the more well-known description of the joint law of (single-path) Busemann functions, e.g.\ see Theorem 5.1(iii) in \cite{busani2024stationary}. The next proposition connects the two-path Busemann functions in Lemma \ref{L: multi-theta-lem} to the two-path Busemann function $\mathcal B^\theta$ introduced at the beginning of the section. Its main ingredient is Theorem \ref{T:stat-horizon}.3, together with a multi-path analogue. Combining this proposition with Lemma \ref{L: multi-theta-lem}.3 will quickly yield Theorem \ref{T: reflected-bm}.

\begin{prop} \label{P: equiv-busemann-def}
   Almost surely, for all $\theta \in \Xi$ the following holds. For all $a > 0$, there exists an $\epsilon = \ep(\theta, a) > 0$ such that for all $(\theta_1, \theta_2) \in \mathcal I^2_<$ with $\theta - \ep < \theta_1 < \theta < \theta_2 < \theta + \ep$ and all $x \in [-a, a]$ we have: 
    $$
        \mathcal{B}^\theta(x) = \mathfrak{B}^{\theta_1,\theta_2}(x).
    $$
    Moreover,
    $$
    G^\theta(x) = B_L^{\theta_1}(x) + B_R^{\theta_2}(x) - \mathfrak{B}^{\theta_1,\theta_2}(x).
    $$
\end{prop}

The proof of Proposition \ref{P: equiv-busemann-def} uses the following lemma.
\begin{lemma} \label{L: close-ang-coal}
    Almost surely, for all $x,\theta \in \R$, 
    $$
        \Ov(\pi_{x,R}^\theta, \pi_{x,R}^{\theta + \ep}) = [0,T_\ep],
    $$
    where $T_\ep \to \infty$ as $\ep \to 0^+$. A similar statement holds for leftmost semi-infinite geodesics by taking $\ep \to 0^-$.
\end{lemma}
\begin{proof}
   By continuity and non-existence of bubbles (\ref{P: no-bubbles}), we have that $\Ov(\pi_{x,R}^\theta, \pi_{x,R}^{\theta + \ep})$ is a closed interval $[0,T_\ep]$.  Note by monotonicity (Theorem \ref{T: semi-inf-geos}.3) that $\pi_{x,R}^\theta \leq \pi_{x,R}^{\theta + \ep_1} \leq \pi_{x,R}^{\theta + \ep_2}$ if $0 < \ep_1 < \ep_2$, and so $T_\ep$ is an increasing sequence. Using this monotonicity and \ref{P: precomp}, we see that the sequence $\{\pi_{x,R}^{\theta + 1/n}\}_{n\in \N}$ must converge in the topology of overlap convergence on compact sets to a limiting semi-infinite geodesic $\pi$ for $(x, 0; \theta)$. Since $\pi \ge \pi_{x, R}^\theta$ and $\pi_{x, R}^\theta$ is rightmost we must have that $\pi = \pi_{x, R}^\theta$ and so $T_\ep \uparrow \infty$.
\end{proof}

\begin{proof}[Proof Of Proposition \ref{P: equiv-busemann-def}] Let $\theta \in \Xi, a > 0$. Let $T > 0$ be a time such that:
\begin{itemize}
    \item $\pi_{0,\square}^\theta\big|_{[T,\infty)} = \pi_{x,\square}^\theta\big|_{[T,\infty)}$ for $\square \in \{L,R\}$ and all $x \in [-a, a]$ (this exists by \ref{P: semi-inf-prop}),
    \item For all $x \ge T$, $\pi_{0,L}^\theta(x) < \pi_{0,R}^\theta(x)$. 
    \item $T > t^*$, where $t^* = t^*([a, b])$ from Proposition \ref{P: opt-that-uses-disj-geos}.
    \item $T$ is large enough so that the formulas \eqref{E:eventual-coal} hold with $t = T$.
\end{itemize}
Put simply, by time $T$, the leftmost and rightmost semi-infinite geodesics from $x$ in direction $\theta$ have coalesced with those from $0$, have separated permanently, and the semi-infinite 2-optimizer $\tau_{x}^\theta$, for $(x,0;\theta)$ stays on the leftmost and rightmost semi-infinite geodesics. 

     By \ref{P: coal-semi-geo-bound}, we can find rational points $y_1 < -a < a < y_2$ such that 
     \begin{equation*}
         \pi_{y_1,L}^\theta\big|_{[0,T]} < \pi_{-a,L}^\theta \big|_{[0,T]} \qquad \text{ and } \qquad \pi_{a,R}^\theta\big|_{[0,T]} < \pi_{y_2,R}^\theta \big|_{[0,T]}.
     \end{equation*}
     Let $T' > T$ be a time by which $\pi_{y_1,L}^\theta$ has coalesced with $\pi_{a,L}^\theta$ and $\pi_{y_2,R}^\theta$ with $\pi_{-a,R}^\theta$ (this exists by \ref{P: semi-inf-prop}).
     Take $\ep > 0$ small enough so that, by Lemma \ref{L: close-ang-coal}, 
     \begin{equation}
        \nonumber \Ov(\pi_{y_1,L}^\theta, \pi_{y_1,L}^{\theta - \ep}), \Ov(\pi_{y_2,R}^\theta, \pi_{y_2,R}^{\theta + \ep}) \supset [0,T'].
     \end{equation}
     Now, consider $x \in [-a, a]$ and let $(\theta_1, \theta_2) \in \mathcal I^2_<$ be as in the statement of Proposition \ref{P: equiv-busemann-def}. We will show that the two claims in the proposition hold for this $x, \theta_1, \theta_2$ by explicitly constructing geodesics and optimizers. We start with the construction of geodesics. We claim that
     \begin{equation}
         \label{e: theta+ep-geo-form}
         \pi_{x,L}^{\theta_1} = \pi_{x,L}^\theta\big|_{[0,T']} \oplus \pi_{y_1,L}^{\theta_1}\big|_{[T',\infty)}, \qquad \qquad \pi_{x,R}^{\theta_2} = \pi_{x,R}^\theta\big|_{[0,T']} \oplus \pi_{y_2,R}^{\theta_2}\big|_{[T',\infty)}.
     \end{equation} 
     We only justify the first equality above, as the second follows symmetrically. For this, observe that for any two leftmost semi-infinite geodesics $\pi, \pi'$, the set of all $x$ where $\pi(x) = \pi'(x)$ must be a closed interval, and if $\pi, \pi'$ are semi-infinite leftmost geodesics with the same direction, then this interval is a closed ray. 
Equation \eqref{e: theta+ep-geo-form} then follows from noting that 
$$
\pi^{\theta}_{x, L}(T') = \pi^{\theta}_{y_1, L}(T') = \pi^{\theta_1}_{y_1, L}(T') = \pi^{\theta_1}_{x, R}(T').
$$
The first and second equalities here follow from the construction of $T'$ and monotonicity (\ref{P: semi-inf-prop}), and the final equality follows from the first two equalities and monotonicity (\ref{P: semi-inf-prop}).
   
    A similar argument shows the same statement for $\pi_{x,L}^{\theta-\delta}$ and the analogous statements for $y_1$ and leftmost semi-infinite geodesics in direction $\theta - \ep$. 
  
   We also have that
    \begin{equation}
     \label{E: multi-ep-opt}   \tau_{x}^{\theta_1, \theta_2} = (\tau_{x,1}^\theta\big|_{[0,T']} \oplus \pi_{y_1,L}^{\theta_1}\big |_{[T',\infty)}, \tau_{x,2}^\theta\big|_{[0,T']} \oplus \pi_{y_2,R}^{\theta_2}\big |_{[T',\infty)}).
    \end{equation}
    is a semi-infinite optimizer. This follows from Lemma \ref{L: multi-theta-lem}.1 and its proof is analogous to that of Proposition \ref{P: opt-that-uses-disj-geos}, so we omit it. 

    Putting together \eqref{e: theta+ep-geo-form}, \eqref{E: multi-ep-opt} and Lemma \ref{L: multi-theta-lem}.2 (or Definition \ref{D: busemann} for single-path Busemann paths), we have:
    \begin{align*}
   \mathfrak{B}^{\theta_1, \theta_2}(x) &= \|\tau_{x}^\theta \big|_{[0,T']}\|- \|\pi_{0,R}^\theta\big|_{[0,T']}\| - \|\pi_{0,L}^\theta\big|_{[0,T']}\| \\
   B^{\theta_1}_L(x) &= \|\pi_{x, L}^\theta \big|_{[0,T']}\|  - \|\pi_{0, L}^\theta \big|_{[0,T']}\|, \qquad \qquad B^{\theta_2}_R(x) = \|\pi_{x, R}^\theta \big|_{[0,T']}\|  - \|\pi_{0, R}^\theta \big|_{[0,T']}\|.
    \end{align*}
The proposition then follows by \eqref{E:eventual-coal} and Lemma \ref{l: buse-gap-length-formula}.
\end{proof}

\begin{proof}[Proof of Theorem \ref{T: reflected-bm}]
By Proposition \ref{P: equiv-busemann-def}, it suffices to show that for any $\theta_1 < \theta_2$, the absolute continuity statement holds for the process
$
H(x) := B_L^{\theta_1}(x) + B_R^{\theta_2}(x) - \mathfrak{B}^{\theta_1,\theta_2}(x),
$
for any $\theta_1 < \theta_2$.
By Lemma \ref{L: multi-theta-lem}.3, we have that
$$
H(x) \stackrel{d}{=} B_1(x) - B_2(x) + M(x).
$$
where $B_1, B_2, M$ are as in that lemma. Letting $W = B_2 - B_1$ be a Brownian motion of drift $\theta_2 - \theta_1$ and diffusion coefficient $2$, we can rewrite this as
$$
H(x) \stackrel{d}{=} \max_{y \le x} W(y) - W(x).
$$
On any compact set $[a, b]$,
the above process equal in law to a two-sided Brownian motion of downward drift $\theta_1 -\theta_2$, started from the initial condition $H(a)$ and independent of $H(a)$, e.g.\ see \cite[Theorem 2.31]{morters2010brownian}. Removing the drift does not affect absolute continuity, yielding the result.
\end{proof}
\appendix
\section{Appendix} \label{Appen: no weak opt}

In this appendix, we show that any weak optimizer in the directed landscape is strictly disjoint. For a $k$-tuple of ordered paths $\gamma: [s,t] \to \R^k_{\leq}$, recall that the length of $\gamma$ in the extended landscape is 
\begin{equation}
\label{E:weak-ext-land-length}
  \notag  \|\gamma\|_\L = \inf_{k \in \N} \inf_{r_0 = s < r_1 < \cdots < r_k = t} \sum_{i=1}^k \L\left(\gamma(r_{i-1}),r_{i-1};\gamma(r_i),r_i\right).
\end{equation}
We say that $\gamma$ is a \textbf{weak optimizer} if
\begin{equation}
\label{E:optimizer}
\|\gamma\|_\L =\L\left(\gamma(s),s;\gamma(t),t\right).
\end{equation}
We say that $\gamma$ is a \textbf{disjoint optimizer} if \eqref{E:optimizer} holds, and 
$$
\gamma_1(r) < \gamma_2(r) < \cdots < \gamma_k(r)
$$
for all $r \in (s, t)$. In this appendix, we show that any weak optimizer is a disjoint optimizer. 

\begin{theorem}
    \label{T:weak-is-disjoint}
    Almost surely, the following holds in the directed landscape. Let $\gamma$ be any weak optimizer from $(\x, s)$ to $(\y, t)$. Then for any $r_1 < r_2 \in (s, t)$, the $\gamma|_{[r_1, r_2]}$ is the unique weak optimizer from $(\gamma(r_1), r_1)$ to $(\gamma(r_2), r_2)$. That is, there are no optimizer bubbles in the directed landscape.
    
As a consequence, any weak optimizer in the directed landscape is a disjoint optimizer. 
\end{theorem}

The proof of Theorem \ref{T:weak-is-disjoint} is closely related to the proofs in  \cite{bhatia2024dualitydirectedlandscapeapplications, dauvergne202327geodesicnetworksdirected} which show that there are no geodesic bubbles in the directed landscape. The main difficulty in showing that there are no geodesic bubbles is showing that any interior segment along a geodesic can be approximated in overlap with rational geodesics. We will similarly show that any interior segment of a weak optimizer can be approximated in overlap by disjoint optimizers. The papers \cite{bhatia2024dualitydirectedlandscapeapplications, dauvergne202327geodesicnetworksdirected} resolve the rational approximation problem in different ways. In our setting, it is easier to adapt the method of \cite{bhatia2024dualitydirectedlandscapeapplications}, which uses the fact that $2$-star points along the interior of geodesics are rare as the key input.

\begin{lemma}[Lemma 3.3.1 \cite{dauvergne202327geodesicnetworksdirected}] \label{L: rat-overlap} Almost surely, for any geodesic $\pi$ in the directed landscape, there exists a sequence of geodesics $(\pi_n)_{n\in \N}$, where $\pi_n$ is the unique geodesic between its rational endpoints, such that 
$$
    \pi_n \overset{overlap}{\to} \pi.
$$
\end{lemma}

\begin{prop}[Theorem 4, \cite{bhatia2022atypical}]
\label{P:2-stars}
Let $\pi$ be the (almost surely unique) geodesic from $(0,0)$ to $(0, 1)$ in the directed landscape. Let
$$
T(\pi) = \{t \in (0, 1) : (t, \pi(t)) \text{ is a forwards or backwards $2$-star}\}.
$$
Then the Hausdorff dimension of $S$ is almost surely $1/3$.
\end{prop}

\begin{corollary}
\label{C:rational-stars}
Fix $t \in \R$. Then almost surely, there are no $2$-star points $(x, t) \in \R^2$ that lie on the interior of a geodesic.
\end{corollary}

\begin{proof}
First, for a geodesic $\pi:[s, t] \to \R$, define the set $T(\pi)$ as in Proposition \ref{P:2-stars}, with $[0, 1]$ replaced by $[s, t]$.
To prove the corollary, it suffices to show that 
$$
t \notin T := \bigcup_\pi T(\pi) 
$$
almost surely, where the union is over all geodesics $\pi$. Next, since any geodesic can be approximated in overlap by rational geodesics (Lemma \ref{L: rat-overlap}), we have
$$
T = \bigcup_{\pi:(x, s) \to (y, t), x, s, y, t \in \Q} T(\pi),
$$
where the right-hand side is a countable union over all rational geodesics. Now, by Proposition \ref{P:2-stars} and the symmetries of the directed landscape (Lemma \ref{L: EDL-symmetries}), for any rational geodesic $\pi$, the set $T(\pi)$ has Hausdorff dimension $1/3$. Therefore $T$ also has Hausdorff dimension $1/3$, and hence it has Lebesgue measure $0$. Finally, $T$ is translation invariant in law by time-shift invariance of the directed landscape (Lemma \ref{L: EDL-symmetries}), and so for any fixed point $t$, $t \notin T$ almost surely.
\end{proof}

To use Corollary \ref{C:rational-stars} to establish Theorem \ref{T:weak-is-disjoint}, we will also need to use that at typical times along any fixed optimizer, that optimizer is disjoint and locally consists of geodesics. The next proposition gives the disjointness.

\begin{prop}[Lemma 8.2, \cite{dauvergne2022disjointoptimizersdirectedlandscape}]
    \label{P:one-time-dj}
    For any fixed $s$, the following holds almost surely for $\L$. For any weak optimizer $\gamma$ defined on an interval $[r, t]$ with $r < s < t$, we have that
    $$
    \gamma_1(s) < \gamma_2(s) < \cdots < \gamma_k(s).
    $$
\end{prop}

To upgrade Proposition \ref{P:one-time-dj} to find local geodesic segments of $\gamma$, we also need the following simple lemma.

\begin{lemma}
\label{L:opt-lemma}
Let $\x, \y \in \R^k_<$ and let $s < t$. Suppose that there exist disjoint geodesics $\pi_i, i = 1, \dots, k$ from $(x_i, s)$ to $(y_i, t)$. Then, if $\gamma$ is any weak optimizer from $(\x, s)$ to $(\y, t)$, the paths $\gamma_1, \dots, \gamma_k$ are all geodesics.
\end{lemma}

\begin{proof}
With notation as in the statement of the lemma, we have
$$
\sum_{i=1}^k \L(x_i, s; y_i, t) = \L(\x, s; \y, t) = \|\gamma\|_\L \le \sum_{i=1}^k \|\gamma_i\|_\L.
$$
Here the first equality is from the assumption that there exist disjoint geodesics from $(x_i, s)$ to $(y_i, t)$ and the definition of the extended landscape, the second equality is because $\gamma$ is a weak optimizer, and the inequality is by the definition \eqref{E:optimizer}, and the fact that, by \eqref{eq: ext-landscape}, for any $\x', \y' \in \R^k_\le$ and $s' < t'$, we have that 
$$
\L(\x', s'; \y', t') \le \sum_{i=1}^k \L(x_i', s'; y_i', t').
$$
On the other hand, $\L(x_i, s; y_i, t) \ge \|\gamma_i\|_\L$ for all $i$, and so $\L(x_i, s; y_i, t) = \|\gamma_i\|_\L$ for all $i$. That is, all the paths $\gamma_i$ are geodesics.
\end{proof}
\begin{prop}
\label{P:locally-geodesic}
The following holds almost surely for $\L$. Consider any weak optimizer $\gamma$ defined on an interval $[r, t]$, and any time $s \in (r, t)$ with 
\begin{equation}
\label{E:gammais}
  \notag  \gamma_i(s) < \gamma_{i+1}(s) \text{ for all } i = 1, \dots, k.
\end{equation}
Then there exists $\ep > 0$ depending on $\gamma, s$ such that the paths
$$
\gamma_1|_{[s-\ep, s+ \ep]}, \dots,  \gamma_k|_{[s-\ep, s+ \ep]}
$$
are disjoint geodesics. The same conclusion holds when $r = s$ or $r = t$, except we must replace $[s-\ep, s+ \ep]$ by $[r, r + \ep]$ or $[t-\ep, t]$, respectively.
\end{prop}

\begin{proof}
We only consider the case when $s \in (r, t)$, as the cases when $s = r$ or $s = t$ are the same up to notational changes.

Let $K \subset \R^2$ be a compact set containing all of the points $(\gamma_i(s'), s'), s' \in [r, t]$. The landscape shape theorem (Lemma \ref{L: DL-bound}) implies that there exists a $K$-dependent random constant $C > 0$ such that for any geodesic $\pi$ between any two points $(x, s'), (y, t')$ in $K$, almost surely the geodesic $\pi$ satisfies
\begin{equation}
\label{E:weak-geo-bd}
\max_{r \in [s', t']} |\pi(r) - x| \le |y - x| + C|t' - s'|^{1/3}.
\end{equation}
On the other hand, by continuity of $\gamma$, for all small enough $\ep > 0$ we have that
$$
\min_{s' \in [s-\ep, s+ \ep], i \in \{1, \dots, k-1\}} \gamma_{i+1}(s') - \gamma_i(s') > 10(C\ep^{1/3} + \max_{i \in \{1, \dots, k\}} |\gamma_i(s - \ep) - \gamma_i(s + \ep)|).
$$
Therefore by \eqref{E:weak-geo-bd}, for such $\ep$, any $k$-tuple of geodesics between $(\gamma(s - \ep), s-\ep)$ and $(\gamma(s+\ep), s + \ep)$ is disjoint. Appealing to Lemma \ref{L:opt-lemma} then completes the proof.
\end{proof}

\begin{proof}[Proof of Theorem \ref{T:weak-is-disjoint}]
Throughout the proof, we work on the almost sure set where Corollary \ref{C:rational-stars} holds for all rational $t$, Proposition \ref{P:one-time-dj} holds for all rational $s$, and Proposition \ref{P:locally-geodesic} holds.

Consider any weak optimizer $\gamma:[s, t] \to \R^k_\le$. Fix rational times $r_1 < r_2 \in (s, t)$. Then by Propositions \ref{P:one-time-dj} and \ref{P:locally-geodesic}, we can find $\ep > 0$ such that
\begin{equation}
\label{E:dj-pieces}
    \gamma_i|_{[r_1 - \ep, r_1 +
\ep]}, \quad i = 1, \dots, k, \qquad \gamma_i|_{[r_2-\ep, r_2 + \ep]}, \quad i = 1, \dots, k, \qquad \text{ are disjoint geodesics.}
\end{equation}

Now, consider decreasing sequences $\mathbf{p}_n, \mathbf{q}_n \in \Q^k_<$ with $\mathbf{p}_n \to \gamma(r_1), \mathbf{q}_n \to \gamma(r_2)$, and let $\sigma^{(n)}$ be the unique weak optimizers from $(\mathbf{p}_n, r_1)$ to $(\mathbf{q}_n, r_1)$. These optimizers are unique by Lemma \ref{L:rlm-opt-unique-no-bubble}.1 and therefore disjoint by Theorem \ref{T:dj-exist}. By Lemma \ref{L: geo-opt-rm-coal}, the optimizers $\sigma^{(n)}:[r_1, r_2] \to \R^k_\le$ converge in overlap to the rightmost weak optimizer from $(\gamma(r_1), r_1)$ to $(\gamma(r_2), r_2)$. Call this weak optimizer $\sigma$. This is a \textit{disjoint} optimizer since disjointness is a closed property in the overlap topology. By applying Proposition \ref{P:locally-geodesic} again we can guarantee that for some $\ep' > 0$, 
\begin{equation}
\label{E:dj-pieces-2}
    \sigma_i|_{[r_1, r_1 +
\ep']}, \quad i = 1, \dots, k, \qquad \sigma_i|_{[r_2-\ep', r_2]}, \quad i = 1, \dots, k, \qquad \text{ are disjoint geodesics.}
\end{equation}
Since $r_1, r_2$ are rational, and the points $(\gamma_i(r_1), r_1), (\gamma_i(r_2), r_2)$ all lie on the interior of geodesics by \eqref{E:dj-pieces}, by Corollary \ref{C:rational-stars} none of these points are geodesic $2$-stars (forwards or backwards). Therefore combining \eqref{E:dj-pieces}, \eqref{E:dj-pieces-2}, there is some $\delta > 0$ such that 
\begin{equation}
\label{E:pi-gamma}
\sigma|_{[r_1, r_1 + \de] \cup [r_2 - \de, r_2]} = \gamma|_{[r_1, r_1 + \de] \cup [r_2 - \de, r_2]}.
\end{equation}
Now fix $\alpha \in (0, \delta]$ and consider $n$ large enough so that $\sigma^{(n)}|_{[r_1 + \alpha, r_2 - \alpha]} = \sigma|_{[r_1 + \alpha, r_2 - \alpha]}.$ By \eqref{E:pi-gamma}, both the $k$-tuple $\sigma^{(n)}$ and the concatenation
$$
\sigma^{(n)}|_{[r_1, r_1 + \alpha]} \oplus\gamma|_{[r_1 + \alpha, r_2 - \alpha]} \oplus \sigma^{(n)}|_{[r_2-\alpha, r_2]} 
$$
are weak optimizers from $(\mathbf{p}_n, r_1)$ to $(\mathbf{q}_n, r_2)$. By uniqueness of $\sigma^{(n)}$, this implies that $\sigma|_{[r_1 + \alpha, r_2 - \alpha]} = \gamma|_{[r_1 + \alpha, r_2 - \alpha]}$. Taking $\alpha \to 0$ gives that $\sigma$ is the unique optimizer from $(\gamma(r_1), r_1)$ to $(\gamma(r_2), r_2)$, and that $\gamma$ is is a disjoint optimizer when restricted to $[r_1, r_2]$. Since $r_1, r_2$ are arbitrary, $\gamma$ is a disjoint optimizer. 
\end{proof}

\newpage

\bibliographystyle{alpha}
\bibliography{sources}

\end{document}